\documentclass[12pt]{article}
\usepackage{mlmodern}
\usepackage{pdfrender,xcolor}
\usepackage{graphicx,color}
\usepackage{amsmath,amsfonts,amssymb,amsthm,MnSymbol}
\usepackage[hidelinks,colorlinks=true,linkcolor=blue,citecolor=blue]{hyperref}
\usepackage{enumerate}
\usepackage{algorithm}
\usepackage{algpseudocode}
\usepackage{rsfso}
\usepackage{comment}

\usepackage{natbib}
\definecolor{darkgreen}{rgb}{0.06, 0.56, 0.2}

\newcommand{\E}{\mathbb{E}}
\newcommand{\R}{\mathbb{R}}
\newcommand{\Prob}{\mathbb{P}}
\newcommand{\Z}{\mathbb{Z}}

\providecommand{\abs}[1]{\lvert#1\rvert}

\providecommand{\babs}[1]{\,\big\lvert#1\big\rvert}

\def\cals{{\mathcal S}}

\DeclareMathOperator{\diag}{diag}

\DeclareMathOperator{\Var}{Var}

\newcommand{\sr}[1]{{\cal #1}}
\newcommand{\dd}[1]{\mathbb{#1}}

\newcommand{\br}[1]{\langle #1 \rangle}

\newcommand{\ol}{\overline}

\newcommand{\eq}[1]{(\ref{eq:#1})}
\newcommand{\lem}[1]{Lemma~\ref{lem:#1}}
\newcommand{\cor}[1]{Corollary~\ref{cor:#1}}

\newcommand{\thm}[1]{Theorem~\ref{thm:#1}}

\newcommand{\pro}[1]{Proposition~\ref{pro:#1}}

\newcommand{\dfn}[1]{Definition~\ref{def:#1}}

\newcommand{\assum}[1]{Assumption~\ref{assum:#1}}

\newcommand{\app}[1]{Appendix~\ref{app:#1}}
\newcommand{\sectn}[1]{Section~\ref{sec:#1}}

\newcommand{\lemt}[1]{\ref{lem:#1}}

\newcommand{\assumt}[1]{\ref{assum:#1}}
\newcommand{\appt}[1]{\ref{app:#1}}

\newcommand{\sect}[1]{\ref{sec:#1}}

\newcommand{\pend}{\hfill \thicklines \framebox(6.6,6.6)[l]{}}

\newenvironment{proof*}[1]{\noindent {\sc  #1} \rm}{\pend}

\newtheorem{theorem}{Theorem}[section]
\newtheorem{lemma}{Lemma}[section]
\newtheorem{assumption}{Assumption}[section]
\newtheorem{proposition}{Proposition}[section]

\newtheorem{corollary}{Corollary}[section]
\newtheorem{definition}{Definition}[section]

\newcommand{\setsection}[2] {
\setcounter{section}{#1}
\setcounter{subsection}{0}
\setcounter{equation}{0}
\setcounter{conjecture}{0}
\setcounter{assumption}{0}
\setcounter{question}{0}
\setcounter{definition}{0}
\setcounter{theorem}{0}
\setcounter{corollary}{0}
\setcounter{lemma}{0}
\setcounter{proposition}{0}
\setcounter{remark}{0}
\setcounter{appen}{0}
\setsection*{\large \bf \thesection. #2}}

\newenvironment{mylist}[1]{\begin{list}{}
{\setlength{\itemindent}{#1mm}}
{\setlength{\itemsep}{0ex plus 0.2ex}}
{\setlength{\parsep}{0.5ex plus 0.2ex}}
{\setlength{\labelwidth}{10mm}}
}{\end{list}}

\newcommand{\setnewcounter} {
\setcounter{subsection}{0}
\setcounter{equation}{0}
\setcounter{conjecture}{0}
\setcounter{assumption}{0}
\setcounter{question}{0}
\setcounter{definition}{0}
\setcounter{theorem}{0}
\setcounter{corollary}{0}
\setcounter{lemma}{0}
\setcounter{proposition}{0}
\setcounter{remark}{0}
}

\begin{document}
\title{\bf \Large  The BAR approach for multiclass queueing networks with SBP service policies\footnote{We thank Xinyun Chen and Jin Guang for helpful discussions on Palm measure exposition.}}

\author{Anton Braverman\\Kellogg School of Management, Northwestern University \and J.G. Dai\\School of Operations Research and Information Engineering, Cornell University\\ \and Masakiyo Miyazawa\\School of Data Science, The Chinese University of Hong Kong, Shenzhen\\}
\normalsize 


\maketitle

\begin{abstract}
  The basic adjoint relationship (BAR) approach is an analysis technique based on the stationary equation of a Markov process.   This approach was introduced to study heavy-traffic, steady-state
  convergence of generalized Jackson networks in which each service station has a
  single job class. We extend it to multiclass queueing networks operating under
  static-buffer-priority (SBP) service disciplines.    Our extension makes a connection with Palm
  distributions that allows one to attack a difficulty arising from
  queue-length truncation, which appears to be unavoidable  in the multiclass
  setting.

  For multiclass queueing networks operating under
  SBP service disciplines, our BAR approach 
  provides an alternative to the ``interchange of limits'' approach
  that has dominated the literature in the last twenty years.
  The BAR approach can produce sharp results and
  allows  one to establish  steady-state convergence under three additional
  conditions: stability, state space collapse (SSC) and a certain matrix
  being ``tight.''  These three conditions do not appear to depend on
  the interarrival and service-time distributions beyond their means,
  and their verification can be studied as three separate modules. In
  particular, they can be studied in a simpler, continuous-time Markov
  chain setting when all distributions are exponential.

  As an example,
  these three conditions are shown to hold in reentrant lines operating
  under last-buffer-first-serve discipline.  In a two-station,
  five-class reentrant line, under the heavy-traffic condition, the
  tight-matrix condition implies both the stability condition and the SSC
  condition.  Whether such a relationship holds   generally is an open problem.
\end{abstract}

 \begin{quotation}

\medskip

\end{quotation}

\section{Introduction}
\label{sec:intro}

In this paper, we prove that the stationary distribution of a
multiclass queueing network  converges to the stationary
distribution of  a semimartingale reflecting Brownian motion  (SRBM)  in
heavy traffic, or as the load at each service station becomes
``critical,'' where it is assumed that the network operates  
under a static-buffer-priority (SBP) service discipline  (see \sectn{multiclass} for its definition). For this proof, we extend the BAR approach developed in \cite{Miya2017}
and \cite{BravDaiMiya2017}, where BAR stands for basic adjoint relationship, which was coined in \cite{HarrWill1987} in the setting of charactering the stationary distribution of an SRBM. The main result of this paper is \thm{main}, which assumes three additional conditions: stability, state space collapse, and a certain matrix being ``tight'' (see Definition~\ref{def:tight}). As of now, it is difficult to characterize when each of these conditions holds in a general setting, but it is known that there are various examples, like reentrant lines under the last-buffer-first-server (LBFS) service discipline, that satisfy them. For a more gradual introduction to the machinery behind \thm{main},  we also work through a pilot example of a two-station, five-class reentrant line in \sectn{2s5c}. In what follows, we first introduce the background for our work, then explain the features of the BAR approach, and finally summarize the contributions of this paper.

The subject of this study is Brownian models for multiclass queueing networks.
These Brownian models were introduced in \cite{Harr1988}.
Multiclass queueing networks
were studied in classical papers such as \cite{BaskChanMuntPala1975} and
\cite{Kell1975}. In these classical papers, the queueing networks are
modeled as continuous-time Markov chains (CTMC) with \emph{discrete}
state spaces.  These CTMCs are shown to have ``product-form''
stationary distributions.  Fueled by applications in computer systems
and communications networks, product-form research was a
dominant theme for more than two decades.  See, for example,
\cite{Serf1999} for a summary of this line of research at the end of
1990s.  Harrison's multiclass queueing networks have general
inter-arrival and service-time distributions and accommodate arbitrary
service disciplines. These queueing networks can be modeled as
piecewise deterministic Markov processes that were formally
introduced in \cite{Davi1984}.  These continuous-time Markov processes
have components with continuous state spaces.  Obtaining the stationary
distributions of these Markov processes, whether analytically or numerically, is often difficult.
This difficulty motivates   the study of Brownian models, which are often represented by  SRBMs.

In addition to introducing multiclass queueing networks that model
real-world systems, \cite{Harr1988} introduced Brownian system models
that serve as alternative models of the same real-world systems. Since
the publication of   \cite{Harr1988}, many papers    proving that a certain
``state'' process of a multiclass queueing network converges in
distribution to the corresponding process of the Brownian system model
in heavy traffic have appeared; see, for example, \cite{Bram1998c}, \cite{Will1998a},
\cite{ChenZhan2000}, \cite{ChenZhan2000b}, \cite{ChenYe2001}. These extend the pioneering works of \cite{Reim1984} and \cite{John1983}
that prove  a heavy-traffic limit theorem for generalized Jackson
networks, a special class of queueing networks in which each service
station has a single job class. These limit theorems are of the type
of functional central limit theorems that approximate the dynamics of a queueing network
 by the dynamics of its Brownian counterpart, but they are silent on steady-state convergence:  whether
the stationary distribution of a multiclass queueing network converges
to that of a Brownian model.

\cite{Gurv2014a} proved steady-state convergence for multiclass
queueing networks operating under a class of queue-ratio service
disciplines that include SBP disciplines as special cases. This   work
was inspired by the pioneering paper of \cite{GamaZeev2006} that
proved steady-state convergence for generalized Jackson networks.
\cite{YeYao2016, YeYao2018} went further by (a) relaxing the
conditions in \cite{Gurv2014a} and, more importantly (b) covering a
wider class of service disciplines. \cite{YeYao2018} represents the
state of the art in  results for steady-state convergence of multiclass
queueing networks. All these works provedthe ``interchange of limits'' by
employing and  extending the  sophisticated ``hydro-dynamic limits''
methodology introduced in \cite{Bram1998c} for process convergence,
establishing rigorously  that   process convergence in functional central
limit theorems is robust enough to carry over to   steady-state
convergence. Since \cite{GamaZeev2006}, interchange of limits has
been proved for many other stochastic models; see the discussion on
page 147 of \cite{BravDaiMiya2017}, including the relevance of using Stein's method to study steady-state convergence.

This paper proves steady-state convergence for
  multiclass queueing networks directly, without working with the dynamics of
either the pre-limit or limit process. The logic for this
possibility is simple: the generator of a Markov
process, when well defined, governs both the dynamics and the
stationary distribution  of the Markov process. By working with the
generator, one does not need to use the dynamics of a Markov process
to understand its steady state. However, for a piecewise
deterministic Markov process, the test functions in the domain of
  the generator need to satisfy a so-called boundary condition. For
the generalized Jackson networks studied in \cite{BravDaiMiya2017},
the test functions of interest are in the domain of the generator and
the corresponding BAR does not have any boundary terms. Taking
advantage of this fact, the authors were able develop the BAR approach
to reproduce the \cite{GamaZeev2006} result under a weaker condition.
In the multiclass queueing networks considered in this
paper, one needs to truncate the queue length terms in the test functions.
As a result, they are no longer in the domain of the generator. The 
BAR in multiclass queueing networks involves boundary terms through
Palm distributions, which are generated by counting processes of the
jumps (see \sectn{BAR-general}). A key step in our proof of \thm{main}
is to show,  using Palm measures,  that those boundary terms are
negligible in heavy traffic, and the asymptotic BAR similar to the one
in \cite{BravDaiMiya2017} still holds.

  The BAR approach promotes
modularity. It separates the stability and steady-state state space
collapse (SSC) results from  steady-state convergence. The stability
of multiclass queueing networks has been extensively studied in the 
literature. See, for example, \cite{Dai1995} and \cite{ChenZhan1997a}.
Sufficient conditions for steady-state state space collapse in
multiclass queueing networks were established in
\cite{CaoDaiZhan2022}, and the conditions  were verified to hold in that
paper for reentrant lines under the first-buffer-first-serve
discipline (FBFS) and LBFS service discipline. Steady-state SSC was
proved for a bandwidth-sharing network in \cite{WangMaguSrikYing2022}.

 The  multifold contributions of this paper are summarized below. 

\begin{enumerate}
\item[$\bullet$]  The BAR approach has been demonstrated to be a natural approach to
  proving heavy-traffic, steady-state convergence, as opposed to the
  limit-interchange approach widely used in the literature. 

\begin{enumerate}
\item 
  It makes the heavy-traffic, steady-state analysis
  essentially not sensitive to the distributions of inter-arrival and
  service times, thus allowing a researcher to start the analysis in a
  CTMC setting. 
\item It can produce the sharpest results
  with minimal moment conditions; our approach assumes the
  existence of the $(2+\delta_0)$-th moments of inter-arrival and service times, where
  \cite{YeYao2018} requires the seventh moment.
\item  It was successfully used in
  \cite{DaiGlynXu2023} to establish asymptotic steady-state
  independence for generalized Jackson networks in multi-scale heavy
  traffic. It is unclear how the ``limit-interchange'' approach in
 \cite{Gurv2014a} and \cite{YeYao2018} can be extended to the multi-scale setting.
\end{enumerate}

\item [$\bullet$] The BAR approach developed in this paper goes significantly
  beyond the restrictive version in \cite{BravDaiMiya2017}.
\begin{enumerate}
\item [(d)] It takes
  care of both the queue-length truncation and inter-arrival and service-time truncation that will likely be encountered in many other
  stochastic processing networks.
\item [(e)] It connects with   Palm
  distributions in a way that was not explored in \cite{BravDaiMiya2017}; see
  \lem{R} and \lem{Zes} in this paper.  \cite{GuanChenDai2023} has already made 
  critical usage of this Palm connection.
\end{enumerate}
 \end{enumerate}

In the discrete-time setting, the BAR approach has been studied
extensively in the literature.  For example, \cite{EryiSrik2012} and
\cite{MaguSrik2016} used carefully engineered polynomial functions as
test functions to get asymptotically tight bounds on the steady-state
moments.  Characterizing all moments allowed \cite{EryiSrik2012} to
also establish steady-state convergence to a limiting distribution.
They coined the term ``drift method'' for  their approach.  By using a family
of exponential test functions (closely related to the ones used in our
paper), \cite{HurtMagu2020} proved steady-state convergence for a
``generalized switch'' that was first studied in \cite{Stol2004}.  The
authors called their approach the ``transform method,'' which is
essentially our BAR approach in the discrete-time setting. Discrete
time offers simplifications not available in our continuous-time
setting. We emphasize that our approach is complicated not only by   continuous time, but also by the presence of general inter-arrival and
service-time distributions. Indeed, \cite{WangMaguSrikYing2022} is one example
of the drift method being applied to the famous bandwidth-sharing model; by
assuming phase-type jobsize distributions, the authors were able to
study the model in the CTMC setting and therefore 
did not need to deal with the added complexity of general
job-size distributions.

This paper is composed of eight sections. In \sectn{2s5c}, we exemplify the BAR approach for a two-station, five-class reentrant line with SBP service discipline. To simplify the analysis, it is assumed that all the inter-arrival and service times are either exponentially distributed or generally distributed but bounded. \pro{2s} is a main result of this section, which is a spacial case of \thm{main}. Although this network is simple, it illustrates the main ideas of the BAR approach. In \sectn{multiclass},
multiclass queueing networks  are introduced, while SRBM and its BAR are discussed in \sectn{srbm}. Then, the main result, \thm{main}, and its corollary are presented in \sectn{main}. The preliminary results for proving \thm{main} are given in \sectn{bar-Xr}. The proof of \thm{main} is divided into six steps. Steps 2 through 6 are proved  in \sectn{proof-main} while step 1 is proved in \sectn{deriving-BAR}, where SSC under the Palm distributions is obtained. This SSC is a key result in this step, and may be  interesting itself. This is the reason why the first step is separately proved in \sectn{deriving-BAR}. Some auxiliary results are given in Appendices \appt{tight1}, \appt{moment-weaker} and \appt{P-expansion}.

\section{A two-station, five-class queueing network}
\label{sec:2s5c}

In this section,   we first introduce a pilot example of a two-station, five-class queueing network operating under a static buffer priority (SBP) service discipline.  We then state  the main result of this paper in the
setting of this two-station network. Finally, we  prove the
result in two steps: (i) when the inter-arrival and
service-time distributions are exponential and (ii) when the inter-arrival and service-time distributions are general with bounded supports.  By
focusing first on the two-station setting, we  avoid an elaborate notational system that is required for a
general queueing network but are able to highlight the key technical
contributions of this  paper.


\subsection{Network description}
\label{sec:2s5c_d}

Figure~\ref{fig:2s5c} depicts a two-station, five-class queueing
network. 
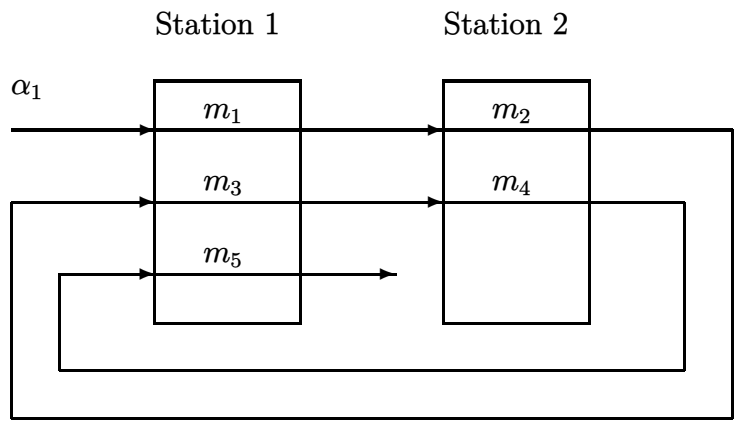
\begin{figure}[htb]
  \begin{center}
   \setlength{\unitlength}{0.00083300in}%
\begingroup\makeatletter\ifx\SetFigFont\undefined
\def\x#1#2#3#4#5#6#7\relax{\def\x{#1#2#3#4#5#6}}%
\expandafter\x\fmtname xxxxxx\relax \def\y{splain}%
\ifx\x\y   
\gdef\SetFigFont#1#2#3{%
  \ifnum #1<17\tiny\else \ifnum #1<20\small\else
  \ifnum #1<24\normalsize\else \ifnum #1<29\large\else
  \ifnum #1<34\Large\else \ifnum #1<41\LARGE\else
     \huge\fi\fi\fi\fi\fi\fi
  \csname #3\endcsname}%
\else
\gdef\SetFigFont#1#2#3{\begingroup
  \count@#1\relax \ifnum 25<\count@\count@25\fi
  \def\x{\endgroup\@setsize\SetFigFont{#2pt}}%
  \expandafter\x
    \csname \romannumeral\the\count@ pt\expandafter\endcsname
    \csname @\romannumeral\the\count@ pt\endcsname
  \csname #3\endcsname}%
\fi
\fi\endgroup
\begin{picture}(4524,2520)(2389,-4273)
\thicklines
\put(3301,-3661){\framebox(900,1500){}}
\put(5101,-3661){\framebox(900,1500){}}
\put(2401,-2461){\vector( 1, 0){900}}
\put(4201,-2461){\vector( 1, 0){900}}
\put(4201,-2911){\vector( 1, 0){900}}
\put(4201,-3361){\vector( 1, 0){600}}
\multiput(3301,-2461)(8.00000,0.00000){113}{\line( 1, 0){  4.000}}
\multiput(5101,-2461)(8.00000,0.00000){113}{\line( 1, 0){  4.000}}
\multiput(3301,-2911)(8.00000,0.00000){113}{\line( 1, 0){  4.000}}
\multiput(5101,-2911)(8.00000,0.00000){113}{\line( 1, 0){  4.000}}
\multiput(3301,-3361)(8.00000,0.00000){113}{\line( 1, 0){  4.000}}
\put(6001,-2461){\line( 1, 0){900}}
\put(6901,-2461){\line( 0,-1){1800}}
\put(6901,-4261){\line(-1, 0){4500}}
\put(2401,-4261){\line( 0, 1){1350}}
\put(2401,-2911){\vector( 1, 0){900}}
\put(6001,-2911){\line( 1, 0){600}}
\put(6601,-2911){\line( 0,-1){1050}}
\put(6601,-3961){\line(-1, 0){3900}}
\put(2701,-3961){\line( 0, 1){600}}
\put(2701,-3361){\vector( 1, 0){600}}
\put(3301,-1861){\makebox(0,0)[lb]{\smash{\SetFigFont{12}{14.4}{rm}Station 1}}}
\put(5101,-1861){\makebox(0,0)[lb]{\smash{\SetFigFont{12}{14.4}{rm}Station 2}}}
\put(2401,-2236){\makebox(0,0)[lb]{\smash{\SetFigFont{12}{14.4}{rm}$\alpha_1$}}}
\put(3601,-2386){\makebox(0,0)[lb]{\smash{\SetFigFont{12}{14.4}{rm}$m_1$}}}
\put(5401,-2386){\makebox(0,0)[lb]{\smash{\SetFigFont{12}{14.4}{rm}$m_2$}}}
\put(3601,-2836){\makebox(0,0)[lb]{\smash{\SetFigFont{12}{14.4}{rm}$m_3$}}}
\put(5401,-2836){\makebox(0,0)[lb]{\smash{\SetFigFont{12}{14.4}{rm}$m_4$}}}
\put(3601,-3286){\makebox(0,0)[lb]{\smash{\SetFigFont{12}{14.4}{rm}$m_5$}}}
\end{picture}
  \end{center}
  \caption{A two-station, five-class reentrant line}
      \label{fig:2s5c}
\end{figure}
Each rectangle represents a single-server station that processes jobs
one at a time. Jobs arrive to the network exogeneously following a
renewal process.  Each job has five processing steps in the network that follow  the flow indicated by the arrows in the figure; server  1 performs steps 1, 3,
and 5 at station $1$, while server 2 performs steps 2 and 4 at station $2$.  When a
job completes its processing at step $k$ and the server at step $k+1$
is busy, the job moves to buffer $k+1$ and waits for its turn to be
processed at step $k+1$. After finishing step 5 processing, jobs exit the network. 

 Each buffer is assumed to have infinite
capacity. Following \cite{Harr1988}, we adopt the notion of job
classes. A job belongs to class $k$ if it is either processing in step $k$
 or waiting in buffer $k$.  We use the terms ``class'' and ``buffer'' interchangeably, with the understanding that a job in step
$k$ processing still belongs to buffer $k$. Let  $m_kT_{s,k}(i)$ be the processing time of the $i$th class $k$ job, and let $\{m_kT_{s,k}(i), i\ge 1\}$ be the corresponding sequence of processing times. We assume that the elements of this sequence are independent and identically distributed
(i.i.d.)  with mean $m_k$ and $\E[T_{s,k}(i)]=1$. The inter-arrival times $\{(1/\lambda_1) T_{e,1}(i), i\ge 1\}$ of the exogenous
arrival process are assumed to be i.i.d.\ with mean $1/\lambda_1$ and $\E[T_{e,1}(i)]=1$. We assume these sequences are defined on a probability space
$(\Omega, \mathcal{F}, \dd{P})$. We also
assume that different i.i.d.\ sequences are  independent.
For a positive random variable $U$, its squared coefficient of
variation (SCV), denoted as $c^2(U)$, is defined to be
\begin{align*}
  c^2(U) = \frac{\text{Var}(U)}{(\E[U])^2}.
\end{align*}

When server $1$ completes the processing of a class $k\in \{1,3, 5\}$
job, it needs a service discipline to decide which buffer the next job
should be picked from.   For our pilot example,
we specify service  discipline by the following list
\begin{equation}\label{eq3.19}
  \{(5,3,1), (2,4)\}.
\end{equation}
This list means that, at station 1, this discipline gives the highest priority to class 5, the next priority to class 3,
and the lowest
priority to class 1, whereas at station 2, the highest priority goes to class 2 and the lowest priority to class 4.
 We further assume that the service discipline
is \emph{preemptive-resume}:  when a job with a higher rank than the
one currently being served arrives at the server's station, the service of
the current job is interrupted.  When all jobs of higher rank are served,  the interrupted service continues from where it left off. This service discipline is referred to as static-buffer-priority (SBP for short), which is defined for a general multiclass queueing network in \sectn{multiclass}.

\subsection{Markov process and its stability}
\label{sec:2s-stability}
 For $t \geq 0$ define 
\begin{align}
X(t)  = 
\begin{pmatrix}
Z(t) \\ U_1(t) \\ V(t)
\end{pmatrix}, \quad
Z(t) = 
\begin{pmatrix}
Z_1(t) \\ Z_2(t) \\ Z_3(t) \\ Z_4(t) \\ Z_5(t)
\end{pmatrix}, \quad
V(t) = 
\begin{pmatrix}
V_1(t) \\ V_2(t) \\ V_3(t) \\ V_4(t) \\ V_5(t)
\end{pmatrix}, 
\label{eq:xtdef}
\end{align}
where $Z_k(t)$ is the number of class $k$ jobs at time $t$, including possibly the one in service,  $U_1(t)$  is the remaining inter-arrival time of the next class $1$ job, and $V_k(t)$  is the remaining service time of the leading class $k$ job at time $t$, assuming that the class $k$ server devotes its entire service capacity to this job. If there is no class $k$ job  in service at time $t$, then $V_k(t)$ is the service time of the next class $k$ job in service.  

 In this paper all vectors as column vectors, but, notationally, column vectors are bulkier  than row vectors. Though we could write 
\begin{align*}
X(t) = (Z^{T}(t),U_1^{T}(t), V^{T}(t))^{T},
\end{align*}
keeping track of all the transpose scripts $^{T}$ is also cumbersome. Therefore, going forward we omit the script $^{T}$ and envision   vectors to be column vectors, unless we state otherwise. For example, we write $X(t) = (Z(t),U_1(t), V(t))$ 
to denote the column vector in \eqref{eq:xtdef}.

%
It is known that
$\{X(t), t\ge 0 \}$ is a continuous-time Markov process with state space
$S=\Z_+^5\times \R_+^6$, where $\Z_+=\{0, 1, \ldots\}$,  and  we call  $X(t)$  the state of the queueing network at time $t$. This process is
piecewise deterministic because between jumps, $X(t)$ evolves
deterministically in $t$. We adopt the convention that each
sample path of the state process is right continuous.


It follows from Theorem 4.1 of
\cite{Dai1995} and Section 8.7 of \cite{DaiHarr2020} that, under a
mild assumption on the inter-arrival time distribution, the Markov
process $\{X(t), t\ge 0 \}$ is positive Harris recurrent  and thus has a
unique stationary distribution  when the conditions
\begin{align}
  \label{eq:2s5c1}
  & \rho_1= \lambda_1 (m_1+m_3+m_5)<1, \\
  & \rho_2=\lambda_1 (m_2+m_4)<1,   \label{eq:2s5c2} \\  
  & \rho_v= \lambda_1(m_2+m_5)<1,   \label{eq:2s5cv0}
\end{align}
are satisfied. In such a case, we use
\begin{align*}
 X= (Z, U_1, V), \quad \mbox{where $Z = (Z_1, Z_2, Z_3, Z_4, Z_5)$ and $V = (V_1, V_2, V_3, V_4, V_5)$}, 
\end{align*}
to denote the random vector
distributed according to  the stationary distribution.

\begin{lemma}
\label{lem:2s_idle}
When conditions \eq{2s5c1}-\eq{2s5cv0} are satisfied, the following are satisfied  
\begin{align}
  & \beta_1\equiv \Prob\{Z_1=0, Z_3=0, Z_5=0\} = 1- \lambda_1(m_1+m_3+m_5)=1-\rho_1,\label{eq:2s_idle1}\\ 
  & \beta_3\equiv \Prob\{Z_3=0,Z_5=0\} = 1- \lambda_1(m_3+m_5), \quad
    \beta_5\equiv \Prob\{Z_5=0\} = 1- \lambda_1 m_5, \label{eq:2s_idle2} \\
 & \beta_4\equiv  \Prob\{Z_2=0,Z_4=0\} = 1- \lambda_1(m_2+m_4)=1-\rho_2, \quad
   \beta_2\equiv \Prob\{Z_2=0\} = 1- \lambda_1 m_2.   \label{eq:2s_idle3}
\end{align}
\end{lemma}

For a proof of  this lemma, see \lem{idle-prob} in \sectn{choosing-tf}.
Thus, when  \eq{2s5c1}-\eq{2s5cv0} hold,
the quantity
$\rho_i$  is the  long-run utilization of server $i \in \{1,2\}$.  Conditions \eq{2s5c1}  and \eq{2s5c2} ensure that servers $1$ and $2$ are not overloaded in the long run.
Condition
\eq{2s5cv0} is  known as the virtual station
condition, where $\rho_{v}$ is the traffic intensity of the virtual station  and is  \emph{unusual}. As explained  in \cite{DaiVand2000}, under the SBP discipline
(\ref{eq3.19}), classes $2$ and $5$ form a virtual station for which
condition \eq{2s5cv0} is the load condition.  When the Markov process has a  stationary distribution, we call the queueing
network stable.  For a general queueing network (to  be introduced in
Section~\ref{sec:multiclass})  operating under an
arbitrary SBP discipline,  characterizing its stability region in a manner similar to \eq{2s5c1}-\eq{2s5cv0} remains an open problem.

\subsection{Heavy-traffic limit theorem}

We consider a sequence of queueing
networks indexed by $r\in (0,1]$. Readers are referred to Section~\ref{sec:multiclass} for a motivation for studying a sequence of networks. For notational simplicity, only the arrival rate
$\lambda_1^{(r)}$   is assumed to depend on $r$.   We assume that $\lambda^{(r)}_1=1-r$ for $r\in (0,1]$
and that
\begin{align}
  & m_1+m_3+m_5=m_2+m_4=1, \quad    \label{eq:2s-m} \\
  & m_2+m_5< 1.  \label{eq:2s-mv}
\end{align}
Note that under condition \eq{2s-m}, \eq{2s-mv} is equivalent to  
  \begin{align}
  \label{eq:m5less}
   m_5< m_4.
  \end{align}
Under condition \eq{2s-m},    $\rho^{(r)}_1=\rho^{(r)}_2=1-r$. Thus,
\begin{align}
  \label{eq:ht2s}
  r^{-1}\big(1-\rho^{(r)}_1\big)=1 \quad \text{ and }\quad   r^{-1}\big(1-\rho^{(r)}_2\big)=1, \quad r \in (0,1).
\end{align}
 In particular, $\rho^{(r)}_i\uparrow 1$ for $i=1,2$,
as $r\to 0$.  Condition \eq{ht2s} is a special case of the heavy-traffic condition \eq{3.4}-\eq{b} to be introduced in Section \ref{sec:main}
for a general sequence of networks. Condition
\eq{2s-mv} implies stability of the queueing network for any $r \in (0,1)$, and we let $X^{(r)}$  denote the random element having the stationary distribution of the Markov process $\{X^{(r)}(t), t\ge 0\}$.
We let 
\begin{align*}
 Z^{(r)} = ( Z^{(r)}_1, \ldots, Z^{(r)}_5)^{T}
\end{align*}
be the column vector of steady-state job counts. The following
  proposition is a special case of \thm{main} in \sectn{main}. The
  SRBM in the proposition has been well studied; see, for example,
  Section 2.3 of \cite{BravDaiMiya2017}. To make this paper as self contained as possible,  the background materials on SRBM will be
  presented in \sectn{srbm}. 

\begin{proposition}\label{pro:2s}
 There exists a random element $(Z^*_1, Z^*_4) \in \R^2_+$ such that
  \begin{align}
    \label{eq:conv}
    r Z^{(r)} \Rightarrow  Z^*= (Z^*_1, 0, 0, Z^*_4, 0)^T, \quad \text{ as } r\to 0,
  \end{align}
  where ``$\Rightarrow$'' denotes convergence in
  distribution. Furthermore, the distribution of $(Z^*_1, Z^*_4)$ on $\R^2_+$ is
  the unique stationary distribution of a semimartingale reflecting
  Brownian motion (SRBM) with covariance matrix $\Sigma$, reflection matrix
  $R$, and drift vector $-Rb$, where
  \begin{align}
    \Sigma =
    \begin{pmatrix}
      \Sigma_{11} & \Sigma_{14} \\
      \Sigma_{41} & \Sigma_{44}
    \end{pmatrix}, \quad
    R =      \begin{pmatrix}
      R_{11} & R_{14} \\
      R_{41} & R_{44}
    \end{pmatrix} = \frac{1}{m_4-m_5}
    \begin{pmatrix}
      m_4 & -m_5 \\
      -1  & 1 
    \end{pmatrix}, \quad
    b=
        \begin{pmatrix}
      b_1\\ b_4
        \end{pmatrix}
    =
    \begin{pmatrix}
      1\\ 1
    \end{pmatrix},
    \label{eq:SRBM-data}
  \end{align}
and
  \begin{align}
    \Sigma_{11}   & = \frac{1} {2(\mu_5 -\mu_{4})^{2}} \left( (\mu_5 -\mu_{4})^{2} c_{e,1}^{2} + m_{1}^{2} \mu_{5}^{2} c_{s,1}^{2} + (\mu_{4} - 1)^{2} c_{s,2}^{2} + m_{3}^{2} \mu_{5}^{2} c_{s,3}^{2} + c_{s,4}^{2} + c_{s,5}^{2} \right), \label{eq:Sigma11}\\
  \Sigma_{14} & = \Sigma_{41} = \frac{-1} {(\mu_5 -\mu_{4})^{2}} \left(m_{1}^{2} \mu_{5}^{2} \mu_{4} c_{s,1}^{2} + (\mu_{4} - 1)^{2} \mu_{5} c_{s,2}^{2} + m_{3}^{2} \mu_{5}^{2} \mu_{4} c_{s,3}^{2} + \mu_{5} c_{s,4}^{2} + \mu_{4} c_{s,5}^{2} \right),\label{eq:Sigma14} \\    
  \Sigma_{44} & = \frac{1} {2(\mu_5 -\mu_{4})^{2}} \left(m_{1}^{2} \mu_{5}^{2} \mu_{4}^{2} c_{s,1}^{2} + (\mu_{4} - 1)^{2} \mu_{5}^{2} c_{s,2}^{2} + m_{3}^{2} \mu_{5}^{2} \mu_{4}^{2} c_{s,3}^{2} + \mu_{5}^{2} c_{s,4}^{2} + \mu_{4}^{2} c_{s,5}^{2} \right).\label{eq:Sigma44}
\end{align}
\end{proposition}
The remainder of this section is dedicated to proving the proposition,
first for the case when inter-arrival and service-time distributions
are exponential  and then for the case when inter-arrival and service
time distributions are general with bounded support.


\subsection{Exponential distributions}
\label{sec:exp}
In this section  we prove \pro{2s}, under the assumption that inter-arrival and
service-time distributions are exponential. In such a case, we can drop the components $(U^{(r)}_1(t), V^{(r)}(t))$ in the state description $X^{(r)}(t)$  because $\{Z^{(r)}(t), t\ge  0\}$ is a continuous-time Markov chain (CTMC) on the state space $S=\Z^5_+$, where $\Z_+=\{0, 1, \ldots\}$.  
When $\lambda^{(r)}_1=1-r$ and  \eq{2s-m}-\eq{2s-mv} are satisfied, each CTMC in the sequence has a unique stationary distribution  and we recall that  $Z^{(r)} \in \Z^5_+$  denotes the steady-state job count.

\paragraph{Basic adjoint relationship (BAR).}
For the moment, we focus on a single network within the sequence of
networks. We omit the index $r$ for convenience; e.g.,  $\lambda_1^{(r)}$ is denoted by $\lambda_1$.  The main
purpose of this section is to derive the Laplace transform  version of the BAR \eq{2sbar_exp}.  We use the terminology ``Laplace transform'' for a moment-generating function (MGF) if its domain is nonpositive.  

It is well known that the stationary distribution $\pi$ of the CTMC is characterized
by the  basic adjoint relationship (BAR) 
\begin{align}
  \label{eq:ctmcbar}
  \E[ G f (Z)] =0 \quad \text{ for each bounded function } f:S\to \R,
\end{align}
where for each state $z\in S$ and each function $f:S\to \R$,
\begin{align}
  G & f(z)   = \lambda_1 \big( f(z+e^{(1)})-f(z)\big) +
            \mu_1 \big( f(z-e^{(1)}+e^{(2)})-f(z)\big) 1(z_5=0, z_3=0, z_1>0) \nonumber \\
          &  \quad + \mu_2 \big( f(z-e^{(2)}+e^{(3)})-f(z)\big) 1(z_2>0) 
         +  \mu_3 \big( f(z-e^{(3)}+e^{(4)})-f(z)\big) 1(z_5=0, z_3>0) \nonumber \\            
&  \quad  +  \mu_4 \big( f(z-e^{(4)}+e^{(5)})-f(z)\big) 1(z_2=0, z_4>0) +  \mu_5 \big( f(z-e^{(5)})-f(z)\big) 1(z_5>0),  \label{eq:ctmcG}
\end{align}
and $e^{(j)} \in \R^{5}$ is the vector with a one in the $j$th component and zeros elsewhere. 
For a proof of \eq{ctmcbar}, see, for example, \cite{GlynZeev2008}. 
When all  states in $S$ are linearly ordered, each test function
$f: S\to \R$ is equivalent to a column vector of infinite
dimensions and $Gf$ is the usual matrix-vector
product, where $G$ is the corresponding square matrix known as the
generator matrix of the CTMC. 

The term
\begin{align*}
 \mu_3 \big(f(z-e^{(3)}+e^{(4)})-f(z)\big) 1(z_5=0, z_3>0) 
\end{align*}
represents the  state transition from state $z$ to
state $z-e^{(3)}+e^{(4)}$ due to the service completion of a
class $3$ job by server $1$. Due to the SBP
discipline in (\ref{eq3.19}), this can happen only when class 5 has no
job and class $3$ has jobs. Hence, the term has the indicator function  of the set $(z_5=0,z_3>0)$.  The service completion triggers a deletion 
of a job in class $3$ (the $-e^{(3)}$ term) and an addition of a job to
class $4$ (the $+e^{(4)}$ term). The $\mu_3$ term reflects the service
rate when server $1$ is fully dedicated to a class $3$ job.
Other terms
in the definition of $Gf$ can be understood similarly. 

Equation \eq{ctmcbar} is a 
shorthand  for
\begin{align*}
  \sum_{z\in S} \Prob\{Z=z\} Gf(z)=0
  \quad \text{ or }  \quad  \sum_{z\in S}\pi(z) Gf(z)=0
  \quad \text{ for each } f:S\to \R.
\end{align*}
The latter sum  is equal to  $\pi G f$  when the stationary distribution  $\pi$ is viewed as a row vector, $G$ as a square
matrix, and $f$ as a column vector. Clearly, $\pi G f=0$ for each bounded
$f:S\to \R$ is equivalent to
\begin{align*}
 \pi G=0, 
\end{align*}
 which is known as the balance equations that characterize
the stationary distribution $\pi$ of a CTMC with generator $G$.

Throughout this paper, we use the following notion. For each integer $d>0$, let
\begin{align*}
  \R^d_- = \big\{\, x=(x_1, \ldots,x_d)^T\in \R^d: x_i\le 0 \text{ for } =1, \ldots, d\,\big\}.
\end{align*}
 Fixing  a  $\theta\in \R^5_-$, we define the bounded test function $g_\theta: S\to \R$ by
\begin{align}
  \label{eq:f2s}
  g_\theta(z) = e^{\langle \theta, z\rangle},
\end{align}
where, for $a,b\in \R^5$, $\langle a, b\rangle=\sum_{k=1}^5{a_k b_k}$.
 Applying $G$ to this function, one has 
\begin{align}
  G  g_\theta&(z)=  \Big[  \lambda_1 \eta_1(\theta_1) + \mu_1 \xi_1(\theta) 1( z_1>0,z_3=0,z_5=0) + \mu_2 \xi_2(\theta)1(z_2>0) \nonumber  \\
  & +
    \mu_3\xi_3(\theta)1(z_3>0, z_5=0 ) +
    \mu_4\xi_4(\theta)1(z_2=0,z_4>0) 
    + \mu_5 \xi_5(\theta)1( z_5>0) \Big] g_\theta(z). \label{eq:2sGf}
\end{align}
where
\begin{align}\label{eq:2s-etaxi}
    \eta_1(\theta_1)=e^{\theta_1}-1, \quad 
   \xi_k(\theta)=e^{\theta_{k+1}-\theta_k}-1 \quad \text{ for } k\in\{1,2,3, 4\},  \quad 
   \xi_5(\theta)=e^{-\theta_5}-1.
\end{align}
It follows from the BAR \eq{ctmcbar} and \eq{2sGf}  that
\begin{align}
 &  \lambda_1\eta_1(\theta_1) \E[g_{\theta}(X)] +  \mu_1 \xi_1(\theta) \E[g_{\theta}(X)1(Z_1>0, Z_3=0, Z_5=0 )\Big] \nonumber \\
  & {} + \mu_3 \xi_3(\theta) \E[g_{\theta}(X)1(Z_3>0, Z_5=0) ]
 + \mu_5 \xi_5(\theta) \E[g_{\theta}(X)1(Z_5>0 )] \nonumber \\
  & { }+ \mu_2 \xi_2(\theta) \E[ g_{\theta}(X)1(Z_2>0)]+
  \mu_4 \xi_4(\theta)\E[ g_{\theta}(X)1(Z_2=0, Z_4 > 0)]=0 \text{ for } \theta\in \R^5_-. \label{eq:2s-bar-exp}
\end{align}
 We call this the Laplace transform version of the BAR \eq{ctmcbar}. Let us define 
\begin{align}
  & \phi(\theta)=\E[g_\theta(Z)], \quad    \phi_1(\theta) = \E[g_\theta(Z) \mid Z_1=0, Z_3=0, Z_5=0], \label{eq:2sphi} \\
  &  \phi_3(\theta) = \E[g_\theta(Z) \mid  Z_3=0, Z_5=0], \quad  \phi_5(\theta) = \E[g_\theta(Z) \mid Z_5=0], \label{eq:2sph3}\\
  & \phi_2(\theta) = \E[g_\theta(Z) \mid Z_2=0],
      \quad  \phi_4(\theta) = \E[g_\theta(Z) \mid Z_2=0, Z_4=0]. \label{eq:2sphi4}
\end{align}
The following lemma rewrites \eqref{eq:2s-bar-exp} in a more convenient form. Namely,  \eqref{eq:2s-bar-exp} is written as a linear combination form of $\phi(\theta)$ and $\phi(\theta) - \phi_{i}(\theta)$ for $i=1,2,\ldots 5$.
\begin{lemma}
For each $\theta\in \R^5_-$
\begin{align}
  & \Big(\lambda_1 \eta_1(\theta_1)+  \sum_{k=1}^5 \lambda_1 \xi_k(\theta)\Big) \phi(\theta) +
     [\mu_5\xi_5(\theta) -
     \mu_3\xi_3(\theta)]
    \beta_5 \big(\phi(\theta)-\phi_5(\theta)\big)  \nonumber \\
  & {} + [\mu_3\xi_3(\theta) -
     \mu_1\xi_1(\theta)]
    \beta_3 \big(\phi(\theta)-\phi_3(\theta)\big) + \mu_1 \xi_1(\theta) (1-\rho_1) \big(\phi(\theta)- \phi_1(\theta)\big) \nonumber \\
    & + \mu_4 \xi_4(\theta) (1-\rho_2) \big(\phi(\theta)- \phi_4(\theta)\big) +  [\mu_2\xi_2(\theta) -
     \mu_4\xi_4(\theta)]
    \beta_2 \big(\phi(\theta)-\phi_2(\theta)\big)=0,\label{eq:2sbar_exp}
\end{align}
where $\beta_k$, defined in  \lem{2s_idle}, is the steady-state probability that all classes with priority greater or equal to class $k$ have no customers.
\end{lemma}
\begin{proof}
Our starting point is  \eq{2s-bar-exp}. Consider the  last line in \eq{2s-bar-exp}. \lem{2s_idle}  implies that 
\begin{align*}
 & \E[g_\theta(Z)1(Z_2>0)] = \E[g_\theta(Z)] - \E[g_\theta(Z)1(Z_2=0)]
   = \phi(\theta) - \beta_2 \phi_2(\theta), \\
 &  \E[g_\theta(Z)1(Z_2=0, Z_4>0)]   = \beta_2 \phi_2(\theta) - \beta_4 \phi_4(\theta).
\end{align*}
 Hence,
\begin{align*}
  &   \E\Big[ \Big(\mu_2 \xi_2(\theta )1(Z_2>0)+ \mu_4 \xi_4(\theta) 1(Z_2=0, Z_4>0)\Big)g_\theta(Z)\Big] \nonumber\\
   = {} & \mu_2\xi_2(\theta) \phi(\theta)+ [\mu_2\xi_2(\theta) -
     \mu_4\xi_4(\theta)]
    \beta_2 (-\phi_2(\theta))+ \mu_4 \xi_4(\theta) \beta_4  (-\phi_4(\theta)) \nonumber \\
   = {} &  \Big(\mu_2\xi_2(\theta) - [\mu_2\xi_2(\theta) -
     \mu_4\xi_4(\theta)]
    \beta_2  -\mu_4\xi_4(\theta)\beta_4 \Big)  \phi(\theta) \nonumber\\
  & + [\mu_2\xi_2(\theta) -
     \mu_4\xi_4(\theta)]
    \beta_2 \big(\phi(\theta)-\phi_2(\theta)\big) + \mu_4 \xi_4(\theta) \beta_4 \big(\phi(\theta)- \phi_4(\theta)\big)\\
   = {} &  \Big(\mu_2\xi_2(\theta)(1-\beta_2) +
     \mu_4\xi_4(\theta) (
    \beta_2 - \beta_4) \Big)  \phi(\theta) \nonumber\\
  & + [\mu_2\xi_2(\theta) -
     \mu_4\xi_4(\theta)]
    \beta_2 \big(\phi(\theta)-\phi_2(\theta)\big) + \mu_4 \xi_4(\theta) \beta_4 \big(\phi(\theta)- \phi_4(\theta)\big).
\end{align*}
From \lem{2s_idle} we know that $\beta_4=1- \lambda_1(m_2+m_4)=1-\rho_2$ and $\beta_2= 1- \lambda_1 m_2$, implying that the right-hand side equals 
\begin{align}
 & (\lambda_1 \xi_2(\theta) + \lambda_1 \xi_4(\theta)) \phi(\theta) \nonumber \\
  & + [\mu_2\xi_2(\theta) -
     \mu_4\xi_4(\theta)]
    \beta_2 \big(\phi(\theta)-\phi_2(\theta)\big) + \mu_4 \xi_4(\theta) (1-\rho_2) \big(\phi(\theta)- \phi_4(\theta)\big).  \label{eq:2s_station2}
\end{align}
Similarly, one can show   that
\begin{align}
  &   \E\Big[ \Big(\mu_5 \xi_5(\theta )1(Z_5>0)+ \mu_3 \xi_3(\theta) 1( Z_3>0,Z_5=0)  +  \mu_1 \xi_1(\theta) 1(Z_1>0,Z_3=0, Z_5=0) \Big)g_\theta(Z)\Big] \nonumber \\
  = {}& \Big(\lambda_1 \xi_1(\theta) + \lambda_1 \xi_3(\theta) +  \lambda_1 \xi_5(\theta) \Big) \phi(\theta) +  [\mu_5\xi_5(\theta) -
     \mu_3\xi_3(\theta)]
    \beta_5 \big(\phi(\theta)-\phi_5(\theta)\big) \nonumber \\
  & + [\mu_3\xi_3(\theta) -
     \mu_1\xi_1(\theta)]
    \beta_3 \big(\phi(\theta)-\phi_3(\theta)\big) + \mu_1 \xi_1(\theta) (1-\rho_1) \big(\phi(\theta)- \phi_1(\theta)\big).  \label{eq:2s_station1}
\end{align}
Finally, \eq{2sbar_exp} follows from \eq{2s-bar-exp}  and the expressions in 
\eq{2s_station2} and \eq{2s_station1}.
\end{proof}

\paragraph{Taylor expansion and asymptotic BAR.}
For each $\theta\in \R^5_-$, define
\begin{align*}
  \phi^{(r)}(\theta) = \E[g_{\theta}(rZ^{(r)})]= \E[g_{r\theta}(Z^{(r)})],
\end{align*}
which is analogous to the definition of $\phi(\theta)$ in \eq{2sphi}.
Define $\phi^{(r)}_k(\theta)$ similarly for $k\in \{1, \ldots, 5\}$. We now present Lemma~\ref{lem:abar}, in which we   start with \eqref{eq:2sbar_exp} and replace $\eta_{1}(\theta_{1})$ and $\xi_{k}(\theta)$   by their second-order Taylor expansions, which are simpler to work with, to derive what we call the asymptotic BAR in \eq{2s-abar}. Later we will see how the asymptotic BAR allows us to characterize the limiting distribution of $r Z^{(r)}$ as $r \to 0$. 

To state Lemma~\ref{lem:abar}, we define 
\begin{align}
   \bar{\eta}_1(\theta_1) = \theta_1 \quad \text{ and } \quad \tilde{\eta}_1(\theta_1)=\frac{1}{2} \theta_1^2. \label{eq:2s-bareta}
\end{align}
We will see in the proof of Lemma~\ref{lem:abar} that $\bar{\eta}_1(\theta_1)$ and $\tilde{\eta}_1(\theta_1)$ are the first- and second-order terms, respectively,  of the Taylor expansion  of $\eta_{1}(\theta_1)$. Similarly, we define 
\begin{align}
  & \bar{\xi}_k(\theta) = \theta_{k+1}-\theta_k, \quad
    \tilde{\xi}_k(\theta) = \frac{1}{2}(\theta_{k+1}-\theta_k)^2 \quad k\in\{1,\ldots, 4\}, \label{eq:2s-barxi}\\
  &\bar{\xi}_5(\theta)=-\theta_{5}, \quad  \tilde {\xi}_5(\theta)= \frac{1}{2}\theta_{5}^2. \label{eq:2s-barxi5}  
\end{align}
Lastly, we define 
\begin{align}
   \eta^*_1(\theta_1)= \bar{\eta}_1(\theta_1) + \tilde{\eta}_1(\theta_1), \quad
   \xi^*_k(\theta)=\bar{\xi}_k(\theta) + \tilde{\xi}_k(\theta), \quad k\in\{1, \ldots, 5\}, \label{eq:2s-star}
\end{align}
to be the second-order approximations of $\eta_{1}(\theta_1)$ and $\xi_{k}(\theta)$, respectively. 

\begin{lemma}
\label{lem:abar}
For each $\theta\in \R^5_-$, as $r\to 0$, 
\begin{align}
  & r^2\Big(\lambda^{(r)}_1 \tilde{\eta}_1(\theta_1)+  \sum_{k=1}^5 \lambda^{(r)}_1\tilde{\xi_k}(\theta)\Big) \phi^{(r)}(\theta)  \nonumber\\
  & \quad + r^2 \mu_1 \bar{\xi}_1(\theta)  \big(\phi^{(r)}(\theta)- \phi^{(r)}_1(\theta)\big) 
  + \mu_4 r^2 \bar{\xi}_4(\theta) \big(\phi^{(r)}(\theta)- \phi^{(r)}_4(\theta)\big)
    \nonumber \\
  & \quad + [\mu_3\xi^*_3(r\theta) -
     \mu_1\xi^*_1(r\theta)]
    \beta^{(r)}_3 \big(\phi^{(r)}(\theta)-\phi^{(r)}_3(\theta)\big)  \nonumber\\
   & \quad + [\mu_5\xi^*_5(r\theta) -
     \mu_3\xi^*_3(r\theta)]
    \beta^{(r)}_5 \big(\phi^{(r)}(\theta)-\phi^{(r)}_5(\theta)\big)
 \nonumber \\
  & \quad    + 
    [\mu_2\xi^*_2(r\theta) -
     \mu_4\xi^*_4(r\theta)]
    \beta^{(r)}_2 \big(\phi^{(r)}(\theta)-\phi^{(r)}_2(\theta)\big)=o(r^2).\label{eq:2s-abar}
\end{align}
\end{lemma}
\begin{proof}

 Replacing
$\theta$ in \eq{2sbar_exp} by $r\theta$, one has that for each $\theta\in \R^5_-$ and each $r\in (0,1)$,
\begin{align*}
  & \Big(\lambda^{(r)}_1 \eta_1(r\theta_1)+  \sum_{k=1}^5 \lambda^{(r)}_1 \xi_k(r\theta)\Big) \phi^{(r)}(\theta) +
     [\mu_5\xi_5(r\theta) -
     \mu_3\xi_3(r\theta)]
    \beta^{(r)}_5 \big(\phi^{(r)}(\theta)-\phi^{(r)}_5(\theta)\big)  \nonumber \\
  & {} + [\mu_3\xi_3(r\theta) -
     \mu_1\xi_1(r\theta)]
    \beta^{(r)}_3 \big(\phi^{(r)}(\theta)-\phi^{(r)}_3(\theta)\big) + \mu_1 \xi_1(r\theta) (1-\rho^{(r)}_1) \big(\phi^{(r)}(\theta)- \phi^{(r)}_1(\theta)\big) \nonumber \\
    & + \mu_4 \xi_4(r\theta) (1-\rho^{(r)}_2) \big(\phi^{(r)}(\theta)- \phi^{(r)}_4(\theta)\big) +  [\mu_2\xi_2(r\theta) -
     \mu_4\xi_4(r\theta)]
    \beta^{(r)}_2 \big(\phi^{(r)}(\theta)-\phi^{(r)}_2(\theta)\big)=0.
\end{align*}
Using the Taylor expansion  $e^x = 1 + x + \frac{1}{2}x^2 + o(x)$ when $x\to 0$, one has
\begin{align}
  &  \eta_1(\theta_1) = \eta^*_1(\theta_1) + o(\theta^2_1) \text{ as } \theta_1\to 0, \label{eq:2s-taylor-1}\\
  &  \xi_k(\theta) =\xi_k^*(\theta) +  o(\abs{\theta}^2)  \text{ as } \theta \to 0 \text{ for } k\in\{1,\ldots, 5\}. \label{eq:2s-taylor-2} 
\end{align}
Therefore,
for each $\theta\in \R^5_-$ as $r\to 0$,
\begin{align*}
  & \Big(\lambda^{(r)}_1 \eta^*_1(r\theta_1)+  \sum_{k=1}^5 \lambda^{(r)}_1 \xi^*_k(r\theta)\Big) \phi^{(r)}(\theta) +
     [\mu_5\xi^*_5(r\theta) -
     \mu_3\xi^*_3(r\theta)]
    \beta^{(r)}_5 \big(\phi^{(r)}(\theta)-\phi^{(r)}_5(\theta)\big)  \nonumber \\
  & {} + [\mu_3\xi^*_3(r\theta) -
     \mu_1\xi^*_1(r\theta)]
    \beta^{(r)}_3 \big(\phi^{(r)}(\theta)-\phi^{(r)}_3(\theta)\big) + \mu_1 \xi^*_1(r\theta) (1-\rho^{(r)}_1) \big(\phi^{(r)}(\theta)- \phi^{(r)}_1(\theta)\big) \nonumber \\
    & + \mu_4 \xi^*_4(r\theta) (1-\rho^{(r)}_2) \big(\phi^{(r)}(\theta)- \phi^{(r)}_4(\theta)\big) +  [\mu_2\xi^*_2(r\theta) -
     \mu_4\xi^*_4(r\theta)]
    \beta^{(r)}_2 \big(\phi^{(r)}(\theta)-\phi^{(r)}_2(\theta)\big)=o(r^2).
\end{align*}
Using the facts that $\bar{\eta}_1(\theta_1)+\sum_{k=1}^5\bar{\xi}_k(\theta)=0$, that  
$\tilde{\xi}_k(r\theta)=r^2\tilde{\xi}_k(\theta)$ for each $\theta\in \R^5_-$, and that $1-\rho^{(r)}_i=r$, we have \eq{2s-abar}, proving the lemma.
\end{proof}
\paragraph{State space collapse (SSC).}
It follows from Theorem 3.7 and Section 4.1 of \cite{CaoDaiZhan2022} that the following moment state space collapse (SSC) holds.
\begin{align}
  \label{eq:2s-ssc-m}
 \limsup_{r\to 0}\E[Z^{(r)}_2+ Z^{(r)}_3 + Z^{(r)}_5]^2<\infty.
\end{align}
In fact, we now argue that as a consequence of \eqref{eq:2s-ssc-m}, for any $\theta\in \R^5_-$,
\begin{align}
  \label{eq:2s-ssc-mgf}
  \lim_{r\to 0} \Big( \phi^{(r)}(\theta) -  \phi^{(r)}(\theta_L,0)\Big)=0,
  \quad  \lim_{r\to 0} \Big( \phi^{(r)}_k(\theta) -  \phi^{(r)}_k(\theta_L,0)\Big)=0, \quad k\in\{1, \ldots, 5\},
\end{align}
where, for a function $f:\R^5\to \R$, $f(\theta_L, 0)$ is a shorthand for
$f(\theta_1, 0, 0, \theta_4, 0)$ with $\theta_L=(\theta_1, \theta_4)^T$. We refer to \eqref{eq:2s-ssc-mgf} as the Laplace transform version of SSC. To prove the first equality in \eq{2s-ssc-mgf}, we note that 
\begin{align*}
  &  \abs{ \phi^{(r)}(\theta) -  \phi^{(r)}(\theta_L,0)}\le   \E\Big [ 1- \Big( e^{\sum_{k\in\{2,3,5\}}\theta_k rZ^{(r)}_k }\Big)\Big] 
\le  \E\Big[\sum_{k\in\{2,3,5\}}\abs{\theta_k} rZ^{(r)}_k \Big],
\end{align*}
where the second inequality follows from $1-e^{-x}\le x$ for $x\ge 0$.
The last term converges to $0$ because of \eq{2s-ssc-m} and Jensen's inequality. The rest of \eqref{eq:2s-ssc-mgf} is proved similarly. 

\begin{proof}[Proof of \pro{2s}.]
The SSC in the preceding paragraph implies that $\lim_{r \to 0} r Z_{k}^{(r)} \Rightarrow 0$ for $k \in \{2,3,5\}$. To prove \pro{2s}, it remains to show that $(rZ_{1}^{(r)}, rZ_{4}^{(r)})$ converges  and characterize the  limit. We begin with the following lemma, which is stated for the general queueing network setting in
\lem{limits}.

\begin{lemma}\label{lem:vaguelaplace}
For any sequence $\{(\phi^{(r_{n})}(\theta), \phi^{(r_{n})}_1(\theta), \ldots, \phi^{(r_{n})}_5(\theta))\}_{n=1}^{\infty}$ with $r_n \in (0,1)$ and $\lim_{n\to \infty}r_n=0$, there exists a subsequence indexed by $\{r_{n_k}\}$ such that
\begin{align*}
  \lim_{k\to\infty} \Big(\phi^{(r_{n_k})}(\theta), \phi^{(r_{n_k})}_1(\theta), \ldots, \phi^{(r_{n_k})}_5(\theta)\Big) =\Big(\phi^*(\theta), \phi^*_1(\theta), \ldots, \phi^*_5(\theta)\Big) \quad \text{ for each } \theta \in \R^5_-.
\end{align*}
We call $(\phi^*(\theta), \phi^*_1(\theta), \ldots, \phi^*_5(\theta))$ a limit point  of $\{(\phi^{(r)}(\theta), \phi^{(r)}_1(\theta), \ldots, \phi^{(r)}_5(\theta))\}_{r  \in (0,1)}$.
\end{lemma}

 We show that the set of all limit points in Lemma~\ref{lem:vaguelaplace} is a singleton by proving that there is a random vector $(Z_1^*, Z_4^*) \in \R^2_+$, independent of the subsequence $\{r_{n_k}\}$, such that
  \begin{align}\label{eq:2s-p1}
    \phi^*(\theta_1,\theta_2,\theta_3,\theta_4,\theta_5) = \E\Big[ e^{\theta_1 Z_1^*+ \theta_4 Z^*_4}\Big] \quad \text{ for each }  \theta\in \R^5_-.
  \end{align} 
  Since every sequence $\{\phi^{(r_n)}(\theta)\}_{n=1}^{\infty}$ contains a convergent subsequence that converges to the limit point defined by \eqref{eq:2s-p1}, it follows that 
  \begin{align*}
    \lim_{r\to 0 } \phi^{(r)}(\theta) = \phi^*(\theta) =  \E\Big[ e^{\theta_1 Z_1^*+ \theta_4 Z^*_4}\Big]   \quad \text{ for each } (\theta_1, \theta_4)^T\in \R^2_-,
  \end{align*}
  which is equivalent to the convergence in \eq{conv}.  The following informal discussion outlines how we prove \eqref{eq:2s-p1}.

Let us assume for simplicity that  $\phi^{(r)}(\theta), \phi^{(r)}_1(\theta), \ldots, \phi^{(r)}_5(\theta)$ converge pointwise as $r \to 0$. Otherwise, we can replace $r$ by $r_{n_{k}}$ and be assured that $\phi^{(r_{n_k})}(\theta), \phi^{(r_{n_k})}_1(\theta), \ldots, \phi^{(r_{n_k})}_5(\theta)$ converge.  We characterize the limit point  using the  asymptotic BAR \eqref{eq:2s-abar} as follows.  Dividing both sides of \eqref{eq:2s-abar} by $r^2$ and letting $r \to 0$ yields 
\begin{align}
  &  \Big( \tilde{\eta}_1(\theta_1)+  \sum_{k=1}^5  \tilde{\xi_k}(\theta)\Big) \phi^{*}(\theta)    + \mu_1 \bar{\xi}_1(\theta)  \big(\phi^{*}(\theta)- \phi^{*}_1(\theta)\big) 
  + \mu_4  \bar{\xi}_4(\theta) \big(\phi^{*}(\theta)- \phi^{*}_4(\theta)\big)
    \nonumber \\
  &= - \lim_{r \downarrow 0}  \frac{1}{r^2}\bigg( [\mu_3\xi^*_3(r\theta) -
     \mu_1\xi^*_1(r\theta)]
    \beta^{(r)}_3 \big(\phi^{(r)}(\theta)-\phi^{(r)}_3(\theta)\big)  \nonumber\\
   & \hspace{2.5cm} + [\mu_5\xi^*_5(r\theta) -
     \mu_3\xi^*_3(r\theta)]
    \beta^{(r)}_5 \big(\phi^{(r)}(\theta)-\phi^{(r)}_5(\theta)\big)
 \nonumber \\
  & \hspace{2.5cm}   + 
    [\mu_2\xi^*_2(r\theta) -
     \mu_4\xi^*_4(r\theta)]
    \beta^{(r)}_2 \big(\phi^{(r)}(\theta)-\phi^{(r)}_2(\theta)\big) \bigg), \quad \theta \in \R^{5}_{-}. \label{eq:asrel}
\end{align}
 We proceed in two steps. In step one, we identify a subset of $\R^{5}_{-}$ such that the right-hand side of \eqref{eq:asrel} is zero for all $\theta$ in this subset. We do this because we are unable to characterize the right-hand side outside this subset. Step one requires Lemmas~\ref{lem:2s-1}, \ref{lem:2s-2}, and \ref{lem:2s-abar-2}, which are stated below. Following these lemmas, we use \eqref{eq:asrel}, the right-hand side of which now equals zero, to characterize $\phi^{*}(\theta)$ and prove \eqref{eq:2s-p1}---this is step two.  

Let us compare our example to \cite{BravDaiMiya2017}, who applied the BAR approach with Laplace transforms to   generalized Jackson networks (GJNs).  In that paper, the authors derived an asymptotic BAR for GJNs, which allowed them to obtain an equation that is analogous to \eqref{eq:asrel}. However, since GJNs are single-class queueing networks, the right-hand side of their equation equals zero. This means that \cite{BravDaiMiya2017} did not need to perform step one of the previous paragraph, whereas we do  because our two-station example is a multiclass queueing network.  
 
We now carry out step one. To understand how to choose $\theta$ so   the right-hand side of \eqref{eq:asrel} equals zero, we examine the first term inside the parentheses. Namely, 
\begin{align*}
& \quad \frac{1}{r^2} [\mu_3\xi^*_3(r\theta) -
     \mu_1\xi^*_1(r\theta)]
    \beta^{(r)}_3 \big(\phi^{(r)}(\theta)-\phi^{(r)}_3(\theta)\big)  \\
&= \frac{1}{r^2} [\mu_3\bar \xi_3(r\theta) -
     \mu_1\bar \xi_1(r\theta)]
    \beta^{(r)}_3 \big(\phi^{(r)}(\theta)-\phi^{(r)}_3(\theta)\big) \\
    &\quad + \frac{1}{r^2} [\mu_3\tilde \xi_3(r\theta) -
     \mu_1\tilde \xi_1(r\theta)]
    \beta^{(r)}_3 \big(\phi^{(r)}(\theta)-\phi^{(r)}_3(\theta)\big).    
\end{align*} 
Since $\bar \xi_k(\theta)$  are linear  in $\theta$, we show in Lemma~\ref{lem:2s-1} that we can choose $\theta$ to make the first term on the right-hand side equal zero. Furthermore, $\sup_{r \in (0,1)} \abs{\tilde \xi_{k}(r\theta)}/r^2 < \infty$ because $\tilde \xi_{k}(\theta)$ are quadratic in $\theta$;  see \eqref{eq:2s-barxi}. Therefore, to prove that the second term vanishes as $r \to 0$, we show in Lemma~\ref{lem:2s-2} that $\lim_{r \to 0} (\phi^{(r)}(\theta) -  \phi^{(r)}_k(\theta))  = 0$ for $k \in \{2,3,5\}$. 

Recall that all vectors are envisioned as column vectors, and recall our convention of writing column vectors  discussed in Section~\ref{sec:2s-stability}.
\begin{lemma}\label{lem:2s-1}
 Recall the definition of $\bar{\xi}_k(\theta)$ from  \eq{2s-barxi} and consider the system of linear equations 
\begin{align}
  &  \mu_3\bar{\xi}_3(\theta) -  \mu_1\bar{\xi}_1(\theta)= \mu_3(\theta_4 - \theta_3) - \mu_1(\theta_2-\theta_1) =0, \label{eq:s2-f1} \\
  &  \mu_5\bar{\xi}_5(\theta) -  \mu_3\bar{\xi}_3(\theta)= -\mu_5 \theta_5 - \mu_3(\theta_4-\theta_3) =0,\label{eq:s2-f2} \\
  &  \mu_2\bar{\xi}_2(\theta) -  \mu_4\bar{\xi}_4(\theta)= \mu_2(\theta_3 - \theta_2) - \mu_4(\theta_5-\theta_4) =0.\label{eq:s2-f3}
\end{align}
For each fixed $\theta_{L} = (\theta_1,\theta_4) \in \R^{2}$,  
there exists a unique $h(\theta_L) = (\theta_2,\theta_3,\theta_5)$ such that $\theta=(\theta_1, \theta_2, \theta_3, \theta_4, \theta_5)$
satisfies \eq{s2-f1}-\eq{s2-f3}.
Furthermore, the set $\Theta_{L}$, defined as
\begin{align}
  \label{eq:2s-thetaL-region}
   \Theta_L =\Big\{ (\theta_1, \theta_4)\in \R^2_-: m_4-\frac{\mu_5m_4-1}{m_1\mu_5} < \frac{\theta_4}{\theta_1} < m_4 \Big\},
\end{align}
 is a non-empty and open set, and
\begin{align}
\label{eq:2s5c-negative}
   h(\theta_L)=(\theta_2,\theta_3,\theta_5)  < 0 \quad \text{ for all } \quad \theta_L=(\theta_1,\theta_4)\in \Theta_L. 
\end{align}
\end{lemma}
%
\begin{proof}
  Fix a $\theta_L=(\theta_1, \theta_4)\in \R^2$. One can verify that
  $\theta=(\theta_1, \theta_2, \theta_3, \theta_4, \theta_5)$ with 
  \begin{align}
&  \theta_5 = \frac{1}{\mu_5m_4-1}[  m_4\theta_1-\theta_4], \label{eq:2s-theta5} \\
& \theta_2 = \theta_1 - m_1\mu_5 \theta_5, 
 \label{eq:2s-theta2}\\ 
   &  \theta_3= \theta_4 + m_3\mu_5 \theta_5.
  \label{eq:2s-theta3} 
\end{align}
satisfies  equations \eq{s2-f1}-\eq{s2-f3}.
Now for  $\theta_L\in \Theta_L$ we  have   $\theta_4> m_4\theta_1$ and $\theta_4 < 0$, which implies that
$\theta_5<0$ and $\theta_3 < 0$.  Finally,
$\theta_2<0$ follows from $m_4-\frac{\mu_5m_4-1}{m_1\mu_5} < \frac{\theta_4}{\theta_1}$ and $\theta_1<0$.
\end{proof}
\begin{lemma} \label{lem:2s-2} For each $\theta \in \R^5_-$,
  \begin{align}\label{eq:2s-ssc-h}
    \lim_{r\to 0} \Big( \phi^{(r)}(\theta) -  \phi^{(r)}_k(\theta)\Big)=0,     \quad k\in\{2, 3, 5\}.
  \end{align}
\end{lemma}
\begin{proof}
  We prove the lemma for $k=2$. Other cases can be proved similarly. Recalling from \eqref{eq:2s-star} that $\xi^*_k(\theta)=\bar{\xi}_k(\theta) + \tilde{\xi}_k(\theta)$, it follows from \eq{2s-abar} that for each $\theta\in \R^5_-$,
  \begin{align}
      & [\mu_3\bar{\xi}_3(\theta) -
     \mu_1\bar{\xi}_1(\theta)]
    \beta^{(r)}_3 \big(\phi^{(r)}(\theta)-\phi^{(r)}_3(\theta)\big)  + [\mu_5\bar{\xi}_5(\theta) -
     \mu_3\bar{\xi}_3(\theta)]
    \beta^{(r)}_5 \big(\phi^{(r)}(\theta)-\phi^{(r)}_5(\theta)\big)
 \nonumber \\
  & {}    + 
    [\mu_2\bar{\xi}_2(\theta) -
     \mu_4\bar{\xi}_4(\theta)]
    \beta^{(r)}_2 \big(\phi^{(r)}(\theta)-\phi^{(r)}_2(\theta)\big)=o(1).\label{eq:2s-f}
  \end{align} 
 For 
 each fixed   $\theta_L = (\theta_1, \theta_4) \in \R^2_-$ and $\theta_5\in \R$, set
$\theta_2$ and $\theta_3$ follow \eq{2s-theta2} and \eq{2s-theta3}, respectively.
One can verify  that
  $\theta=(\theta_1, \theta_2, \theta_3, \theta_4, \theta_5)$ satisfies \eq{s2-f1} and \eq{s2-f2}. Furthermore,   it follows from \eq{2s-theta5} that 
  one can choose $\theta_5<0$ small enough  so that $\theta_2<0$, $\theta_3<0$ and
\begin{align}
  &  \mu_2\bar{\xi}_2(\theta) -  \mu_4\bar{\xi}_4(\theta) \neq 0.\label{eq:2s-f3a}
\end{align}
For this choice of $\theta=(\theta_1, \theta_2, \theta_3, \theta_4, \theta_5)$,  \eq{2s-f} gives
\begin{align*}
  \lim_{r\to 0}\Big( \phi^{(r)}(\theta)-\phi^{(r)}_2(\theta)\Big)=0,
\end{align*}
which,  together with SSC \eq{2s-ssc-mgf},  yields
\begin{align}\label{eq:2s-temp1}
  \lim_{r\to 0}\Big( \phi^{(r)}(\theta_L,0)-\phi^{(r)}_2(\theta_L, 0)\Big)=0 \quad \text{ for each } \quad \theta_{L} \in \R^{2}_{-}.
\end{align}
Now for any $\theta\in \R^5_-$, equation  \eq{2s-temp1} and SSC \eq{2s-ssc-mgf} imply
\eq{2s-ssc-h} for $k=2$.
\end{proof}  
\begin{lemma}\label{lem:2s-abar-2}
  For each $\theta_L\in \Theta_L$, let $\theta=(\theta_L, \theta_H)$ be the unique
  $\theta$ that satisfies \eq{s2-f1}-\eq{s2-f3}. 
  Then any limit point $(\phi^*(\theta), \phi^*_1(\theta), \ldots, \phi^*_5(\theta))$ satisfies
\begin{align}
 0   &=  \Big( \tilde{\eta}_1(\theta_1)+  \sum_{k=1}^5  \tilde{\xi_k}(\theta)\Big) \phi^{*}(\theta)    + \mu_1 \bar{\xi}_1(\theta)  \big(\phi^{*}(\theta)- \phi^{*}_1(\theta)\big) 
  + \mu_4  \bar{\xi}_4(\theta) \big(\phi^{*}(\theta)- \phi^{*}_4(\theta)\big) \nonumber \\
  &=   \Big( \tilde{\eta}_1(\theta_1)+  \sum_{k=1}^5  \tilde{\xi_k}(\theta)\Big) \phi^{*}(\theta_{L},0)    + \mu_1 \bar{\xi}_1(\theta)  \big(\phi^{*}(\theta_{L},0)- \phi^{*}_1(\theta_{L},0)\big) \nonumber
  \\
  & \qquad + \mu_4  \bar{\xi}_4(\theta) \big(\phi^{*}(\theta_{L},0)- \phi^{*}_4(\theta_{L},0)\big), \hspace{6cm} \theta_{L} \in \Theta_{L}. \label{eq:2s-bar-2}
\end{align}  
\end{lemma}
\begin{proof}
The first equality follows by combining Lemmas~\ref{lem:2s-1} and \ref{lem:2s-2} with the discussion preceding Lemma~\ref{lem:2s-1}. The second equality follows from the Laplace transform  version of SSC \eq{2s-ssc-mgf}. 
\end{proof}

We now prove that for
  each limit point $\phi^*$, \eq{2s-p1} holds for some random vector
  $(Z_1^*, Z^*_4)$ that is independent of the subsequence that
  generates the limit point.

Our starting point is \eq{2s-bar-2} in Lemma~\ref{lem:2s-abar-2}.
  In \eq{2s-bar-2}, for each
  $\theta_{L}=(\theta_1, \theta_4) \in \Theta_{L}$,
  $\theta$ is set to be the vector $(\theta_1, \theta_2, \theta_3, \theta_4, \theta_5)$ with
  $\theta_5, \theta_2$, and with $\theta_3$ being defined through
  \eqref{eq:2s-theta5}--\eqref{eq:2s-theta3}.  Observe that
  \begin{align*}
   &\mu_1 \bar{\xi}_1(\theta) = \mu_1( -\theta_1+\theta_2) = - \mu_5\theta_5
     = -\frac{1}{m_4 -m_5}[ m_4\theta_1 - \theta_4] = - \langle \theta_L, R^{(1)}\rangle, \\
   & \mu_4\bar{\xi}_4(\theta) = \mu_4(-\theta_4+\theta_5) = -\frac{1}{m_4-m_5}[-m_5\theta_1  + \theta_4]=-\langle \theta_L, R^{(4)}\rangle,
  \end{align*}
  where the $2\times 2$ matrix $R$ is given in \eq{SRBM-data} and $R^{(1)}$ and $R^{(4)}$ are the $1$st and $4$th columns of $R$, respectively.
Also,
  \begin{align}
\label{eq:2s-Sigma}
    &  \tilde{\eta}_1(\theta_1)+  \sum_{k=1}^5 \tilde{\xi_k}(\theta) =\frac{1}{2}\Big( \theta_1^2 +\sum_{k=1}^4 (-\theta_k+\theta_{k+1})^2 + \theta_5^2 \Big) 
      = \Sigma_{11}\theta_1^2 + 2 \Sigma_{14}\theta_1\theta_4 + \Sigma_{44}\theta_4^2, 
  \end{align}
  where the second equality follows from \lem{2s-q} below and $\Sigma_{11}$,
  $\Sigma_{14}$ and $\Sigma_{44}$ are given by
  \eq{Sigma11}-\eq{Sigma44} with $c^2_{e,1}=c^2_{s,k}=1$ for
  $k=1, \ldots, 5$; the latter is true  because the inter-arrival and service-time
  distributions are exponential. Therefore, \eq{2s-bar-2} is reduced to
  \begin{align}
    &\Big( \Sigma_{11}\theta_1^2 + 2 \Sigma_{14}\theta_1\theta_4 + \Sigma_{44}\theta_4^2\Big) \phi^{*}(\theta_{L},0) 
+ \langle \theta_L, R^{(1)} \rangle \big( \phi^{*}_1(\theta_{L},0)-\phi^{*}(\theta_{L},0)\big) + {} \nonumber \\
 & \quad {} + \langle \theta_L, R^{(4)} \rangle \big( \phi^{*}_4(\theta_{L},0)-\phi^{*}(\theta_{L},0)\big) \quad \text{ for each }
 \theta_{L}=(\theta_1, \theta_4) \in \Theta_{L}. \label{eq:2s-bar-3}
  \end{align}
Since  $R$ in \eq{SRBM-data} is an $\mathcal{M}$ matrix,
 it follows from Proposition 5.1 and the proof of (6.3) and (6.4) in \cite{BravDaiMiya2017} that
  $\phi^*(0-, 0, 0, 0-, 0)=1$, $\phi^*_1(0, 0, 0, 0-, 0)=1$, and
  $\phi^*_4(0-, 0, 0, 0, 0)=1$, where
    \begin{align*}
      &      \phi^*(0-, 0, 0, 0-, 0)=\lim_{\theta_1\uparrow 0, \theta_4\uparrow 0}\phi^*(\theta_1, 0, 0, \theta_4,0), \\
      & \phi^*_1(0, 0, 0, 0-, 0)=\lim_{\theta_4\uparrow 0}\phi^*_1(0, 0, 0, \theta_4,0),      \\
      & \phi^*_4(0, 0, 0, 0-, 0)=\lim_{\theta_1\uparrow 0}\phi^*_4(\theta_1, 0, 0, 0,0).     \end{align*}
    It follows that $\phi^*(\theta_1, 0, 0, \theta_4,0)$ is the  Laplace transform  of a
    probability measure $\nu$ on $\R^2_+$,
    $\phi^*_1(0, 0, 0, \theta_4,0)$ is the Laplace transform of a probability measure
    $\nu_1$ on $\R_+$, and $\phi^*_4(\theta_1, 0, 0, 0,0)$ is the Laplace transform of a
    probability measure $\nu_4$ on $\R_+$, namely
    \begin{align*}
      & \phi^*(\theta_1, 0, 0, \theta_4, 0)= \int_{\R_+^2} e^{\theta_1 x_1 + \theta_4 x_4} d \nu(x_1, x_4) \text{ for } (\theta_1, \theta_4)<0 \\
      &\phi^*_1(0, 0, 0, \theta_4,0) = \int_{\R_+} e^{\theta_4 x_4} d\nu_1(x_4) \text{ for }  \theta_4<0,  \quad
     \phi^*_4(\theta_1, 0, 0, 0 ,0) = \int_{\R_+} e^{\theta_1 x_1} d\nu_4(x_1) \text{ for } \theta_1<0;   
    \end{align*}
    see, for example, Lemma 6.1 of \cite{BravDaiMiya2017} for an
    argument.  Furthermore, it follows from \lem{rbm-bar} in \sectn{srbm} that
    the probability measures $\nu$, $\nu_1$, and $\nu_4$ are unique. Let
    $(Z^*_1, Z^*_4)$ be a random vector that has the distribution of
    $\nu$. Then,
    \begin{align*} 
 \phi^*(\theta_1, \theta_2, \theta_3, \theta_4, \theta_5)=   \phi^*(\theta_1, 0, 0, \theta_4, 0) = \E\big[ e^{\theta_1 Z_1^*+ \theta_4 Z_4^*}\big] \quad \text{ for any } \quad \theta\in \R^5_-,
    \end{align*}
    where the first equality follows the Laplace version of SSC  \eq{2s-ssc-mgf}.
    The uniqueness of the probability measure $\nu$ proves
    \eq{2s-p1}.
\end{proof}

Thus, the proof of \pro{2s} is completed by \lem{2s-q} below, which computes $\Sigma_{i,j}$ for $i,j=1,4$.

\begin{lemma}\label{lem:2s-q}
  For each $(\theta_1, \theta_4)\in \R^2$, let $\theta_2$, $\theta_3$, and $\theta_5$ be defined through \eq{2s-theta5}--\eq{2s-theta3}. Then,  the   quadratic equation
  \begin{align}
    \frac 12 \left( c_{e,1}^{2} \theta_{1}^{2} + \sum_{k=1}^{4} c_{s,k}^{2} (-\theta_{k} + \theta_{k+1})^{2} + c_{s,5}^{2} \theta_{5}^{2} \right) 
          = \Sigma_{11}\theta_1^2 + 2 \Sigma_{14}\theta_1\theta_4 + \Sigma_{44}\theta_4^2 \label{eq:2s-q}
  \end{align}
holds for each $(\theta_1, \theta_4)\in \R^2$ if and only if $\Sigma_{11}$, $\Sigma_{14}$, and $\Sigma_{44}$ are given by \eq{Sigma11}-\eq{Sigma44}.
\end{lemma}
\begin{proof}
  Since $\theta_{5} = (\theta_{1} - \mu_{4} \theta_{4})/(\mu_{5} - \mu_{4})$ and $(m_{1}+m_{3})\mu_{5} = \mu_{5} - 1$ by $m_{1}+m_{3}+m_{5} =1$, quadratic terms in
  the left side of \eq{2s-q} are computed as
\begin{align*}
 (\theta_{2} - \theta_{1})^{2} & = m_{1}^{2} \mu_{5}^{2} \theta_{5}^{2} = \frac{m_{1}^{2} \mu_{5}^{2}} {(\mu_{5} - \mu_{4})^{2}} (\theta_{1} - \mu_{4} \theta_{4})^{2},\\
 (\theta_{3} - \theta_{2})^{2} & = (\theta_{4} - \theta_{1} + (m_{1}+m_{3})\mu_{5} \theta_{5})^{2} = (\theta_{4} - \theta_{1} + (\mu_{5} - 1) \theta_{5})^{2} \\
 & = \left(\theta_{4} - \theta_{1} + \frac{\mu_{5} - 1} {\mu_{5} - \mu_{4}} (\theta_{1} - \mu_{4} \theta_{4}) \right)^{2} = \frac{(\mu_{4} - 1)^{2} } {(\mu_{5} - \mu_{4})^{2}} (\theta_{1} - \mu_{5} \theta_{4})^{2},\\
 (\theta_{4} - \theta_{3})^{2} & = \frac{m_{3}^{2} \mu_{5}^{2}} {(\mu_{5} - \mu_{4})^{2}} (\theta_{1} - \mu_{4} \theta_{4})^{2},\\
(\theta_{5} - \theta_{4})^{2} & = \left(\frac{\theta_{1} - \mu_{4} \theta_{4}} {\mu_{5} - \mu_{4}} - \theta_{4}\right)^{2} = \frac{1} {(\mu_5 -\mu_{4})^{2}} (\theta_{1} - \mu_{5} \theta_{4})^{2},\\
 \theta_{5}^{2} & = \frac {1}{(\mu_{5} - \mu_{4})^{2}} (\theta_{1} - \mu_{4} \theta_{4})^{2}.
\end{align*}
Hence, collecting the coefficients of $\theta_{1}^{2}$, we have
\begin{align*}
  \Sigma_{11} & = \frac 12\Bigg(c_{e,1}^{2} + c_{s,1}^{2} \frac{m_{1}^{2} \mu_{5}^{2}} {(\mu_{5} - \mu_{4})^{2}} + c_{s,2}^{2} \left(\frac{\mu_{4} - 1} {\mu_{5} - \mu_{4}} \right)^{2} \\
 & \qquad + c_{s,3}^{2} \frac{m_{3}^{2} \mu_{5}^{2}} {(\mu_{5} - \mu_{4})^{2}} + c_{s,4}^{2} \frac{1} {(\mu_5 -\mu_{4})^{2}} + c_{s,5}^{2} \frac{1} {(\mu_5 -\mu_{4})^{2}} \Bigg)\\
  & = \frac{1} {2(\mu_5 -\mu_{4})^{2}} \left( (\mu_5 -\mu_{4})^{2} c_{e,1}^{2} + m_{1}^{2} \mu_{5}^{2} c_{s,1}^{2} + (\mu_{4} - 1)^{2} c_{s,2}^{2} + m_{3}^{2} \mu_{5}^{2} c_{s,3}^{2} + c_{s,4}^{2} + c_{s,5}^{2} \right).
\end{align*}
\eq{Sigma14} and \eq{Sigma44} are similarly obtained.
\end{proof}



\subsection{General bounded distributions}
\label{sec:bounded-d}

In this section, we prove \pro{2s} when  inter-arrival and service-time distributions
are general.  To keep our notational system simple, we further assume
these distributions have bounded supports.  The bounded support
assumption will be replaced with a moment condition in
Sections \sect{proof-main} and \sect{deriving-BAR}.

Define
\begin{align}
  & \tilde{\eta}_1(\theta_1)=\frac{1}{2} c_{e,1}^2 \theta_1^2, \quad
    \tilde{\xi}_k(\theta) = \frac{1}{2} c_{s,k}^2(\theta_{k+1}-\theta_k)^2 \quad k\in\{1,\ldots, 4\},\quad
   \tilde {\xi}_5(\theta)= \frac{1}{2}c_{s,5}^2\theta_5^2, \label{eq:2s-tildeetaxi-g}
\end{align}
where $c^2_{e,1}$ is the SCV of the inter-arrival time distribution, and $c^2_{s,k}$
is the SCV of the class $k$ service-time distribution.

The main purpose of this section is to prove the following lemma.
\begin{lemma}\label{lem:2s-abar-g}
  Assume that inter-arrival and service-time distributions have
  bounded supports. Then Lemma~\ref{lem:2s-abar-2} continues to hold with $\tilde \eta_1(\theta_1)$ and
  $\tilde{\xi}_k(\theta)$    defined in \eq{2s-tildeetaxi-g}  and
  $\bar{\eta}_1(\theta_1)$ and  $\bar{\xi}_k(\theta)$ defined in
  \eq{2s-bareta} and \eq{2s-barxi}.
\end{lemma}
Once \lem{2s-abar-g} is proved, the remaining steps in the proof of \pro{2s} for the general distribution case are the same as for the exponential case.
To prove \lem{2s-abar-g},
we define, for each $\theta\in \R^5$,  $\eta_1(\theta_1)$ and $\xi_k(\theta)$ as the solutions to
\begin{align}
  & e^{\theta_1} \E\Big(e^{-\eta_1(\theta_1)T_{e,1}}\Big)=1,\label{eq:2s-eta-g} \\
  & e^{-\theta_k+\theta_{k+1}} \E\Big(e^{-\xi_k(\theta)T_{s,k}}\Big)=1, \quad k\in\{1, \ldots, 4\}, \label{eq:2s-xi-g} \\
  & e^{-\theta_5} \E\Big(e^{-\xi_5(\theta)T_{s,5}}\Big)=1. \label{eq:2s-xi5-g}
\end{align}
We intentionally reuse the notation $\eta_1(\theta)$ and $\xi_k(\theta)$ from \eq{2s-etaxi}.  This causes no
harm because these two sets of definitions are identical when
$T_{e,1}$ and $T_{s,k}$ are exponentially distributed.
It is proved in \cite{BravDaiMiya2017} that when $T_{e,1}$ and $T_{s,k}$ have bounded support, then $\eta_1(\theta_1)$ and $\xi_k(\theta)$ are well defined for each $\theta\in \R^5_-$.
Furthermore, the
following lemma holds.
\begin{lemma}\label{lem:2s-taylor-g}
  Taylor expansions \eq{2s-taylor-1}--\eq{2s-taylor-2} continue to hold with $\tilde \eta_1(\theta_1)$ and $\tilde \xi_k(\theta)$ defined in \eq{2s-tildeetaxi-g}.
\end{lemma}

Recall that we assume the sequence of two-station, five-class networks has
arrival rates $\lambda^{(r)}_1=1-r$ and mean service times satisfying
\eq{2s-m} and \eq{2s-mv}. Let $\kappa>0$ be the constant such that
the support of each distribution is contained in the interval $[0, \kappa]$.
Recall that
$X^{(r)}$ 
is the random vector
representing the unique stationary distribution on
$S\equiv \Z^5_+\times [0,\kappa]^6$ of the corresponding Markov process. In
the following, we use
\begin{align*} 
x\equiv (z_1, \ldots, z_5, u_1, v_1, \ldots, v_5)\in S
\end{align*}
to denote a generic state. For each $\theta\in \R^5_-$, define
\begin{align}
  f_\theta(x)= \exp\Big({\langle \theta, z\rangle}\Big) \exp\Big({- \lambda_1  \eta_1(\theta_1)u_1 -\sum_{k=1}^5 \mu_k \xi_k(\theta) v_k}\Big), \quad x\in S.\label{eq:f-theta}
\end{align}
For each fixed $\theta\in\R^5_-$, it is clear that $f_{\theta}(x)$ is a bounded function of  $x\in S$.
For each $\theta\in \R^5_-$, define
\begin{align*}
  & \psi^{(r)}(\theta)=\E[f_{r\theta}(X^{(r)})], \quad    \psi^{(r)}_1(\theta) = \E[f_{r\theta}(X^{(r)}) \mid Z^{(r)}_1=0, Z^{(r)}_3=0, Z^{(r)}_5=0],\\
  &  \psi^{(r)}_3(\theta) = \E[f_{r\theta}(X^{(r)}) \mid  Z^{(r)}_3=0, Z^{(r)}_5=0], \quad  \psi^{(r)}_5(\theta) = \E[f_{r\theta}(X^{(r)}) \mid Z^{(r)}_5=0], \\
  & \psi^{(r)}_2(\theta) = \E[f_{r\theta}(X^{(r)}) \mid Z^{(r)}_2=0],
      \quad  \psi^{(r)}_4(\theta) = \E[f_{r\theta}(X^{(r)}) \mid Z^{(r)}_2=0, Z^{(r)}_4=0]. 
\end{align*}

\lem{2s-abar-g} follows immediately from the following two
lemmas. 
\begin{lemma}\label{lem:2s-abar-g2}
 For each $\theta_L\in \Theta_L$, let $\theta=(\theta_L, \theta_H)$ be the unique
  $\theta$ that satisfies \eq{s2-f1}--\eq{s2-f3}. 
  Then, any limit point $(\psi^*(\theta), \psi^*_1(\theta), \ldots, \psi^*_5(\theta))$ of $\{(\psi^{(r)}(\theta), \psi^{(r)}_{1}(\theta), \ldots, \psi^{(r)}_{5}(\theta))\}_{r \in (0,1)}$ satisfies
\begin{align*}
   \Big( \tilde{\eta}_1(\theta_1)+  \sum_{k=1}^5  \tilde{\xi_k}(\theta)\Big) \psi^{*}(\theta)    + \mu_1 \bar{\xi}_1(\theta)  \big(\psi^{*}(\theta)- \psi^{*}_1(\theta)\big) 
  + \mu_4  \bar{\xi}_4(\theta) \big(\psi^{*}(\theta)- \psi^{*}_4(\theta)\big)  = 0,  
\end{align*}  
where $\tilde \eta_1(\theta_1)$ and $\tilde{\xi}_k(\theta)$ are defined
in \eq{2s-tildeetaxi-g} and $\bar{\xi}_k(\theta)$ is defined in \eq{2s-barxi}.
\end{lemma}
This lemma is similar to \lem{2s-abar-2}, with $\psi^{(r)}(\theta)$  replacing
$\phi^{(r)}(\theta)$. The proof of  \lem{2s-abar-g2} will be at the
end of this section after we introduce the BAR for $X^{(r)}$.
 The following lemma allows one to derive \lem{2s-abar-2}
from \lem{2s-abar-g2} immediately. 

\begin{lemma}\label{lem:ssc-phipsi}
  For each $\theta\in \R^5_-$, as $r\to 0$,
  \begin{align*}
    & \phi^{(r)}(\theta) -     \psi^{(r)}(\theta) =o(1) \quad \text{ and } \quad  \phi^{(r)}_k(\theta) -     \psi^{(r)}_k(\theta)=o(1).
  \end{align*}
\end{lemma}
\begin{proof} Fix a $\theta\in \R^5_-$.  For each $x\in S$ and each $r\in (0,1)$,
  \begin{align*}
     \abs{g_{r\theta}(z) - f_{r\theta}(x)} & = g_{r\theta}(z) \babs{1-  \exp\big({- \lambda^{(r)}_1  \eta_1(r\theta_1)u_1 -\sum_{k=1}^5 \mu_k \xi_k(r\theta) v_k}\big)} \\
    & \le e^{\abs{  \Lambda(r,\theta,  u_1, v)}} \abs{  \Lambda(r,\theta,  u_1, v)}
      \le e^{\Lambda(r\theta)} \Lambda(r\theta),
  \end{align*}
  where
\begin{align*}
  & \Lambda(r,\theta,  u_1, v) \equiv - \lambda^{(r)}_1  \eta_1(r\theta_1)u_1 -\sum_{k=1}^5 \mu_k \xi_k(r\theta) v_k, \\
& \Lambda(\theta)     \equiv  \kappa  \abs{\eta_1(\theta_1)} + \sum_{k=1}^5 \kappa \mu_k \abs{\xi_k(\theta)}.
\end{align*}
Therefore,  
  \begin{align*}
    & \abs{\phi^{(r)}(\theta) -     \psi^{(r)}(\theta)} \le
      e^{\Lambda(r\theta)} \Lambda(r\theta)  \quad \text{ and } \quad  \abs{\phi^{(r)}_k(\theta) -     \psi^{(r)}_k(\theta)} \le
         e^{\Lambda(r\theta)} \Lambda(r\theta).    
  \end{align*}
  It follows from \lem{2s-taylor-g} that
  \begin{align*}
    \lim_{r\to 0} \eta_1(r\theta_1)=0 \text{ and } \lim_{r\to 0} \xi_k(r\theta)=0,
  \end{align*}
  which implies that $\lim_{r\to 0}\Lambda(r\theta)=0$ and the lemma is proved.
\end{proof}

\paragraph{Basic adjoint relationship (BAR).}
This section proves that for each $\theta\in \R^5_-$,
\begin{align}
   r^2\Big(\lambda^{(r)}_1 \tilde{\eta}_1(\theta_1) & +  \sum_{k=1}^5 \lambda^{(r)}_1\tilde{\xi_k}(\theta)\Big) \psi^{(r)}(\theta)   + r^2 \mu_1 \xi^*_1(\theta)  \big(\psi^{(r)}(\theta)- \psi^{(r)}_1(\theta)\big) \nonumber  \\
  & + \mu_4 r^2 \xi^*_4(\theta) \big(\psi^{(r)}(\theta)- \psi^{(r)}_4(\theta)\big)
     + [\mu_3\xi^*_3(r\theta) -
     \mu_1\xi^*_1(r\theta)]
    \beta^{(r)}_3 \big(\psi^{(r)}(\theta)-\psi^{(r)}_3(\theta)\big)
\nonumber \\
   &  + [\mu_5\xi^*_5(r\theta) -
     \mu_3\xi^*_3(r\theta)]
    \beta^{(r)}_5 \big(\psi^{(r)}(\theta)-\psi^{(r)}_5(\theta)\big)
 \nonumber \\
  & {}    + 
    [\mu_2\xi^*_2(r\theta) -
     \mu_4\xi^*_4(r\theta)]
    \beta^{(r)}_2 \big(\psi^{(r)}(\theta)-\psi^{(r)}_2(\theta)\big)=o(r^2),\label{eq:2s-abar-g}
\end{align}
where $\tilde{\eta}_1(\theta_1)$ and $\tilde{\xi}_k(\theta)$ are defined in \eq{2s-tildeetaxi-g} and $\xi^*_k(\theta)$ retains the definition in \eq{2s-star}.
Equation \eq{2s-abar-g} is analogous to \eq{2s-abar} in the exponential case.
In the exponential case, \eq{2s-abar} leads to the proof of \lem{2s-abar-2}.
Copying exactly the same proof, one can readily prove \lem{2s-abar-g2} from
\eq{2s-abar-g},   thereby proving \pro{2s} for the general distribution case.

In the remainder of this section we   prove \eq{2s-abar-g}.
In the following, we drop the superscript $(r)$ everywhere to focus on one
queueing network within the family of queueing networks. Recall that 
\begin{align*}
  X=  (Z, U_1, V), \; \mbox{ where } Z = (Z_1, Z_2, Z_3, Z_4, Z_5)  \text{ and }  V= (V_1, V_2, V_3, V_3, V_5),
\end{align*}
is the random vector distributed according to the stationary
distribution of $\{X(t), t\ge 0\}$.  We first develop a Laplace
transform version of the BAR for $X$.
Readers should be aware that
  the rest of this section is a simplified  version of \sectn{bar-Xr}, \sectn{proof-main}, and
\sectn{deriving-BAR}. The simplification comes from the simplified notational system  offered by the two-station, five-class reentrant line  and the bounded support
assumption on inter-arrival and service-time distributions.

Let $\mathcal{D}$ be the set of bounded function $f:S\to \R$
satisfying the following conditions: (a) $f(x)$ is bounded in
$x\equiv (z, u_1, v_1, \ldots, v_5)\in S$. (b) For each fixed
$z\in \Z^5_+$, $f(z, u_1, v_1, \ldots, v_5)$ has partial derivative
from the right in $u_1$ and $v_k$, and these partial derivatives are
bounded. For each $f\in \mathcal{D}$, define ``interior operator''
\begin{align}
  \mathcal{A}f(x) = & - \frac{\partial f}{\partial u_1} (x)
  - \frac{\partial f}{\partial v_1} (x) 1(z_1>0, z_3=0, z_5=0)
  -\frac{\partial f}{\partial v_3} (x) 1(z_3>0, z_5=0) \nonumber \\
&  -\frac{\partial f}{\partial v_5} (x) 1(z_5>0)
  -\frac{\partial f}{\partial v_2} (x) 1(z_2>0)
  -\frac{\partial f}{\partial v_4} (x) 1(z_2=0,z_4>0).\label{eq:2s-A}
\end{align}

We intend to derive a BAR corresponding to \eq{ctmcbar} using this operator. We first note that $\sr{A}f(X(t))$ is the derivative of $f(X(u))$ at $u=t$ when $X(u)$ is continuous at $u=t$. However, $f(X(t))$ may change at jump instants of $X(t)$. Taking this into account, we observe that the total change  of the sample path of $f(X(u))$ from $u=0$ to $u=t$ is
\begin{align}
\label{eq:2s-evolution-f}
  f(X(t)) - f(X(0)) = \int_{0}^{t} \sr{A}f(X(u)) du + \sum_{0 < u \le t} (f(X(u)) - f(X(u-)),
\end{align}
where $X(u-)= \lim_{t\uparrow u} X(t)$ is the left limit of $X(\cdot)$ at $u$, 
and 
there are finitely many $u \in (0,t]$ such that $f(X(u)) \not= f(X(u-))$.
Because $f \in \sr{D}$, the summation in \eq{2s-evolution-f} is well defined.

We fix $t=1$ in  \eq{2s-evolution-f}. Assuming $X(0)$ follows the stationary distribution, $\{X(u), 0\le u\le 1\}$ is a stationary process. 
Taking the expectations in both sides of \eq{2s-evolution-f} and
using
\begin{align*} 
\E\Big[\int_0^1  \sr{A}f(X(u)) du\Big] =
\int_0^1 \E[ \sr{A}f(X(u))] du
\end{align*}
due to the boundedness of $\sr{A}f(X(u))$,
we have
\begin{align}
\label{eq:2s-Ef}
  \dd{E}[\sr{A}f(X)] + \dd{E}\Big[\sum_{0 < u \le 1} (f(X(u)) - f(X(u-))\Big] = 0,
\end{align}
where all the expectations are well defined because $f$ and its partial derivatives are bounded.

By our convention, $X(u)$ is right continuous. Therefore, $U_1(u)>0$
and $V_k(u)>0$ for each $u>0$. When $f(X(u)) \not= f(X(u-))$ at $u>0$, 
at least one of the following events happens: 
\begin{itemize}
\item [(a)] an external arrival
occurs at $u$, which is equivalent to $U_1(u-)=0$, or,  
\item [(b)] a service completion occurs at class $k$, which is equivalent to $V_k(u-)=0$,
$k\in \{1, 2, 3, 4, 5\}$.
\end{itemize}

For evaluating the second expectation in \eq{2s-Ef},  we separate different event types and define probability distributions $\dd{P}_{e,1}$ and $\dd{P}_{s,k}$ for $k \in \sr{K} \equiv \{1,2,\ldots,5\}$ on $S^{2}$ as, for $B \in \sr{B}(S^{2})$,
\begin{align}
\label{eq:2s-palm-de}
 & \dd{P}_{e,1}[B] = \frac{1}{\lambda_1} \E\Big[ \sum_{0 < u \le 1} 1((X(u-),X(t)) \in B) 1(U_{1}(u-) =0) \Big],\\
\label{eq:2s-palm-ds}
 & \dd{P}_{s,k}[B] = \frac{1}{\lambda_1} \E\Big[\sum_{0 < u \le 1} 1((X(u-),X(t)) \in B) 1(V_{k}(u-) =0) \Big],\qquad k \in \mathcal{K},
\end{align}
where $\dd{P}_{e,k}$ and $\dd{P}_{s,k}$ are indeed probability distributions because $\dd{E}[\sum_{0 < u \le 1} 1(U_{1}(u-) =0)]$ and $\dd{E}[\sum_{0 < u \le 1} 1(V_{k}(u-) =0)]$ are the mean arrival rate of exogenous customers at station $1$ and the mean departure rate at station $k$, respectively, and both of them are $\lambda_{1}$; {see \lem{rate} for a proof.
We call these distributions Palm distributions concerning the exogenous arrivals at station $1$ and the departures from class $k$.

Denote an identity function from $S^{2}$ to $S^{2}$ by $(X_-, X_+)$; then it can be considered as a pair of random variables taking values in $S^2$ on the measurable space $(S^{2}, \sr{B}(S^{2}))$. We consider it on the probability spaces $(S^{2}, \sr{B}(S^{2}), \dd{P}_{e,1})$ and $(S^{2}, \sr{B}(S^{2}), \dd{P}_{s,k})$. For $f \in \sr{D}$, let
\begin{align*}
  \Delta f(X_{-},X_{+}) = f(X_{+}) - f(X_{-}),
\end{align*}
then we have
\begin{align*}
 & \dd{E}_{e,1}[\Delta f(X_{-},X_{+})] = \frac{1}{\lambda_1} \E\Big[ \sum_{0 < u \le 1} (f(X(u-)) - f(X(t))) 1(U_{1}(u-) =0) \Big],\\
 & \dd{E}_{s,k}[\Delta f(X_{-},X_{+})] = \frac{1}{\lambda_1} \E\Big[\sum_{0 < u \le 1} (f(X(u-)) - f(X(t))) 1(V_{k}(u-) =0) \Big],\qquad k \in \{1, \ldots, 5\},
\end{align*}
where $\E_{e,1}$ is the expectation under the Palm distribution
$\Prob_{e,1}$, and $\E_{s,k}$ is the expectation under the Palm distribution
$\Prob_{s,k}$.

Substituting these formulas into \eq{2s-Ef}, we have the following lemma.

\begin{lemma}
The random vectors $X$ and $(X_-, X_+)$ satisfy the following BAR: for each $f\in \mathcal{D}$,
\begin{align}
  \dd{E}[ \mathcal{A}f(X)] + \lambda_1 \dd{E}_{e,1}[\Delta f(X_{+}, X_-)]  + \sum_{k=1}^5 \lambda_1 \dd{E}_{s,k}[\Delta f(X_{+},X_-)] = 0. \label{eq:2s-bar-g}
\end{align}
\end{lemma}

Applying \eq{2s-bar-g}, it remains to evaluate expectations under the Palm distributions. From the definitions, \eq{2s-palm-de} and \eq{2s-palm-ds}, one can see that $X_{-}$ represents the network state just before its jump  instants under the Palm distributions, and $X_{+}$ does so just after the jump instants. More specifically, one can intuitively see that 
\begin{align}
\label{eq:2s-X+e}
 & X_{+} = X_{-} + \big(e^{(1)}, \frac{1}{\lambda_1} T_{e,1}, 0\big), \qquad \mbox{under } \dd{P}_{e,1},\\
\label{eq:2s-X+s}
 & X_{+} = X_{-} + \big(-e^{(k)} + e^{(k+1)} 1(k \le 4)), 0,  m_{k} T_{s,k}e^{(k)}\big), \qquad \mbox{under } \dd{P}_{s,k}, k \in \sr{K},
\end{align}
where, under $\dd{P}_{e,1}$, $T_{e,1}$ is independent of $X_{-}$ and has the same distribution $T_{e,1}(1)$ under $\dd{P}$, and, under $\dd{P}_{s,k}$, $T_{s,k}$ is independent of $X_{-}$ and has the same distribution $T_{e,1}(1)$ under $\dd{P}$. This representation is formally proved for the general multiclass network with SBP service discipline in \lem{X+}.
Note that $X$ and $(X_-, X_+)$ are defined on different probability spaces. Random vector $X$ under $\dd{P}$ follows the stationary distribution of the Markov process. 

Fix a $\theta\in \R^5_-$. For  the $f_{\theta}$ in \eq{f-theta}, one can check that $f_\theta\in \mathcal{D}$ and 
\begin{align}
  \mathcal{A}f_{\theta}(x)  =&  f_{\theta}(x) \lambda_1\eta_1(\theta_1)+  \mu_1 \xi_1(\theta) f_{\theta}(x)(z_1>0, z_3=0, z_5=0 ) \nonumber \\
  & {} + \mu_3 \xi_3(\theta) f_{\theta}(x)(z_3>0, z_5=0 )
 + \mu_5 \xi_5(\theta) f_{\theta}(x)(z_5>0 ) \nonumber \\
  & { }+ \mu_2 \xi_2(\theta) f_{\theta}(x)(z_2>0)+
  \mu_4 \xi_4(\theta) f_{\theta}(x)(z_2=0, z_4=0).\label{eq:2s-A-g}
\end{align}
Setting $f=f_\theta$, it follows from \eq{2s-X+e}, \eq{2s-X+s} and \eq{2s-eta-g}--\eq{2s-xi5-g} that
\begin{align}\label{eq:2s-Delta}
\dd{E}_{e,1}[f(X_{+}) - f(X_{-})] = 0, \qquad \dd{E}_{s,k}[f(X_{+}) - f(X_{-})] = 0, \quad k = 1,2,\ldots,5.
\end{align}
Hence, \eq{2s-bar-g} becomes that $\dd{E}[ \mathcal{A}f(X)] = 0$. Define
\begin{align}
  & \psi(\theta)=\E[f_\theta(X)], \quad    \psi_1(\theta) = \E[f_\theta(X) \mid Z_1=0, Z_3=0, Z_5=0], \label{eq:2spsi} \\
  &  \psi_3(\theta) = \E[f_\theta(X) \mid  Z_3=0, Z_5=0], \quad  \psi_5(\theta) = \E[f_\theta(X) \mid Z_5=0], \label{eq:2spsi3}\\
  & \psi_2(\theta) = \E[f_\theta(X) \mid Z_2=0],
      \quad  \psi_4(\theta) = \E[f_\theta(X) \mid Z_2=0, Z_4=0]. \label{eq:2spsi4}
\end{align}
Then, it follows from \eq{2s-A-g} that 
\begin{align}
 &  \lambda_1\eta_1(\theta_1) \E[f_{\theta}(X)] +  \mu_1 \xi_1(\theta) \E[f_{\theta}(X)(Z_1>0, Z_3=0, Z_5=0 )\Big] \nonumber \\
  & {} + \mu_3 \xi_3(\theta) \E[f_{\theta}(X)(Z_3>0, Z_5=0 ]
 + \mu_5 \xi_5(\theta) \E[f_{\theta}(X)(Z_5>0 )] \nonumber \\
  & { }+ \mu_2 \xi_2(\theta) \E[ f_{\theta}(X)(Z_2>0)]+
  \mu_4 \xi_4(\theta)\E[ f_{\theta}(X)(Z_2=0, Z_4>0)]=0,\label{eq:2s-bar-g1}
\end{align}
which is analogous to \eq{2s-bar-exp} for the exponential case.
Identical to the derivation of \eq{2sbar_exp} from \eq{2s-bar-exp}, one has 
\begin{align}
  & \Big(\lambda_1 \eta_1(\theta_1)+  \sum_{k=1}^5 \lambda_1 \xi_k(\theta)\Big) \psi(\theta) +
     [\mu_5\xi_5(\theta) -
     \mu_3\xi_3(\theta)]
    \beta_5 \big(\psi(\theta)-\psi_5(\theta)\big)  \nonumber \\
  & {} + [\mu_3\xi_3(\theta) -
     \mu_1\xi_1(\theta)]
    \beta_3 \big(\psi(\theta)-\psi_3(\theta)\big) + \mu_1 \xi_1(\theta) (1-\rho_1) \big(\psi(\theta)- \psi_1(\theta)\big) \nonumber \\
    & + \mu_4 \xi_4(\theta) (1-\rho_2) \big(\psi(\theta)- \psi_4(\theta)\big) +  [\mu_2\xi_2(\theta) -
     \mu_4\xi_4(\theta)]
    \beta_2 \big(\psi(\theta)-\psi_2(\theta)\big)=0,\label{eq:2sbar_g}
\end{align}
where the $\beta_k$ is the steady-state probability that all classes with priority greater or equal to class $k$ have no customers, as  defined in
\eq{2s_idle1}--\eq{2s_idle3}.
Finally, \eq{2s-abar-g} follows from
 \lem{2s-taylor-g} and \eq{2sbar_g} with $\theta$ being replaced by $r\theta$ and $\psi(r\theta)$ being replaced by $\psi^{(r)}(\theta)$.

\section{Multiclass queueing  networks}
\label{sec:multiclass}
\setnewcounter

In this section, we introduce multiclass queueing networks that
operate under SBP service disciplines. Our
terminology and notation follow~\citet{BramDai2001} closely.  In a multiclass queueing network, there
are $J$ service stations that process $K$ classes of jobs, 
where $J$ and $K$ are positive integers such that $J < K$.
(When $K=J$, our multiclass queueing networks become generalized Jackson networks,  which were studied in \cite{BravDaiMiya2017}.)
Denote
\begin{align*}
  \sr{J} = \{1,2,\ldots, J\} \quad \text{ and }\quad  \sr{K} = \{1,2,\ldots, K\}.
\end{align*}
Each station is assumed to have a single server  with unlimited waiting space. When a job arrives from outside the network, it receives service at a finite number of stations sequentially, after which it leaves the network. At any given time during its lifetime in the network, the job belongs to one of  the job \emph{classes}. The  moves through the network, changing classes each time a service is completed; all jobs within a class are served at a unique station.
Each job is assumed to eventually leave the network. The ordered sequence of classes that a job visits in the network is called its route; if all jobs follow the same route, the network is called a \emph{reentrant line}. An example of a reentrant line is depicted in Figure~\ref{fig:2s5c}.

Stations are
labeled $j\in \mathcal{J}$, and classes are labeled $k\in \mathcal{K}$.  We use
$\mathcal{C}(j)$ to denote the set of classes belonging to station $j$,
and $s(k)$ to denote the station to which class $k$ belongs.
Associated with each class $k$ of a queueing network  are
two independent and identically distributed (i.i.d.) sequences of random variables, $T_{e,k}(\cdot)=\{T_{e,k}(i),\; i\ge
1\}$ and $T_{s,k}(\cdot)=\{T_{s,k}(i),\; i\ge 1\}$,
 one i.i.d.\ sequence of  $\R^K$-valued random vectors
$\Phi^{(k)}(\cdot)=\{\Phi^{(k)}(i), i\ge 1\}$, 
and two real numbers, $a_{k}\ge 0$ and $m_{k}>0$. 

We assume
that the $3K$ sequences
\begin{equation}\label{eq2.1}
 T_{e,1}(\cdot),\dots\, ,T_{e,K}(\cdot), \, T_{s,1}(\cdot),\dots\, ,T_{s,K}(\cdot), \Phi^{(1)}(\cdot), \dots\, , \Phi^{(K)}(\cdot)
\end{equation}  
are defined on a common probability space $(\Omega, \mathcal{F}, \dd{P})$ and  are mutually independent.
We use $T_{e,k}$, $T_{s,k}$,  and $\Phi^{(k)}$  to denote  generic random element in sequences $T_{e,k}(\cdot)$, $T_{s,k}(\cdot)$, and  $\Phi^{(k)}(\cdot)$, respectively.
We assume that
$T_{e,k}$ and $T_{s,k}$ are
\emph{unitized}; i.e., $\dd{E}[T_{e,k}]=1$ and
$\dd{E}[T_{s,k}]=1$,   and
  $\Phi^{(k)}$ takes values in $\{e^{(0)}, e^{(\ell)}, \ell\in \mathcal{K}\}$,
where $e^{(\ell)}$ is the $K$-vector with component $\ell$ being $1$
and all other components being $0$, and $e^{(0)}$ is the $K$-vector of
zeros.   For each $i$, $a_{k} T_{e,k}(i)$ will denote the inter-arrival
time between the $(i-1)$th and the $i$th \emph{externally} arriving
job at class $k$, $m_{k}T_{s,k}(i)$ will denote the \emph{service}
time for the $i$th class $k$ job, and  $\Phi^{(k)}(i)=e^{(\ell)}$ means
that the job that completes the $i$th class $k$ service will join next
as a class $\ell$ job  for $\ell\in \mathcal{K}$ or will exit the network when $\ell=0$.

Let $\lambda_{k} = 1/a_{k}$ and $\mu_{k} = 1/m_{k}$ for each class $k$. Then $\lambda_{k}$ is the external arrival rate to class $k$, and $m_{k}$ is the mean service time for class $k$ jobs.  We allow
$\lambda_{k}=0$ for some classes $k$, in which case class $k$ has no external arrivals.
We set
\begin{align*}
  \mathcal{E}=\{k\in \sr{K}:\lambda_{k}\ne 0\}, \qquad E=\abs{\mathcal{E}}, 
\end{align*}
where $|A|$ denotes the number of elements of a set $A$. We assume  that
there exists a $\delta_{0}>0$ such that 
\begin{align}
  \label{eq:3m}
\dd{E}[T_{e,k}^{2+\delta_{0}}] <\infty, \quad k \in \sr{E}, \quad \text{ and } \quad \dd{E}[T_{s,k}^{2+\delta_{0}}] <\infty \quad k\in \sr{K},
\end{align}
and set 
\begin{align*}
  c^2_{e,k}=\Var (T_{e,k}), \quad k \in \sr{E}, \qquad c^2_{s,k}=\Var(T_{s,k}), \quad k \in \sr{K}.
\end{align*}
Thus, $c^2_{e,k}$ and $c^2_{s,k}$ are the squared coefficients of variation for inter-arrival and service times.
 Let 
\begin{align}
  \label{eq:routing}
  P_{k\ell}=\Prob\{\Phi^{(k)}=e^{(\ell)}\}, \quad \ell\in \mathcal{K}, \quad
  P_{k0} =\Prob\{\Phi^{(k)}=e^{(0)}\}= 1- \sum_{\ell\in \mathcal{K}}P_{k\ell}.
\end{align} 
 The $K\times K$ matrix $P=(P_{k\ell})$ is the
\emph{routing matrix} of the network.  We assume our networks are
\emph{open}; that is, the matrix $(I-P)$ is invertible  where $I$ denotes the
  identity matrix.
We note that $P_{k0}$ in \eq{routing} is the  probability of a job leaving the network after completing
a class $k$ service.

\paragraph{Service discipline.}
A service discipline dictates the order in which jobs are served at each
station.  A service discipline is said to be \emph{non-idling} if a server is
always active when there are jobs waiting to be served at its station.
In this paper, we restrict our discipline to 
static buffer priority (SBP), which is defined below.
 Under an SBP
discipline, the classes at each station are assigned a fixed ranking.
When the server switches from one job to another, the new job will be taken
from the leading (or longest-waiting) job at the highest-ranking non-empty
class at the server's station.  We assume that the ranking is strict; i.e.,
there is no tie in the ranking.  We also assume that the service discipline
is \emph{preemptive-resume}.  That is, when a job  with a higher rank than the
one currently being served  arrives at the server's station, the service of
the current job is interrupted.  When service of all jobs with higher ranks
is completed, the interrupted service continues from where it left off.

Two SBP disciplines for reentrant lines that have been studied in the
literature are first-buffer first-served (FBFS) and last-buffer first-served
(LBFS).  Under the FBFS discipline, earlier classes along the route are
assigned higher priorities.  Under the LBFS discipline, later classes along
the route are assigned higher priorities.
For the two-station, five-class reentrant line pictured in
Figure~\ref{fig:2s5c}, we have
$\sr{K} = \{1,2,3,4,5\}$, $\sr{E} = \{1\}$}, $\sr{L} = \{1,4\}$ and $\sr{H} = \{2,3,5\}$.

\paragraph{Notation facilitating an SBP discipline.}
For each class $k\in \sr{K}$, denote
\begin{align}\label{eq:Hk}
H(k)  
\end{align}
as the set of classes at
station $s(k)$ whose priorities are at least as high as class $k$. Let
\begin{align}\label{eq:H+}
 H_+(k)=H(k)\setminus \{k\} 
\end{align}
be the set of classes at station $s(k)$
whose priorities are strictly higher than class $k$.  $H_+(k)$ is
empty when class $k$ has the highest priority at station $s(k)$.
Under our preemptive-resume priority discipline, class $k$ jobs are
processed only when there are no class $\ell$ jobs for \emph{all}
$\ell\in H_+(k)$. For each station $j\in \mathcal{J}$, define
$\ell(j)$ to be the lowest-priority class at station $j$, and define
$h(j)$ to be highest-priority class at station $j$. Define
\begin{align}\label{eq:L}
\sr{K}_1=\{h(1), h(2), \ldots, h(J)\}  \quad \text{ and } \quad
\mathcal{L} =\{\ell(1), \ell(2), \ldots, \ell(J)\}
\end{align}
to be the sets of the highest and  lowest classes, respectively, and
\begin{align}\label{eq:H}
\mathcal{H}   = \sr{K}\setminus\{\ell(1), \ldots, \ell(J)\}  
\end{align}
to be 
the set of ``high priority'' classes that \emph{exclude} all the lowest
priority classes. Clearly, $\mathcal{H}\cap \mathcal{L}=\emptyset$, but
$\sr{K}_1\cap \mathcal{L}$ is \emph{not} necessarily empty as there may be stations serving only one job class. For a subset $\sr{C} \subset \sr{K}$, let $|\sr{C}|$ be its cardinality, that is, the number of its elements. Let $H = |\sr{H}|$ and $L = |\sr{L}|$. The latter is also equal to $J$.

For each class $k\in \mathcal{H}$, define
\begin{align}
  \label{eq:k-}
\parbox{.8\textwidth} {$k-$ to be the highest class in $\{\ell \in \sr{K}; s(\ell)=s(k)\} \setminus H(k)$,} \\
  \label{eq:k+}
\parbox{.8\textwidth} {$k+$ to be the lowest class in $H_{+}(k)$, namely, in $H(k) \setminus \{k\}$.}
\end{align}
When $k-$ and $k+$ are undefined, the quantities indexed by them are explained there.

\paragraph{Traffic equations.}
To investigate open multiclass queueing networks, one employs the
solution $\alpha_{\ell}$, $\ell \in \sr{K}$, of the \emph{traffic equations}
\begin{equation}\label{eq:traffic}
\alpha_{k}=\lambda_{k}+\sum_{\ell \in \sr{K}} \alpha_{\ell} P_{\ell k}, \qquad k \in \sr{K},
\end{equation}
or equivalently, in vector form, of $\alpha=\lambda + P^{T}\alpha$.
All vectors in this paper are to be interpreted as column vectors unless we
explicitly state  otherwise. Since the network corresponding to $P$ is
open, the unique solution to \eq{traffic} is
$\alpha=(I-P^{T})^{-1}\lambda$.  The
term $\alpha_{k}$ is referred to as the \emph{nominal total arrival rate}
at class $k$; it depends on both external and internal arrivals.  If, for each
class $k$, there is a long-run average rate of flow into the class that is
equal to the long-run average rate out of that class, this rate will equal
$\alpha_{k}$.

Employing $m$ and $\alpha$, one defines the \emph{traffic intensity}
$\rho_j$ for the $j$th server as
\begin{align}
  \label{eq:rho}
  &\rho_j=\sum_{k\in \mathcal{C} (j)} \gamma_{k}, \quad \text{ where } \\
  & \gamma_{k} = \alpha_{k} m_{k}.\label{eq:gamma}
\end{align}
In vector form, $\rho$ is given by $\rho=CM\alpha$, where $M=\diag (m)$ and
$C$ is the $\sr{J} \times \sr{K}$ \emph{constituency matrix}
\begin{align}
  C_{jk}=
  \begin{cases} 1 & \text{if } j=s(k),\\
    0 & \text{otherwise.}
  \end{cases}
        \label{eq:C}
\end{align}
In our study, it is convenient to replace $\sr{J}$ with $\sr{L}$ using the fact that $\sr{J}$ and $\sr{L}$ have the same cardinality.

(For a
$d$-dimensional vector $x$, $\diag (x)$ denotes the $d\times d$
matrix whose diagonal entries are given by the components of $x$
and all other entries are 0.) When $\rho_j\le 1$, $\rho_j$ is
also referred to as the nominal fraction of time that server $j$
is busy.  In this paper, we are interested in networks in which
$\rho_j$ is close to one for each station $j$.  Such networks
are  said to be ``heavily loaded." The precise meaning of that term will be defined in Section~\ref{sec:main}.

\paragraph{Markov process.}
At time $t \ge 0$, for $k \in \sr{K}$, let $Z_{k}(t)$ be the number
of class $k$ jobs including possibly the one in service, and let $R_{s,k}(t)$ be
the remaining service time of a class $k$ job in service or the
service time of the next class $k$ job if $Z_{k}(t)=0$. For
$k \in \sr{E}$, let $R_{e,k}(t)$ be the remaining time for a class $k$
job to externally arrive. Let $Z(t),R_{e}(t),R_{s}(t)$ be the
random vectors whose $k$th entries are
$Z_{k}(t),R_{e,k}(t),R_{s,k}(t)$, respectively. Let
\begin{align}\label{eq:Xstate}
  X(t) \equiv \big(Z(t),R_{e}(t),R_{s}(t)\big), \qquad t \ge 0.
\end{align}
Let $\dd{R}_{+}^{\sr{E}}$ be the set of all vectors $(u_{k}; k \in \sr{E})$ for $u_{k} \ge 0$, and let $X(\cdot) = \{X(t), t \ge 0\}$. Then, when dropping the bounded supports assumption on the inter-arrival and service-time distributions, 
$X(t)$ has state space
\begin{align*}
  S \equiv \dd{Z}_{+}^{K} \times \dd{R}_{+}^{\sr{E}} \times \dd{R}_{+}^{K},
\end{align*}
and $X(\cdot)$ is a Markov process with respect to the filtration
$\dd{F}^{X} \equiv \{\sr{F}^{X}_{t}; t \ge 0\}$, where
$\sr{F}_{t} = \sigma(\{X(u); 0 \le u \le t\})$. 
We here assume that
$X(\cdot)$ is right continuous on $[0, \infty)$ and has a limit from the
left in $(0,\infty)$. When inter-arrival and service-time distributions  are exponential, because of
the memoryless property of an exponential distribution,
$\{Z(t), t\ge 0\}$ itself is a continuous-time Markov chain with
(discrete) state space $\Z_+^K$.

A distribution $\pi$ on $S$ is said
to be a stationary distribution of the Markov process $X(\cdot)$ if
$X(t)$ follows  distribution $\pi$ for any $t > 0$ when $X(0)$ is
initialized with distribution $\pi$. A necessary condition for the existence of a stationary distribution is 
\begin{align}
  \label{eq:rholess1}
  \rho_j<1, \quad j\in \mathcal{J}.
\end{align}
(See, for example, Theorem 5.2 of \cite{DaiHarr2020} for a proof when all distributions
are  phase type.)
\cite{Dai1995} provides a sufficient condition for the existence of a
stationary distribution. The condition is in terms of the stability of
a fluid model corresponding to the  queueing network.

\section{Basic adjoint relationship of an SRBM}
\label{sec:srbm}
\setnewcounter

In this paper, semimartingale reflecting Brownian motions (SRBMs) are
not used explicitly, but they are in the background somewhat
prominently. For the definition of an SRBM, see, for example, Section
2.3 of \cite{BravDaiMiya2017} or Definition~3.1 in \cite{BramDai2001}.
Recall  that  the queueing network defined in Section~\ref{sec:multiclass}
has $K$ classes and $J$ stations. In the following, $\mathcal{L}$ can
be any subset of $\mathcal{K}$ and $L=\abs{\mathcal{L}}$. To be
specific, the set $\sr{L}\subset \mathcal{K}$ is the lowest classes at
stations in the queueing network.  Therefore, $L=J$ the number of
stations in the network.  We use $\R^{\mathcal{L}}$ to denote the
$L$-dimensional Euclidean space; for a vector $x=(x_\ell)\in \R^{\mathcal{L}}$,
its components $x_\ell$ are indexed by $\ell\in \mathcal{L}$.  Similarly, for
an $\mathcal{L}\times \mathcal{L}$ matrix $A=(A_{ij})$, its entries $A_{ij}$ are indexed by $i,j\in\mathcal{L}$. For a subset $\mathcal{C}\subset \mathcal{L}$,
matrix $(A_{ij}, i,j\in \mathcal{C})$ is called a principal submatrix of $A$.

\begin{definition}[Completely-$\sr{S}$ matrix]\label{def:comp-s}
Let $R$ be an $\sr{L}\times \sr{L}$ matrix. Then $R$ is called an $\cals$ matrix if there exists $u \in \dd{R}_{+}^{\sr{L}}$ such that $Ru>0$ (vector inequalities are to be interpreted componentwise). The matrix $R$ is said to be \emph{completely} $\cals$ if each principal
submatrix of $R$ is an $\cals$ matrix.
\end{definition}

Given a finite measure $\nu$ on $\R^{\sr{L}}_+ \equiv \{x\in \R^\sr{L}: x\ge 0\}$ whose total mass is not greater than $1$, that is, $\nu$ is a sub-probability distribution, let $\phi$ be its Laplace transform. Namely, 
\begin{align}
\label{eq:Laplace1}
  \phi(\theta) = \int_{\R^{\sr{L}}_+} e^{\br{\theta, x}} \nu(dx), \quad \theta\in \R^{\sr{L}}_-,
\end{align}
where $\br{x,y} = \sum_{i \in \sr{L}} x_i y_i$ is the inner product of of vectors $x, y\in \R^{\sr{L}}$.  

The following lemma follows from the uniqueness of the stationary distribution of an SRBM \cite{DaiKurt1994}. The current form follows from 
Lemma~2.1 of \cite{BravDaiMiya2017} and the appendix of the arXiv version of \cite{DaiMiyaWu2014}.

\begin{lemma}\label{lem:rbm-bar}
  Given an $\sr{L}\times \sr{L}$ positive definite matrix $\Sigma$, an $\sr{L}\times \sr{L}$
  completely-$\mathcal{S}$ matrix $R$, and a positive $\sr{L}$-vector $b$,
  there is at most one set of probability measures $\nu$ and $\nu_j$,
  $j\in \mathcal{L}$, such that   $\nu_{j}$  has the support in
  $\{x\in \R^{\sr{L}}_+: x_j=0\}$ for each $j\in \mathcal{L}$ and  for Laplace transforms $\phi$ and $\phi_{\ell}$ of $\nu$ and $\nu_{\ell}$, respectively,
\begin{align}
  \label{eq:rbm-bar1}
  \langle \theta, \Sigma \theta \rangle \phi(\theta)+ \sum_{\ell \in \mathcal{L}} b_{\ell}
  \langle \theta, R^{(\ell)} \rangle \Big(\phi_{\ell}(\theta)-\phi(\theta)\Big) =0, \qquad \theta \in \dd{R}_{-}^{\sr{L}},
\end{align}
where $R^{(\ell)}$ is the $\ell$th column of matrix $R$ for $\ell \in \sr{L}$.
\end{lemma}

When the probability measure $\nu$ in the lemma exists, it is the
unique stationary distribution of an SRBM with reflection matrix $R$,
covariance matrix $\Sigma$, and drift vector $-Rb$. In   what follows,
we say $\nu$ is the distribution uniquely determined by the set of
parameters $(R, \Sigma, b)$.

In our application of \lem{rbm-bar}, $\phi$ and $\phi_{\ell}$ are obtained as the limits of a sequence of the Laplace transforms of probability distributions. Namely, they are the Laplace transforms of vague limits of the probability distributions. Hence, they may not be Laplace transforms of probability distributions. Thus, for successfully using \lem{rbm-bar}, we need to verify that those $\phi$ and $\phi_{\ell}$ are the Laplace transforms of probability distributions. We introduce a notion of a tight system for this verification.

\begin{definition}[Tight system]\label{def:tight}
Given an $\sr{L}\times \sr{L}$
matrix $R$ and  a positive $\sr{L}$-vector $b$, we say that $(R,b)$ is a tight system if the set of linear equations and inequalities
\begin{align}
\label{eq:tight-condition1}
&\sum_{j\in \mathcal{L}} b_j R_{ij} \big(x_{A}^{(j)}-x_{A}\big) = 0, &\quad i \in A,\ A \subset \mathcal{L},\\
&x_{A}, x_{A}^{(j)} \in [0,1], &\quad A \subset A' \subset \mathcal{L}, \nonumber\\
& x_{A} \geq x_{A'},\ x_{A}^{(j)} \geq x_{A'}^{(j)},  &\quad j \in \mathcal{L}, A \subset A' \subset \mathcal{L}, \nonumber\\
&x_{A}^{(j)} = x_{A\setminus \{j\}}^{(j)}, &\quad j \in A \subset \mathcal{L}, \nonumber\\
&x_{\emptyset} = x_{\emptyset}^{(j)} = 1, &\quad j \in \mathcal{L}. \nonumber
\end{align}
has a unique solution $x_{A} = x_{A}^{(j)} = 1$ for all $j \in \mathcal{L}$ and $A \subset \mathcal{L}$.
\pend
\end{definition}

The meaning of the condition \eq{tight-condition1} of tight system can be seen through the following lemma, which is proved similarly to Lemma 5.1 of \cite{BravDaiMiya2017}. For completeness, we prove it in \app{tight1}.
\begin{lemma}
\label{lem:tight1}
Let $\phi$ and $\phi_{\ell}$ for $\ell \in \sr{L}$ be the Laplace transforms of finite measures on $\dd{R}_{+}^{\sr{L}}$ that satisfy \eq{rbm-bar1}. Then, (i) for $A \subset \sr{L}$,
\begin{align}
\label{eq:phi0}
  \phi_{A}(0-) = \lim_{\theta \uparrow 0} \phi(\theta_{A}), \qquad \phi_{A,\ell}(0-) = \lim_{\theta \uparrow 0} \phi_{\ell}(\theta_{A}), \quad \ell \in \sr{L}.
\end{align}
are well defined, where $\theta_{A}$ is the $L$-dimensional vector $\theta$ whose entries $\ell \in A$ are replaced by $0$. (ii) If $(R,b)$ is a tight system for $(x_{A}, x^{(\ell)}_{A}) \equiv (\phi_{A}, \phi_{A,\ell})$, then there exist unique probability distributions whose Laplace transforms $\phi$ and $\phi_{\ell}$ satisfy \eq{rbm-bar1}.
\end{lemma}

All our work is to find $\phi$ and $\phi_{\ell}$ satisfying \eq{rbm-bar1} and to verify $(R,b)$ to be a tight system by \lem{tight1}. We list sufficient conditions for this tightness.
\begin{itemize}
\item [(4.a)] If $R$ is an $\mathcal{M}$ matrix and $b > 0$, then $(R,b)$ is a tight system for any $b$ as long as $b > 0$, where an $\sr{L}\times \sr{L}$ matrix is said be an $\mathcal{M}$ matrix if it is invertible and has nonpositive off-diagonal entries and positive diagonal entries.
\item [(4.b)] For $L = 2$, $R \equiv \{R_{i,j}; i.j=1,2\}$ with $R_{i,i} > 0$ for $i=1,2$ and $b > 0$ is a tight system if and only if one of the following conditions holds: (4b.1) $R_{12} \le 0$ and $R_{21} \le 0$, (4b.2) $R_{12} < 0, R_{21} \ge 0$, (4b.3) $R_{12} \ge 0, R_{21} < 0$.
\item [(4.c)] $(R,b)$ for any reentrant line operating LBFS is tight.
\end{itemize}
Here, (4.a) follows from the proof of Proposition 5.1 of \cite{BravDaiMiya2017}, while (4.b) and (4.c) are proved in \cite{DaiJiMiya2023}.

\section{Heavy traffic assumption and the main result}
\label{sec:main}
\setnewcounter

In this section, we first introduce five assumptions that will be used
in our main theorem.  We will then state the main theorem. Finally, we
discuss reasons for making these assumptions after the
statement of the main theorem.

We consider a sequence of multiclass networks with SBP service
discipline indexed by $r \in (0,1]$, where $r$ monotonically tends to $0$. (With some abuse of notation,
we refer to such networks as a sequence of networks.) For the $r$th
network, let $a^{(r)}_{k} T_{e,k}(n)$ and $m^{(r)}_{k} T_{s,k}(n)$ be
the $n$th inter-arrival time of exogenous class $k$ arrivals and the
$n$th service times of class $k$ customers, respectively, where
$T_{e,k}(\cdot) \equiv \{T_{e,k}(i), i\ge 1 \}$ and  $T_{s,k}(\cdot) \equiv \{T_{s,k}(i), i \ge 1\}$ are independent sequences  of i.i.d.\ random variables introduced
in (\ref{eq2.1}). Thus, we employ
the same primitive increments for the entire sequence of queueing networks.
In a more general setting, these families of variables are given by
\emph{triangular} arrays of random variables, where the underlying
$T^{(r)}_{e,k}(i)$, $T^{(r)}_{s,k}(i)$ vary with $r$.  Heavy-traffic limit theorems under this
more general setup are robust under perturbations of the inter-arrival and
service vectors.  The purpose of the present setup is to
keep the notation simple.  Since \cite{BravDaiMiya2017}
use the framework of triangular arrays, our main result, Theorem~\ref{thm:main},
can be generalized straightforwardly to that setting.

We assume the following moment condition on inter-arrival and service-time distributions.
\begin{assumption}
\label{assum:moment}
  Assume condition (\ref{eq:3m}) is satisfied. Namely, inter-arrival
  and service times have finite $2+\delta_{0}$ moments for some $\delta_{0} >0$.
\end{assumption}

For $r \in (0,1]$ and $k \in \sr{K}$, let
\begin{align*}
  \lambda^{(r)}_{k} = 1/a^{(r)}_{k}, \qquad \mu^{(r)}_{k} = 1/m^{(r)}_{k}, 
\end{align*}
which are the exogenous arrival and service rates of class $k$
customers, respectively.
We note that the squared coefficients of variation of
the inter-arrival times and service times for class $k$, $c_{e,k}$ and
$c_{s,k}$, do not depend on the index $r$.
We assume that
$\{k \in \sr{K}; \lambda^{(r)}_{k} > 0\}$ does not depend on
$r \in (0,1]$  and is also denoted by $\sr{E}$. We also assume that the
routing matrix $P$ and the priority order do not depend on 
$r \in (0,1]$.

With vectors $\lambda^{(r)} \equiv (\lambda^{(r)}_{k}; k \in \sr{E})$ and $m^{(r)} \equiv (m^{(r)}_{k}; k \in \sr{K})$ replacing $\lambda$ and $m$, define $\alpha^{(k)}_k$, $\gamma_k^{(r)}$, and $\rho^{(r)}_j$,  following \eq{traffic}, (\ref{eq:gamma}), and (\ref{eq:rho}), respectively, for $k\in \mathcal{K}$ and $j\in \mathcal{J}$. In vector form, 
\begin{align}
\label{eq:traffic1}
 & \alpha^{(r)}=(I-P^{T})^{-1}\lambda^{(r)}, \qquad M^{(r)}=\diag (m^{(r)}), \quad \gamma^{(r)}= M^{(r)}\alpha^{(r)}, \qquad \rho^{(r)}=C\gamma^{(r)},
\end{align}
where $C$ is the constituency matrix defined in (\ref{eq:C}). Define
\begin{align}
\label{eq:beta}
    \beta^{(r)}_{k} = 1-\sum_{\ell \in {H}(k)} \gamma^{(r)}_{\ell}, \qquad k\in \sr{K},
\end{align}
which is the probability that class $k$ customers can be served. Recall that $\mathcal{L}$ is the set of low-priority classes defined in
  (\ref{eq:L}), and $s(k)$ is the station at which class $k$ customers get service. Since $H(\ell) = \{k \in \sr{K}; s(k)=\ell\}$ for $\ell \in \sr{L}$, it follows from (\ref{eq:traffic1}) and (\ref{eq:beta}) that
  \begin{align}
    \label{eq:betarho}
    \beta^{(r)}_{\ell} = 1- \rho^{(r)}_{s(\ell)}, \qquad \ell \in \mathcal{L}.
  \end{align}

\begin{assumption}
  \label{assum:ht}
We assume that
there are $K$-vectors $\lambda\ge 0$, $\lambda^*$, $m>0$ and $m^*$
 with
$\lambda_{k}=\lambda^*_{k}=0$ for $k\not\in \mathcal{E}$
such that, for index $r \in (0,1]$,
\begin{align}
\label{eq:3.4}
&\lambda^{(r)}_{k} = \lambda_{k}- r \lambda_{k}^*>0 \quad\mbox{for } k\in\mathcal{E},\qquad
m^{(r)}_{k}= m_{k} -  r m^*_{k}>0 \quad\mbox{for } k\in \sr{K}, \\
& \rho=CM\alpha=e, \label{eq:rho=1} \\
& c \equiv C[\diag(m^*)\alpha + \diag(m)\alpha^*] > 0,
\label{eq:b}
\end{align}
where $M = \diag(m)$, $e$ is the $\sr{J}$-vector of all 1's,  $\alpha=(I-P^{T})^{-1}\lambda$, $\alpha^*=(I-P^{T})^{-1}\lambda^*$.
\end{assumption}
Condition (\ref{eq:rho=1}) implies $\rho^{(r)}\to e$, which says that 
each station is \emph{critically loaded} in the limit as $r \downarrow 0$.
We do not impose a sign restriction on the ``deviation vectors''
$\lambda^*\in \R^K$ and $m^*\in \R^K$ in (\ref{eq:3.4}). One can check that
\begin{align}
    & \rho^{(r)} = e -  r c  + r^2 C\diag(m^*)\alpha^*.   \label{eq:rho-r}
\end{align}
We do require $c>0$ in condition (\ref{eq:b}) to make sure
$\rho^{(r)}<e$, a necessary condition for the $r$th network to be stable,
when $r>0$ is small enough.    Recall there is a one-to-one map between
  $\mathcal{J}$ and $\mathcal{L}$. Both $\rho$ and $c$ are
  $\mathcal{J}$-vectors. For notational convenience in the rest of the
  paper, we convert the $\mathcal{J}$-vector $c$ into an equivalent
  $\mathcal{L}$-vector $b$ via
  \begin{align}
    \label{eq:c2b}
    b_\ell = c_{s(\ell)} \quad  \text{ for } \ell \in \mathcal{L}.
  \end{align}

Note that  for the reentrant line considered in \sectn{2s5c}, the heavy-traffic conditions
\eqref{eq:3.4}--\eqref{eq:b} are satisfied with $b_1=b_4=1$.

We are concerned with the sequence of the multiclass queueing networks with SBP service disciplines that are indexed by $r\in (0,1]$. Denote the Markov process describing this network with index $r$ by $X^{(r)}(\cdot) \equiv \{X^{(r)}(t); t \ge 0\}$, where
\begin{align}
\label{eq:Xr-state}
  X^{(r)}(t) = \big(Z^{(r)}(t), R^{(r)}_{e}(t), R^{(r)}_{s}(t)\big) \in S\equiv \Z_+^K\times \R^{\sr{E}}_+\times \R^K_+, \qquad t \ge 0.
\end{align}
We are interested in the stationary distributions of the Markov process $X^{(r)}(\cdot)$. This motivates us to make the following assumption.

\begin{assumption}
\label{assum:stable}
For each index $r\in (0,1]$, $X^{(r)}(\cdot)$ has a unique stationary distribution.
\end{assumption}

It is well known that $\rho^{(r)}< e $ is not sufficient
for Assumption~\ref{assum:stable} to hold. For example, for the two-station, five-class reentrant line in \sectn{2s5c},  we remarked that the virtual station stability condition \eq{2s5cv0} is needed in addition to $\rho^{(r)}< e $ for Assumption~\ref{assum:stable} to be satisfied.

For each $r \in (0,1]$, denote an $S$-valued random variable subject to the stationary distribution $\pi^{(r)}$ of $X^{(r)}(\cdot)$ by $X^{(r)} \equiv (Z^{(r)}, R^{(r)}_{e}, R^{(r)}_{s})$.
Our main objective is to study the weak limit of the distribution of $rZ^{(r)}$ as $r \downarrow 0$ under the heavy-traffic condition.  We wish to prove
\begin{align}
  \label{eq:limit1}
  rZ^{(r)}\Rightarrow Z^*=(Z^*_L, 0_H)  \quad \text{ as } r\to 0,
\end{align}
where ``$\Rightarrow$'' denotes convergence in distribution, and,  for a  vector $z\in \R^K$, we define $z_L=(z_k; k\in \mathcal{L})\in \dd{R}^{\sr{L}}$ and $z_H=(z_k; k\in \mathcal{H})\in \dd{R}^{\sr{H}}$. Readers are warned that  we have abused the notation  in (\ref{eq:limit1}) by adopting the convention that
\begin{align}\label{eq:zLH}
 z=(z_L, z_H)
\end{align}
for a vector $z\in \R^K$.  Recall that vectors are considered  
column vectors unless stated otherwise.  Thus $z$, $z_L$, and $z_H$
are all column vectors with appropriate dimensions. We believe notation (\ref{eq:zLH}) is more attractive than the cumbersome expression $z=(z_L^{T}, z_H^{T})^{T}$, where the superscript $^{T}$ denotes transpose.

In (\ref{eq:limit1}), $r Z^{(r)}_H\Rightarrow 0$. This is an example of
state space collapse (SSC) under the priority service discipline.
The SSC is somewhat expected  because 
heavy-traffic conditions (\ref{eq:3.4})--(\ref{eq:b}) imply that
\begin{align*}
  \beta^{(r)}_k \to \beta_k=1-\sum_{\ell\in H(k)}  \alpha_\ell m_\ell>0, \quad k\in \mathcal{H},
\end{align*}
which means that each server has excess capacity after serving all
high-priority jobs. Therefore, the load from high-priority jobs is
\emph{not} in heavy traffic, and jobcounts in high-priority classes
should not blow up even though the entire network goes into heavy
traffic.  However, the preceding intuition  holds  only when
Assumption~\ref{assum:stable} holds, and proving SSC in steady state
is an independent task, which can be difficult. In
this paper, we assume the following.
\begin{assumption}
\label{assum:m-ssc2}
For each $k \in \sr{H}$, the collection of the steady-state job-count vectors $\{Z^{(r)}_{k}; r \in (0,1]\}$ is uniformly integrable, namely,
\begin{align}\label{eq:ui}
  \lim_{a \to\infty} \sup_{r\in (0,1]}\E\big[Z^{(r)}_k1(Z_k^{(r)}> a)\big]=0.
\end{align}
As a consequence,
\begin{align}
\label{eq:m-ssc p}
 \sup_{r\in (0, 1]} \dd{E}\big[  Z^{(r)}_{k}\big]<\infty \quad \mbox{ and } \quad
  \lim_{r\to 0} \dd{E}\big[  Z^{(r)}_{k}1(Z^{(r)}_k\ge h(r))\big] =0, \quad k\in \mathcal{H}
\end{align}
for any $h:(0,1]\to \R_+$ satisfying $\lim_{r\to 0}h(r)=\infty$.
\end{assumption}
It is well known that
\begin{align} \label{eq:mss+}
\sup_{r\in (0, 1]} \dd{E}\big[ ( Z^{(r)}_{k})^{1+\delta}\big]<\infty,
\end{align}
for some $\delta>0$, implies that the collection  is  uniformly integrable.
 \cite{CaoDaiZhan2022}
develops a sufficient condition to prove  \eq{mss+} based on SSC of corresponding fluid models.

As stated in the main theorem below, the random vector $Z^*_L\in \dd{R}^{\sr{L}}_+$ in (\ref{eq:limit1}) has the 
stationary distribution $\nu$ of an SRBM. For a given set of parameters $(R, \Sigma, b)$, 
such a  stationary distribution
$\nu$ is characterized in Lemma~\ref{lem:rbm-bar}.  To state the theorem,
 we need to define two $\mathcal{L}\times \mathcal{L}$ matrices $R$ and $\Sigma$.
By imposing conditions on $R$ and $\Sigma$, we will argue that probability
measures $\nu$ and $\nu_j$ on $\R^{\mathcal{L}}_+$ for $j\in \mathcal{L}$
corresponding to $(R, \Sigma, b)$ in Lemma~\ref{lem:rbm-bar}  exist 
and  are  unique, where $b>0$ is the vector in heavy-traffic condition
(\ref{eq:b}) and \eq{c2b}.

To define $R$, let
\begin{align}\label{eq:A}
  A= (I-P^{T})  \diag(\mu)(I-B),
\end{align}
and $B$ is the $K\times K$ matrix defined by
\begin{align}\label{eq:B}
   B_{k\ell}=
  \begin{cases}
    1 & \text{ if } \ell=k+, \\
    0 & \text{ otherwise},
  \end{cases}
\end{align}
for  classes $\ell, k\in \sr{K}$, where we call that $k+$ is the lowest class in $H(k) \setminus \{k\}$ (see \eq{k+}).
Define matrices $A_{L}$, $A_{LH}$, $A_{HL}$ and $A_{H}$ as
\begin{mylist}{3}
\item [(i)] $A_{L}$ and $A_{H}$ are the principal submatrices corresponding the index set $\mathcal{L}\subset \sr{K}$ and
$\mathcal{H}\subset \sr{K}$, respectively.
\item [(ii)] $A_{LH}$ and $A_{HL}$ are corresponding other blocks of $A$.
\end{mylist}
Thus, if the indexes of customer classes are appropriately chosen, then we can write $A$ as
\begin{align}\label{eq:A-b}
  A =
  \begin{pmatrix}
    A_{L} & A_{LH} \\
    A_{HL} & A_{H}
  \end{pmatrix}.
\end{align}

\begin{assumption}\label{assum:R}
  The matrix $A_{H}$ is assumed to be invertible. Furthermore, define $\sr{L} \times \sr{L}$ matrix ${R}$ via
  \begin{align}
  \label{eq:R}
  {R} = A_{L} - A_{LH} A_{H}^{-1} A_{HL};
\end{align}
then $R$ is completely $\sr{S}$, and $(R,b)$ is a tight system as defined in Definition~\ref{def:tight}. \end{assumption}

To define the $\mathcal{L} \times\mathcal{L}$ matrix $\Sigma$,  let
\begin{align}
\label{eq:q1}
  q(\theta) & = \frac 12 \sum_{k \in \sr{E}} \lambda_{k} c_{e,k}^{2} \theta_{k}^{2} \nonumber\\
  & \quad + \frac 12 \sum_{k \in \sr{K}} \alpha_{k} \Big[\sum_{\ell \in \sr{K}} P_{k,\ell} \theta_{\ell}^{2} - \Big(\sum_{\ell \in \sr{K}} P_{k,\ell} \theta_{\ell}\Big)^{2} + c_{s,k}^{2} \Big(\theta_{k} - \sum_{\ell \in \sr{K}} P_{k,\ell} \theta_{\ell} \Big)^{2} \Big], \quad \theta \in \dd{R}^{K}.
\end{align}
By Jensen's inequality, $\sum_{\ell \in \sr{K}} P_{k,\ell} \theta_{\ell}^{2} - \Big(\sum_{\ell \in \sr{K}} P_{k,\ell} \theta_{\ell}\Big)^{2} \ge 0$; thus, $q(\theta)$ is a nonnegative quadratic function of $\theta \in \R^K$.
For each (column) vector $\theta_L\in {\dd{R}^{\sr{L}}}$, define  vector
\begin{align}
  \label{eq:thetaH}
 \theta_H = - (A_{H}^{-1})^{T} (A_{LH})^{T} \theta_L.
\end{align}
With $\theta_H$ defined through equation (\ref{eq:thetaH}),
it is clear that $q(\theta_L, \theta_H)$ is 
a nonnegative quadratic function of
$\theta_L\in {\dd{R}^{\sr{L}}}$, where $(\theta_L, \theta_H)$ is the $K$-vector $\theta$ following convention (\ref{eq:zLH}).
Therefore, there is a  nonnegative  definite $\mathcal{L} \times \mathcal{L}$ symmetric matrix $\Sigma $ such that 
\begin{align}
\label{eq:Sigma}
  q(\theta_{L},\theta_H) = \br{\theta_{L}, \Sigma \theta_{L}}, \qquad \theta_{L} \in {\dd{R}^{\sr{L}}}.
\end{align}

\begin{theorem}\label{thm:main}
  Assume Assumptions~\ref{assum:moment}--\ref{assum:R} hold. Assume further
  that $\Sigma$ defined in (\ref{eq:Sigma}) is positive definite. Then
  the heavy-traffic steady-state convergence (\ref{eq:limit1}) holds,
  where $Z^*_L$ is a random vector whose distribution is uniquely
  determined from \eq{rbm-bar1} in Lemma~\ref{lem:rbm-bar} with parameters $(R, \Sigma, b)$
  defined in (\ref{eq:R}), (\ref{eq:Sigma}), and (\ref{eq:c2b}).
\end{theorem}

We have used \assum{moment} in this theorem for simplicity in the exposition. This assumption can be replaced by the following weaker one.

\renewcommand{\theassumption}{\thesection.1\alph{assumption}}
\setcounter{assumption}{0}
\begin{assumption}
  \label{assum:moment-weaker}
  All   inter-arrival distributions and the service-time distributions of classes $k \in \sr{K}_{1}$ have finite second moments, while  the service-time distributions of class $k \in \sr{K} \setminus \sr{K}_{1}$ have finite $(2+\delta_{0})$-th moments for some $\delta_{0} > 0$,
    where the set of highest classes $\mathcal{K}_1$ is defined in \eq{L}.
\end{assumption}
We explain in \app{moment-weaker} the reason that \assum{moment} in \thm{main} can be replaced by \assum{moment-weaker}. One may wonder whether \assum{moment-weaker} can be replaced by a further weaker assumption.
\begin{assumption}
\label{assum:moment-conjecture}
All the inter-arrival and service-time distributions have finite second moments.
\end{assumption}
At this point, we could not verify this replacement, and we leave it as
a future research topic.

Specializing \thm{main} to reentrant lines, we have the following corollary.
For a reentrant line, without loss of generality, we assume
$\mathcal{E}=\{1\}$.
\begin{corollary}\label{cor:1}
  Consider a sequence of reentrant lines in the setting of this
  section that satisfies  \assum{ht}. 
  Assume that
  \begin{align}
    \label{eq:2+m}
    \E(T_{e,1})^{2+\delta_{0}}<\infty, \quad      \E(T_{s,k})^{2+\delta_{0}}<\infty, \quad k\in \mathcal{K}
  \end{align}
  for some $\delta_{0}>0$.  Then, matrix $\Sigma$ is well defined through
  (\ref{eq:Sigma}). Assume $\Sigma$ is positive definite.  Assume
  further that $T_{e,1}$ is ``unbounded'' and ``spread out'' as
  defined in (1.4) and (1.5) of \cite{Dai1995}. Then,
\begin{itemize}
\item [(a)] The heavy-traffic
  steady-state convergence (\ref{eq:limit1}) holds for reentrant
  lines under LBFS service discipline.\\

\item[(b)] The heavy-traffic steady-state convergence (\ref{eq:limit1})
  holds for the two-station, five-class reentrant line in \sectn{2s5c} operating under SBP discipline (\ref{eq3.19}) if condition \eq{m5less}
  is satisfied.
  
\end{itemize}

\end{corollary}
\begin{proof}
  Clearly, condition \eq{2+m} implies \assum{moment}. For both cases,
  \assum{stable} is verified by using Theorem~4.1 of \cite{Dai1995}.
  To verify that this assumption is satisfied for  each of these two cases, we
  need to prove the stability of the corresponding fluid model.  The
  stability of the LBFS reentrant fluid model is proved in Theorem 4.4 of
  \cite{DaiWeis1996}. Under condition \eq{m5less}, the fluid model stability of
  the two-station, five-class reentrant line is proved in  Theorem 8.25 of
  \cite{DaiHarr2020}.  In both cases, state space collapse condition
  \eq{mss+} is satisfied by \cite{CaoDaiZhan2022}, which in turn
  implies \assum{m-ssc2}. In both cases, \assum{R} is verified in
  \cite{DaiJiMiya2023}.
\end{proof}

For the two-station, five-class  reentrant line, Assumptions \assumt{m-ssc2} (SSC) and \assumt{R} (matrix $R$) are equivalent to
condition \eq{m5less}. However, \assum{stable} (stability) is weaker than \eq{m5less}. Understanding the relationship among these three
assumptions in a general network is a future research direction 
  
In Section~\ref{sec:outline}, we give an outline of the proof of \thm{main}. The main
tool is the basic adjoint relationship (BAR) that characterizes the
stationary distribution $\pi^{(r)}$. For this BAR, we extend the approach that was developed in
\cite{BravDaiMiya2017} for single-class networks (generalized Jackson
networks), which we call the BAR approach. For a spacial case, we have done this for the two-station, five-class  reentrant line in \sectn{2s5c}, starting from the case that inter-arrival and service-time distributions are exponential.

Compared with \cite{BravDaiMiya2017},
the novelty of the present BAR approach, to be developed in detail  in
Sections \sect{bar-Xr}, \sect{proof-main}, and \sect{deriving-BAR}, is to
explicitly define Palm distributions carefully  and demonstrate the
intricate interplay between the stationary measure and Palm measures in
the setting of multiclass queueing networks. This interplay has  recently been  explored in \cite{GuanChenDai2023}.

Readers who are not familiar with our setting may be puzzled by our
reason for introducing a \emph{sequence} of networks and imposing the
heavy-traffic conditions (\ref{eq:3.4})--(\ref{eq:b}).  As a motivation,
one can consider the following situation.  In a production system, it
is up to the manager to decide how quickly jobs are to be released
into the system.  In particular, one needs to decide how heavily the
system should be loaded  to effectively use its resources.
Ideally, one would like to choose each $\rho_j$ close to 1. A sequence
corresponding to such a network arises by varying the load condition
imposed by the manager; one envisions the network as a member of the
sequence, with $r$ chosen small since $\rho$ is close to $e$.  The
heavy-traffic limit corresponding to this sequence of networks should
then provide insight on the behavior of the original network.

Finally, for later usage, we state the following lemma, whose proof is straightforward, where we  recall that $\beta^{(r)}_{k}$ is the probability that class $k$ can be served.
\begin{lemma}
\label{lem:mubetar}
  Equations  (\ref{eq:3.4}) and (\ref{eq:rho=1}) imply that as $r\to 0$,
\begin{align}
\label{eq:mu-r} 
  & \mu^{(r)}_{k} = \mu_{k} + r \mu_{k}^* + o(r), \qquad k\in \mathcal{K}, \\
\label{eq:beta-rL}
  &\beta^{(r)}_{\ell} =   1 - \rho^{(r)}_{s(\ell)} = r b_{\ell} + o(r), \qquad \ell \in  \sr{L}, \\
\label{eq:beta-rH}
& \beta^{(r)}_k = \beta_k + o(1) \text{ with } \beta_k>0, \qquad k \in \sr{H},    
\end{align}
where  $\mu^*_{k}=-m^*_{k}/(m_{k}^2)$, $\beta_{k} = 1 -  \sum_{\ell \in {H}(k)} \gamma_{k}$, and  $\gamma_{k} =  \alpha_{k} m_{k} > 0$  for $k\in \sr{K}$. 
 Here and later, we adopt the convention that 
$f(r)=o(g(r))$ and $f(r)=O(g(r))$ mean, respectively,
\begin{align*}
  \lim_{r \downarrow 0} \frac{f(r)}{g(r)}=0, \qquad \limsup_{r \downarrow 0} \left|\frac{f(r)}{g(r)}\right| < \infty.
\end{align*}
\end{lemma}

\section{BAR approach for queueing networks}
\label{sec:bar-Xr}
\setnewcounter

Our final goal is to prove \thm{main}, in which the inter-arrival and service times are generally distributed. Under this distributional assumption, $Z(\cdot) \equiv \{Z(t); t \ge 0\}$ is not a Markov process. Because of this, we first consider continuous-time Markov process $X(\cdot) \equiv \{X(t); t \ge 0\}$, which was introduced in \sectn{multiclass}. A prominent feature of this Markov process is that it has finitely many {jumps} in each finite interval and partially differentiable deterministic sample paths between adjacent {jump} instants of them. This class of Markov processes is called a piecewise deterministic Markov process  and was studied in \cite{Davi1984}.

Our starting point is the stationary distribution of this Markov process $X(\cdot)$, assuming its existence. To consider this distribution, we derive its basic adjoint relationship (BAR), also known in the literature as the stationary equation~\cite{Miya1994}. In \cite{Davi1984}, such a stationary equation is derived, but it {requires a boundary condition as an additional condition}, which is hard to handle. Here, we do not use such an additional condition. Namely, we recapitulate
the BAR approach that was first developed in \cite{Miya2017} and later
expanded in \cite{BravDaiMiya2017} for generalized Jackson networks.
Most of the foundational results are carefully developed here again
for the purpose of completeness and easy reference.

Note that the basic adjoint relationship, BAR for short, has been used to characterize the
stationary distributions for various Markov processes. See, for
example, \cite{EthiKurt1986} for diffusion processes,
\cite{HarrWill1987} for SRBMs, \cite{GlynZeev2008} for Markov
chains. The present BAR approach is in the same line, but a crucially different feature is included to handle well  the discontinuous state changes of $X(\cdot)$, for which Palm distributions are used. We will fully detail these Palm distributions, which are slightly different from those in \cite{Miya2017} and  not explicitly considered in \cite{BravDaiMiya2017}.

\subsection{Framework for deriving BAR for $X(\cdot)$}
\label{sec:BAR-general}

In this section, we present a framework of our BAR approach for the piecewise deterministic Markov process $X(\cdot)$ introduced in \sectn{multiclass}. Recall that this Markov process has state space {$S = \dd{Z}_{+}^{K} \times \dd{R}_{+}^{\sr{E}} \times \dd{R}_{+}^{K}$}, and $X(t) =(Z(t),R_{e}(t),R_{s}(t))$ at time $t \ge 0$. 

Define the set $\sr{D}$ as the set of functions $f:S\to \R$ satisfying the following conditions.
\begin{itemize}
\item [(\sect{bar-Xr}.a)]  $f(x)$ is bounded in $x \equiv (z,u,v) \in S$ and $f(z,u,v)$ is continuous in $(u, v)\in \R_+^{E+K}$ for each fixed $z\in \Z_+^K$.
\item [(\sect{bar-Xr}.b)] For  each fixed $z\in \Z_+^K$, $f(z,u, v)$ has partial derivatives from the right in $u_{\ell}$ and $v_{k}$ for each $\ell\in \mathcal{E}$ and $k \in \sr{K}$, and these partial derivatives are bounded, where $u \equiv (u_{k})_{k \in \sr{E}}$ and $v \equiv (v_{k})_{k \in \sr{K}}$.
\end{itemize}

For $f\in \mathcal{D}$, we introduce the following notations for describing the dynamics of $X(\cdot)$. 
\begin{align}
  \label{eq:A1}
 & \mathcal{A}f(x) = - \sum_{k\in \mathcal{E}} \frac{\partial f}{\partial u_{k}}(x)
  - \sum_{k\in \sr{K}} \frac{\partial f}{\partial v_{k}}(x)1(z_{k}>0, z_{H_+(k)}=0),\\
\label{eq:D1}
 & \Delta f(X)(t) = f(X(t)) - f(X(t-)),
\end{align}
where again $H_+(k)$ is defined in (\ref{eq:H+}), and $X(t-) \equiv \lim_{\varepsilon \downarrow 0} X(t-\varepsilon)$ for $t > 0$. It follows from the
fundamental theorem of calculus that  for $t > 0$,
\begin{align}
\label{eq:fX1}
  f(X(t))-f(X(0))=\int_0^t \mathcal{A}f(X(s))ds + \sum_{m=1}^{\infty} \Delta f(X)(\tau_{m}) 1(0 < \tau_{m} \le t),
\end{align}
where $\tau_{m}$ is the $m$th jump instant of $X(\cdot)$ for $m \ge 1$. It is not hard to see that $\tau_{m}$ is finite for each $m \ge 1$ and $\tau_{m} \to \infty$ as $m \to \infty$ since $T_{e,k}$ and $T_{s,k}$ are finite with probability one.

To facilitate the introduction of our version of Palm probability
measures, for each class $\ell\in \mathcal{E}$, we use $t^n_{e,\ell}$
to denote the arrival time of the $n$th class $\ell$ external
arrival. Similarly, for each $k\in \mathcal{K}$, we use $t^n_{s,k}$
to denote the service completion time of the $n$th class $k$ job.  We
call each entry in $\{t^n_{e,\ell}, n\ge 1\}$ for
$\ell\in \mathcal{E}$ and $\{t^n_{s,k}, n\ge 1\}$ for
$k\in \mathcal{K}$ an \emph{event time} of the queueing
network. It is possible
that multiple event times are identical, corresponding to a single
jump instant of $X(\cdot)$. In the following, we first assume that
\begin{itemize}
\item [(\sect{bar-Xr}.c)]    multiple events \emph{cannot} occur  simultaneously.
\end{itemize}
 At the end of this
subsection, we will argue that all results in this section continue to
hold when this assumption is removed.

For each $\ell\in \mathcal{E}$ and $k\in \mathcal{K}$, let
$N_{e,\ell}(\cdot)=\{ N_{e,\ell}(t), t\ge 0\}$ and
$N_{s,k}(\cdot)=\{ N_{s,k}(t), t\ge 0\}$ be the counting processes
associated with $\{t^n_{e,\ell}, n\ge 1\}$ and
$\{t^n_{s,k}, n\ge 1\}$, respectively.
In general, $N(\cdot) = \{N(t); t \ge 0\}$ is called a counting process if it is a nonnegative integer-valued process on $[0,\infty)$ that  is nondecreasing and right-continuous and that has limits from the left. Note that $N(\cdot)$ must have finitely many jump instants in each finite interval. Define the integration of a function $g:\R_+\to \R$ by counting process $N(\cdot)$ by 
\begin{align*}
  \int_{(0, t]} g(s)dN(0, s] = \sum_{m=1}^\infty {1(s_{m} \le t)} g(s_m)\Delta N(s_m)
\end{align*}
where $0<s_1<s_2<\ldots< s_m< {\ldots}$ are the jump times
of $N(\cdot)$ {and} $\Delta N(s_m)$ is the jump size at time $s_m$. It follows from  (\ref{eq:fX1}) that
\begin{align}
  &  f(X(t))-f(X(0))= \int_{0}^{t} \mathcal{A}f(X(s))ds \nonumber \\
  & +   \sum_{\ell\in \mathcal{E}} \int_{(0, t]} \Delta f(X(s)) d N_{e, \ell}(s)  +   \sum_{k\in \mathcal{K}} \int_{(0, t]} \Delta f(X(s)) d N_{s, k}(s).\label{eq:fX2}
\end{align}
Note that \eq{fX2} directly follows from \eq{fX1} when $N_{e,\ell}(t)$
and $N_{k}(t)$ for $\ell \in \sr{E}$ and $k \in \sr{K}$ do not have a
common jump at any time $t \ge 0$.

Under \assum{stable}, $X(\cdot)=X^{(r)}(\cdot)$ has the stationary distribution. Taking it as the initial distribution at time $0$, then $X(\cdot)$ is a stationary process. In what follows, we always assume that $X(\cdot)$ is a stationary Markov process and denote an $S$-valued random vector subject to the stationary distribution $\pi=\pi^{(r)}$ of $X(\cdot)$ by $X \equiv (Z, R_{e}, R_{s})$.

For $f\in \mathcal{D}$, since both $f$ and $\mathcal{A}f$ are bounded,  taking expectation for both sides of \eq{fX2}  with  $t=1$ yields
\begin{align}
 &  \dd{E}(\sr{A}f (X)) + \dd{E}
  \Big(
   \sum_{\ell\in \mathcal{E}} \int_{(0, 1]} \Delta f(X(s)) d N_{e, \ell}(s) 
 +   \sum_{k\in \mathcal{K}} \int_{(0, 1]} \Delta f(X(s)) d N_{s, k}(s) \Big) = 0.
\label{eq:A2}   
\end{align}

Note that this equation exactly corresponds to \eq{2s-Ef} of the two-station, five-class  reentrant network in \sectn{bounded-d}. 
We there have introduced Palm distributions to handle  well  the second expectation in \eq{2s-Ef}, which corresponds to the last two terms in \eq{A2}. We will take the same approach here. We first introduce a state space for the network states just before its jump instants. For $\ell\in \mathcal{E}$ and $k\in \mathcal{K}$,  define
  \begin{align*}
   & \Gamma_{e,\ell} = \big\{x=(z,u, v)\in S: u_\ell=0\big\}, \quad
     \Gamma_{s,k} =\big \{x=(z,u, v)\in S: v_k=0, z_k>0, z_{H_+(k)}=0\big\}, \\
    & \Gamma = \big(\cup_{k\in \mathcal{E}}\Gamma_{e,k}\big)\cup\big(\cup_{k\in \mathcal{K}}\Gamma_{s, k}\big).
  \end{align*}
We call $\Gamma$ the boundary of state space $S$.  
By our convention, the state process $X(\cdot)$ is right continuous. As a consequence,  one can verify that $X(t)\in S\setminus \Gamma$ for $t\ge 0$, and $X(t_m-)\in \Gamma$ for each $m\ge 1$.

As we have experienced in the definitions \eq{2s-palm-de} and \eq{2s-palm-ds}, it is important to evaluate $\dd{E}[N_{e, k}(1)]$ and $\dd{E}[N_{s, k}(1)]$ for defining Palm distributions. For this, recall that $\{\alpha_{k}; k \in \sr{K}\}$ is the unique solution of the traffic equation \eq{traffic}. We also note that $\dd{E}[N_{e, k}(1)]$ and $\dd{E}[N_{s, k}(1)]$ must be finite by the law of large numbers because $T_{e,k}$ and $T_{s,k}$ have finite and positive expectations.
\begin{lemma}\label{lem:rate}
  Assume Assumption~\ref{assum:stable}. For each $r\in (0,1]$,
  \begin{align}
    \label{eq:rateE} 
    & \dd{E}[N_{e, k}(1)] = \lambda_k, \quad k\in \mathcal{E}, \\
    & \dd{E}[N_{s, k}(1)] = \alpha_k, \quad k \in \mathcal{K}.
    \label{eq:rateS}       
  \end{align}
\end{lemma}
\begin{proof}
  We first prove  equation (\ref{eq:rateE}). The proof is
different from the proof for (A.12) in \cite{BravDaiMiya2017}. 
 Fix a $k\in \mathcal{E}$. For constant $\kappa > 0$, take $f(x)=u_k \wedge \kappa$ for \eq{fX2}. Since $\{R_{e,k}(t); t \ge 0\}$ is a stationary process and
\begin{align*}
  \sr{A}f(x) = - 1(u_{k} \le \kappa), \qquad \Delta f(X(t^{n}_{e,k})) = a_k T_{e,k}(n) \wedge \kappa,
\end{align*}
where $t^{n}_{e,k}$ is the $n$th increasing instant of $N_{e,k}(\cdot)$, taking the expectation of \eq{fX2} for $t=1$, we have
\begin{align*}
  - \dd{P}[R_{e,k} \le \kappa] + \dd{E}\Big[\sum_{n=1}^{\infty} (a_k T_{e,k}(n) \wedge \kappa) 1 (n \le N_{e,k}(1)) \Big] = 0.
\end{align*}
Since $T_{e,k}(n)$ and $\{n \le N_{e,k}(1)\} = {\Omega \setminus \{N_{e,k}(1)} < n\}$ are independent for each $n\ge 1$, this yields that 
\begin{align*}
  \dd{P}[R_{e,k} \le \kappa] = \dd{E}[ a_k T_{e,k}(1) \wedge \kappa] \dd{E}[N_{e,k}(1)].
\end{align*}
Letting $\kappa \to \infty$ in this formula, we have \eq{rateE} because $\dd{E}[a_k T_{e,k}(1) \wedge \kappa]$ converges to $a_k=1/\lambda_{k}$. \eq{rateS} is similarly proved. In this case, for $k \in \sr{K}$, we take $f(x)=z_{k}$ for \eq{fX2}. Then,
\begin{align*}
  \Delta f(X(s)) & = 1(R_{e,k}(s-) =0) - 1(\Phi^{(k)} \not= e^{(k)}) 1(R_{s,k}(s-) = 0) \\
    & \quad + \sum_{\ell \in \sr{K} \setminus \{k\}} 1(\Phi^{(\ell)} = e^{(k)}) 1(R_{s,\ell}(s-) = 0),
\end{align*}
and $\sr{A}f(x) = 0$. Hence, taking the expectation of \eq{fX2} yields
\begin{align*}
  \dd{E}[N_{e, k}(1)] - (1-P_{k,k}) \dd{E}[N_{s, k}(1)] + \sum_{\ell \in \sr{K} \setminus \{k\}} \dd{E}[N_{s, \ell}(1)] P_{\ell,k} = 0.
\end{align*}
Since $\dd{E}[N_{e, k}(1)] = \lambda_{k}$ by \eq{rateE}, this equation implies that $\{\dd{E}[N_{s, k}(1)]; k \in \sr{K}\}$ is the solution of the traffic equation \eq{traffic}. Hence, we have \eq{rateS} because $\{\alpha_{k}; k \in \sr{K}\}$ is the unique solution of \eq{traffic}.
\end{proof}

Similarly to \eq{2s-palm-de} and \eq{2s-palm-ds}, we define probability distributions $\dd{P}_{e,k}$ and $\dd{P}_{s,k}$ on $(S^{2}, \sr{B}(S^{2}))$ as
\begin{align}
\label{eq:palm-de}
 & \dd{P}_{e,k}[B] = \frac{1}{\lambda_k} \E\Big[ \int_{(0, 1]} 1((X(t-),X(t)) \in B) d N_{e, k}(t)\Big],\qquad B \in \sr{B}(S^{2}), \; k \in \mathcal{E},\\
\label{eq:palm-ds}
 & \dd{P}_{s,k}[B] = \frac{1}{\alpha_k} \E\Big[ \int_{(0, 1]} 1((X(t-),X(t)) \in B) d N_{s, k}(t)\Big],\qquad B \in \sr{B}(S^{2}), \; k \in \mathcal{K},
\end{align}
where $\dd{P}_{e,k}$ and $\dd{P}_{s,k}$ are indeed probability distributions because \lem{rate} implies $\dd{P}_{e,k}[S^{2}] = \dd{P}_{s,k}[S^{2}] = 1$. We call these distributions Palm distributions concerning $N_{e,k}$ and $N_{s,k}$, respectively. Similar to the arguments in \sectn{bounded-d}, let $(X_-, X_+)$ be  a pair of canonical random variables taking values in $S^2$ on the measurable space $(S^{2}, \sr{B}(S^{2}))$. Namely,
\begin{align*}
  (X_-, X_+)(x_1,x_2) =(x_1, x_2) \quad \text{ for each pair } \quad  (x_1,x_2)\in S^2.
\end{align*}
It follows from definitions, on each one of the probability spaces
$(S^{2}, \sr{B}(S^{2}),\dd{P}_{e,k})$ and
$(S^{2}, \sr{B}(S^{2}),\dd{P}_{s,k})$,  with probability one,
\begin{align}
\label{eq:X+-}
  X_+ \equiv (Z_{+},R_{+,e},R_{+,s})\in S\setminus\Gamma, \quad
  X_{-} \equiv (Z_{-},R_{-,e},R_{-,s})\in \Gamma.
\end{align}
Hence, $X_{-}$ can be considered the pre-jump state for each jump type caused by either an exogenous arrival or a service completion, and $X_+$ is the post-jump state under the Palm distributions. Denote the expectations under $\dd{P}_{e,k}$ and $\dd{P}_{s,k}$ by $\dd{E}_{e,k}$ and $\dd{E}_{s,k}$, respectively. Then, for Borel measurable functions
$f:\Gamma\to \R_+$ and $g:S\to \R_+$, we have
\begin{align}
 &  \E_{e,k}[f(X_{-})g(X_+)] = \frac{1}{\lambda_k} \E\Big[ \int_{(0, 1]} f(X(t-))g(X(t)) d N_{e, k}(t)\Big],\quad k \in \mathcal{E}, \label{eq:palm_e}\\
&  \E_{s,k}[f(X_{-})g(X_+)] = \frac{1}{\alpha_k} \E\Big[ \int_{(0, 1]} f(X(t-))g(X(t)) d N_{s, k}(t)\Big], \quad k\in \mathcal{K}.\label{eq:palm_s}
\end{align}

Probability distributions $\Prob_{e,k}$ and $\Prob_{s,k}$ are closely
related to Palm measures in the literature; see, for example,
\cite{BaccBrem2003, Miya1994}. In this regard, we make the following two
remarks.  First, Palm measures in \cite{Miya1994} are defined on the same measurable
space $(\Omega, \mathcal{F})$, where $(\Omega, \mathcal{F}, \dd{P})$
is the original probability space on which primitives such as
$T_{e,k}(\cdot)$ and $T_{s,k}(\cdot)$ are defined; in our definitions,
the measurable space is $(S^{2}, \sr{B}(S^{2}))$ on which $(X_-,X_+)$ is defined.
This paper does not require  knowledge of the Palm measures
beyond what is defined in \eq{palm-de} and \eq{palm-ds} and thus is self contained.
However, one must be careful because we are dealing with multiple probability spaces $(\Omega, \mathcal{F}, \dd{P})$, $(S^{2}, \sr{B}(S^{2}), \dd{P}_{e,k})$ and $(S^{2}, \sr{B}(S^{2}), \dd{P}_{s,k})$ at once. Since this is different from the standard stochastic analysis using a single probability space, one may be puzzled at first. However, we will work only through expected values computed on those probability spaces, so there should be no confusion.

For each measurable function $f$ from $S$ to $\dd{R}$, let
\begin{align}
\label{eq:delta-f1}
  \Delta f(X_{+},X_{-}) = f(X_{+}) - f(X_{-}).
\end{align}
Thus, random variables $X_{-}$,  $X_{+}$, and $\Delta f(X_{+},X_{-})$ are defined on the common measurable space $(S^{2}, \sr{B}(S^{2}))$. Note that the notation $\Delta f(X_{+},X_{-})$ is inconsistent with $\Delta f(X(t))$, but they can be distinguished by their arguments. Immediately from \eq{palm_e} and \eq{palm_s}, we have the following lemma.
  \begin{lemma}\label{lem:strongm} For each bounded Borel measurable function
    $f:S\to \R$,
\begin{align}
  & \E\Big[ \int_{(0, 1]} \Delta f(X(s)) d N_{e, \ell}(s)\Big]
    = \lambda_\ell \E_{e,\ell}[\Delta f(X_{+},X_{-})] ,\quad \ell \in \mathcal{E}, \label{eq:sm1}\\
  & \E\Big[ \int_{(0, 1]}\Delta f(X(s)) d N_{s, k}(s)\Big]=
       \alpha_k\E_{s,k}[\Delta f(X_{+},X_{-})] , \quad k\in \mathcal{K}.\label{eq:sm2}
\end{align}
\end{lemma}
With the new notational system, BAR (\ref{eq:A2}) becomes
\begin{align}
   \E[\mathcal{A}f(X)] & + \sum_{\ell\in \mathcal{E}} \lambda_\ell \E_{e,\ell}[\Delta f(X_{+},X_{-})] + \sum_{k\in \mathcal{K}} \alpha_k \E_{s,k}[\Delta f(X_{+},X_{-})]= 0, \quad f\in \mathcal{D}.
\label{eq:bar}                         
\end{align}

At this point, \eq{A2} and  \eq{bar} differ only in symbols, not in 
mathematical substance. Our next lemma and its corollary allow a practical way to compute
$\E_{e,\ell}[\Delta f(X_{+},X_{-})]$ and
$\E_{s,k}[\Delta f(X_{+},X_{-})]$. They also show that $X_{+} - X_{-}$ exactly corresponds to $X(t) - X(t-)$ at jump instants.

\begin{lemma}
\label{lem:X+}
The pre-jump state $X_-$ and  the post-jump state $X_+$ have the following representation,
\begin{align}
\label{eq:X+e}
 & X_{+} = X_{-} + (e^{(k)},a_{k} T_{e,k} e^{(k)}, 0)), \qquad \mbox{under } \dd{P}_{e,k}, k \in \sr{E},\\
\label{eq:X+s}
 & X_{+} = X_{-} + (-e^{(k)} + \Phi^{(k)},0,  m_{k} T_{s,k}e^{(k)}), \qquad \mbox{under } \dd{P}_{s,k}, k \in \sr{K},
\end{align}
where $T_{e,\ell}$ for $\ell \in \sr{E}$ and $T_{s,k}, \Phi^{(k)}$ for $k \in \sr{K}$ are random variables defined on the measurable space $(S^{2}, \sr{B}(S^{2}))$ such that, under Palm distribution $\dd{P}_{e,\ell}$, $T_{e,\ell}$ is independent of $X_{-}$ and has the same distribution as that of $T_{e,\ell}(1)$ on $(\Omega,\sr{F},\dd{P})$, and, under Palm distribution $\dd{P}_{s,k}$, $(T_{s,k},\Phi^{(k)})$ is independent of $X_{-}$ and has the same distribution as that of $(T_{s,k}(1), \Phi^{(k)}(1))$ on $(\Omega,\sr{F},\dd{P})$.
\end{lemma}


\begin{proof}
Let $t^{m}_{e,k}$ be the $m$th increasing instant of $N_{e,k}(\cdot)$.
From \eq{palm_e}, for bounded Borel measurable functions
$f:\Gamma\to \R_+$ and $h:S\to \R_+$
\begin{align}
\label{eq:fh-X}
  \dd{E}_{e,k}[f(X_{-})h(X_{+}-X_{-})] & = \frac{1}{\lambda_k} \E\Big[ \int_{(0, 1]} f(X(t-))h(X(t) - X(t-)) d N_{e, k}(t)\Big] \nonumber\\
 & = \frac 1{\lambda_{k}}  \dd{E}\left[\sum_{m=1}^{\infty} 1(t^{m}_{e,k} \le 1) f(X(t^{m}_{e,k}-))
  h(X(t^{m}_{e,k})-X(t^{m}_{e,k}-))
  \right] \nonumber\\
  & = \frac 1{\lambda_{k}} \sum_{m=1}^{\infty} \dd{E}\left[ 1(t^{m}_{e,k} \le 1) f(X(t^{m}_{e,k}-))
\right] \dd{E}[h((e^{k},a_{k}T_{e,k}(1),0))] \nonumber\\
  & = \dd{E}_{e,k}[f(X_{-})] \dd{E}[h((e^{k},a_{k}T_{e,k}(1),0))],
\end{align}
where in obtaining the third equality, we have  used the following three facts on probability space $(\Omega, \mathcal{F}, \dd{P})$: (a)
\begin{align*}
  X(t^{m}_{e,k}) = X(t^{m}_{e,k}-)+ \big(
  e^{(k)},  a_{k} T_{e,k}(m) e^{(k)}, 0)\big),
\end{align*}
(b) $T_{e,k}(m)$ is independent of $X(t^{m}_{e,k}-)$ and $t^{m}_{e,k}$,
and (c) $\{T_{e,k}(m), m\ge 1\}$ is an $i.i.d.$ sequence.
Since \eq{fh-X} implies that $\dd{E}_{e,k}[h(X_{+}-X_{-})] = \dd{E}[h((e^{k},a_{k}T_{e,k}(1),0))]$ for $f(x) \equiv 1$, we have
\begin{align*}
   \dd{E}_{e,k}[f(X_{-})h(X_{+}-X_{-})] = \dd{E}_{e,k}[f(X_{-})] \dd{E}_{e,k}[h(X_{+}-X_{-})].
\end{align*}
That is, $X_{-}$ and $X_{+} - X_{-}$ are independent under $\dd{P}_{e,k}$, and $X_{+} - X_{-}$ under $\dd{P}_{e,k}$ has the same distribution as $(e^{k},a_{k} T_{e,k}(1), 0)$ under $\dd{P}$.
Thus, there exists a random variable $T_{e,k}$ on $(S^{2}, \sr{B}(S^{2}))$ that has the same distribution as $T_{e,k}(1)$ under that of $\dd{P}$ such that
\begin{align*}
 X_{+} = X_{-} + (e^{(k)},a_{k} T_{e,k}e^{(k)}, 0), \quad \mbox{under} \quad \dd{P}_{e,k},
\end{align*}
where $T_{e,k}$ is independent of $X_{-}$. This proves all the claims on $X_{+} - X_{-}$. Similar results are also obtained for $\dd{P}_{s,k}$. Thus, the lemma is proved.
\end{proof}

The following lemma is immediate from this lemma.

\begin{corollary}\label{cor:palm}
For each $f \in \sr{D}$, define $\bar{f}_{e,k}$, $\bar{f}_{s,k}$ and $\bar{f}$ as
\begin{align*}
  & \bar{f}_{e,k}(x) = \dd{E}_{e,k}\Big[f\Big(z+ e^{(k)},\,   u+ a_{k} T_{e, k}(1) e^{(k)}, v\Big) \Big ] \quad  x\in \Gamma_{e,k}, \quad k\in \mathcal{E},
\\
  & \bar{f}_{s,k}(x) = \dd{E}_{s,k}\Big[ f\Big(z-e^{(k)}+ \Phi^{(k)}(1), \, u, \, v+ m_{k} T_{s,k} e^{(k)}\Big) \Big]\quad x\in \Gamma_{s,k}. \quad                              k\in \mathcal{K}, 
 \\
  &  \bar{f}(x) = \sum_{k\in\mathcal{E}} \bar{f}_{e,k}(x)1(u_k=0)
  + \sum_{k\in\mathcal{K}} \bar{f}_{s,k}(x)1(v_k=0),\quad x=(z,u, v)\in S, 
\end{align*}
then
\begin{align}
\label{eq:fbar_e}
 & \E_{e,k}[\Delta f(X_{+},X_{-})]  = \E_{e,k}[\Delta\bar{f}(X_{-})],  \quad k\in \mathcal{E}, \\
\label{eq:fbar_s} 
& \E_{s,k}[\Delta f(X_{+},X_{-})]    =     \E_{s,k}[\Delta\bar{f}(X_{-})], \quad k\in \mathcal{K},
\end{align}
where $\Delta\bar{f}(x) = \bar{f}(x)-f(x), \quad x\in S$.
\end{corollary}

It can be proved that \eq{bar} fully characterizes the stationary distribution of
$X(\cdot)$ in the following sense (e.g., see \cite{Miya1991} for its proof). The stationary distribution exists if and only if there are distributions $\nu$ on $S$ and $\nu_{e,k}, \nu_{s,k}$ on $\Gamma$ such that
\begin{align*}
   \int_{S} \mathcal{A}f(x) \nu(dx) & + \sum_{k\in \mathcal{E}} \lambda_{k} \int_{\Gamma_{e,k}} \Delta \bar{f}(dx) \nu_{e,k}(dx) + \sum_{k\in \mathcal{K}} \alpha_k \int_{\Gamma_{s,k}} \Delta \bar{f}(dx) \nu_{s,k}(dx)= 0, \quad f\in \mathcal{D}.
\end{align*}

We now use \eq{bar} to prove the following lemmas.
For each of our proofs, we construct a particular test function
$f\in \mathcal{D}$ to be used in BAR \eq{bar}. For $f\in \mathcal{D}$,
both $f$ and $\mathcal{A}f$ need to be bounded. In the following, our
$f$'s are not always bounded. To overcome this difficulty, we apply \eq{bar}
to test function $f\wedge \kappa$ for each fixed $\kappa>0$.  Then, we
take the limit in each of the terms in \eq{bar} as $\kappa\to \infty$. Since
this limit procedure is standard (see, for example, the proof of (A.12) in \cite{BravDaiMiya2017}), we omit it in our proofs below.

We next state and prove a lemma that evaluates the tail of expectations.
\begin{lemma}\label{lem:R} For $n \ge 0$,
  \begin{align}
    \label{eq:Re2}
    &   \E[R^{n}_{e,k}1(R_{e,k}\ge c) ]= \frac{1}{n+1}  a^{n}_k \E[ (T^{n+1}_{e,k} - c^{n+1}/a^{n+1}_{k} )1(a_kT_{e,k}\ge c)], \qquad k\in \mathcal{E}, \; c\in \R_+ \\
    & \E[R^n_{s,k}1(R_{s,k}\ge c, Z_k>0, Z_{H_+(k)}=0)]\nonumber \\
    & \qquad = \frac{1}{n+1} \gamma_k m_k^{n} \E[ (T_{s,k}^{n+1} - c^{n+1}/m^{n+1}_{k})1(m_k T_{s,k}\ge c)], \qquad k\in \mathcal{K}, \; c\in \R_+,     \label{eq:Rs2} 
  \end{align}
  where we recall $\gamma_k=\alpha_km_k$.
\end{lemma}
\begin{proof}
Fix $n \ge 0$, $k\in \mathcal{E}$, and $c\in \R_+$. We first prove \eq{Re2}.  Let
  $f(x)= [\max (u_k, c)]^{n+1}$ for $x=(z,u,v)\in S$. It follows that
\begin{align*}
    & \mathcal{A}f(X)=-(n+1) R^{n}_{e,k} 1(R_{e,k} \ge c), \\
    & \Delta f(X_{+},X_{-}) = a_k^{n+1} (T_{e,k}^{n+1} - c^{n+1}/a^{n+1}_{k})1(a_kT_{e,k}\ge c) 1(R_{-,e,k} = 0).
  \end{align*}
 By \eq{bar},
  \begin{align*}
   (n+1) \E[ R^{n}_{e,k}1(R_{e,k}\ge c)] & =\lambda_k a_k^{n+1} (\E[T^{n+1}_{e,k} - c^{n+1}/a^{n+1}) 1(a_kT_{e,k}\ge c)]\E_{e,k}[1]\\
   & = a^{n}_k \E[(T^{n+1}_{e,k} - c^{n+1}/a^{n+1})1(a_kT_{e,k}\ge c)],
  \end{align*}
  proving (\ref{eq:Re2}). Next we prove \eq{Rs2}. Let $f(x)=[\max(v_k,c)]^{n+1}$, then
 \begin{align*}
    & \mathcal{A}f(X)=-(n+1) R_{s,k}^n (R_{s,k} \ge c, Z_{k}>0, Z_{H_+(k)}=0), \\
    & \Delta f(X_{+},X_{-}) = m_k^{n+1} (T^{n+1}_{s,k} - c^{n+1}/m^{n+1}_{k})1(m_k T_{s,k}\ge c) 1(R_{-,s,k} = 0).
  \end{align*}
Hence, similarly to \eq{Re2}, BAR \eq{bar} implies \eq{Rs2}.
\end{proof}

This lemma will be used in the proofs of Lemmas \lemt{idle-prob} and \lemt{R-r2} and \app{moment-weaker}. We now make a connection with \cite{BravDaiMiya2017}.
The following lemma appeared in \cite{BravDaiMiya2017}. Its proof follows from \eq{bar} immediately.
\begin{lemma}
  \label{lem:boundary1}
  Assume $f\in \mathcal{D}$ satisfies
 \begin{align}
  \label{eq:martingale_e}
  & \bar{f}_{e,k}(x)=f(x), \qquad k\in \mathcal{E}, \quad x\in \Gamma_{e,k}, \\
  & \bar{f}_{s,k}(x)=f(x), \qquad k\in \mathcal{K}, \quad x\in \Gamma_{s,k}.
    \label{eq:martingale_s}
\end{align} 
Then
\begin{align}
\label{eq:A3}
  \dd{E}(\sr{A}f (X)) = 0.
\end{align}
\end{lemma}
BAR \eq{A3} is the main tool used in \cite{BravDaiMiya2017}. We can still rely
on \eq{A3} to prove some cases of \thm{main}  in this paper.
For example, assume that $T_{e,k}$ for $k \in \sr{E}$ and $T_{s,k}$ for $k \in \sr{K}$ have  general distributions  but have bounded  supports. Then, similar to \eq{f-theta}, redefine $f_{\theta}$ as
\begin{align}\label{eq:f}
  f_{\theta}(x) = g_{\theta}(z)\exp\big( - \langle \eta(\theta), \lambda u \rangle - \langle \xi(\theta), \mu v\rangle\big), \quad x=(z, u, v) \in S, \quad \theta \in \dd{R}^{K},
\end{align}
 where $\lambda u = (\lambda_k u_k,k\in\mathcal{E})$, $\mu v=(\mu_kv_k, k\in \mathcal{K})$, and
\begin{align}
\label{eq:g}
   g_{\theta}(z) = \exp(\br{\theta,z}), \qquad z \in \dd{Z}_{+}^{K},
\end{align}
and,   similar to \eq{2s-eta-g}--\eq{2s-xi5-g},
functions 
$\eta_{k}(\theta_{k})$ and $\xi_{k}(\theta)$ are defined
through the following equations:
\begin{align}
\label{eq:eta-r}
 &  e^{\theta_{k}} \dd{E}(e^{-\eta_{k}(\theta_{k}) T_{e,k}}) = 1, \qquad k \in \sr{E},\\
\label{eq:xi-r}
  & \sum_{\ell \in \ol{\sr{K}}} P_{k,\ell} e^{-\theta_{k}+\theta_{\ell}} \dd{E}(e^{-\xi_{k}(\theta) T_{s,k}}) = 1, \qquad k \in \sr{K}.
\end{align}
Then we can derive a BAR of $X^{(r)}$ because conditions \eq{martingale_e} and \eq{martingale_s} are satisfied, respectively. Indeed, it follows from (\ref{eq:g}) and \eq{A1} that 
\begin{align*}
  \sr{A}f_{\theta}(X) = \sum_{k \in \sr{E}} \lambda_{k} \eta_{k}(\theta_{k}) f_{\theta}(X) + \sum_{k \in \sr{K}} \mu_{k} \xi_{k}(\theta) f_{\theta}(X) 1(Z_{k} > 0, Z_{H_{+}(k)} = 0).
\end{align*}
Therefore, \eq{A3} implies that for each $\theta \in \R^K_-$
\begin{align}
\label{eq:light-pre-bar0}
  \sum_{k \in \sr{E}} \lambda_{k} \eta_{k}(\theta_{k}) \dd{E}\left[f_{\theta}(X)\right] + \sum_{k \in \sr{K}} \mu_{k} \xi_{k}(\theta) \dd{E}\left[f_{\theta}(X) 1(Z_{k} > 0, Z_{H_{+}(k)} = 0)\right] = 0.
\end{align}
Furthermore, if a set $\Theta_L\subset \R^{\mathcal{L}}_-$ similar to
the one defined in \eq{2s-thetaL-region} is nonempty, all the arguments in the
case of the two-station, five-class  network in \sectn{2s5c} similarly work
for the present network, and \thm{main} can be proved.

However, it is too strong to assume that $T_{e,k}$ for $k \in \sr{E}$ and $T_{s,k}$ for $k \in \sr{K}$ have bounded supports, and we are not able to prove $\Theta_L$ nonempty in general. To prevent these extra assumptions, we truncate $u$, $v$ and $z_{H}$ in the test function $f_{\theta}(z,u,v)$ in \eq{f}, where $z=(z_L, z_H)$. This kind of   truncation was done for $u, v$ in \cite{BravDaiMiya2017}. However, the truncation of $z_{H}$ causes a serious problem in deriving a BAR because \eq{martingale_e} and \eq{martingale_s} are no longer satisfied after the truncation. This is a challenge that did not arise in \cite{BravDaiMiya2017}. We will attack this problem using a so-called asymptotic BAR in the next section.

We end this section by outlining an approach to remove the no-simultaneous-events assumption (\sect{bar-Xr}.c).
First we sort all event times $t^n_{e,\ell}$ and $t^n_{e,\ell}$,
$n\ge 1$, $\ell\in \mathcal{E}$ and $k\in \mathcal{K}$. The sorted
sequence in nondecreasing order is denoted by $\{\tau_m, m\ge 1\}$.
We assume there is a rule to break ties when multiple event times are
equal. For example, one may adopt a rule that arrival events precede
service completion events, and low-class events precede high-class
events. Note that the event sequence $\{\tau_m,
m\ge 1\}$ here is different from the jump instant sequence in \eq{fX1}. Here, when $\tau_m=\tau_{m+1}$
by definition
\begin{align}
  \label{eq:mm}
  X(\tau_m)=  X(\tau_{m+1}) \quad \text{ and }\quad
    X(\tau_m-)=  X(\tau_{m+1}-).
\end{align}
We now define what we call intermediate states $Y_{\tau_m}$ and $Y_{\tau_m-}$ for $m\ge 1$. In general,
$Y_{\tau_m}\neq Y_{\tau_{m+1}}$ in contrast to $ X(\tau_m)=X(\tau_{m+1})$ in \eq{mm}.
To define intermediate states, we call
\begin{align}
  \label{eq:block}
  \tau_{n-1}< \tau_{n} = \ldots =\tau_{n-1+\delta} < \tau_{n+\delta}
\end{align}
an event block of size $\delta$ starting from $n$. We now define
$Y_{\tau_m}$ and $Y_{\tau_m-}$ for $m=n, \ldots,  n-1+\delta$ as follows:
\begin{align*}
& Y_{\tau_n-}=X(\tau_{n}-), \\  
& \text{for }   m=n, \ldots,  n-1+ \delta, \\
& \qquad Y_{\tau_m} = Y_{\tau_m-} +
  \begin{cases}
 \big(
    e^{(\ell)},  a_{\ell} T_{e,\ell}(i) e^{(\ell)}, 0)\big) & \text{ if } \tau_m=t^i_{e,                                                  \ell} \\
 \big(
 - e^{(k)}+\Phi^{(k)}(j),  0, m_{k} T_{s,k}(j) e^{(k)})\big) & \text{ if } \tau_m=t^j_{s,k}     
  \end{cases},
  \\
& \qquad Y_{\tau_{m+1}-}=Y_{\tau_m}.
\end{align*}
One can verify that all the exposition and proofs in this
section continue to be valid as long as for each $m\ge 1$,  $X(\tau_m)$
and $X(\tau_m-)$ are replaced by $Y_{\tau_m}$ and  $Y_{\tau_m-}$, respectively.
The paragraph below (3.13) of \cite{BravDaiMiya2017} provides a
similar approach to  dealing with simultaneous events in generalized
Jackson networks.  See also (2.6) and (M4) of \cite{Miya2023} for a
similar treatment.

\subsection{Test functions for BAR of $X^{(r)}$}
\label{sec:choosing-tf}

Recall that $X^{(r)}=(Z^{(r)}, R^{(r)}_{e}, R^{(r)}_{s})$ is subject to the stationary distribution of Markov process $X^{(r)}(\cdot)$. To prove \thm{main}, we need to find an equation to characterize the limit of the distributions $Z^{(r)}$ as $r \downarrow 0$. To this end, we first derive a BAR for $X^{(r)}$, then derive a BAR for $Z^{(r)}$. For the BAR of $X^{(r)}$, we take a test function from the state space $S = \dd{Z}_{+}^{K} \times \dd{R}_{+}^{\sr{E}} \times \dd{R}_{+}^{K}$ to $\dd{R}$. The choice of this test function is crucial  to our approach. 

As discussed at the end of \sectn{BAR-general}, we truncate $z_{H}, u, v$ in the test function $f^{(r)}_{\theta}(x)$. This is done in the following way. We first truncate the queue length vector, and define test function $g_{\theta,s}$ for $z \in \dd{Z}_{+}^{K}$ as
\begin{align}
\label{eq:tg1}
  g_{\theta,s}(z) = \exp\left(\br{\theta_{L}, z_L} + \br{\theta_{H}, z_H \wedge 1/s} \right), \qquad z \in \dd{Z}_{+}^{K},
\end{align}
where $z_{H} \wedge 1/s$ is the $H$-dimensional vector whose $k$th entry is $\min(z_{k},1/s)$ for $k \in \sr{H}$. Then, incorporating those two truncations on $u,v$, we define the test function  $f^{(r)}_{\theta,s,t}$ for $r, s, t \in (0,1]$ and $\theta \in \Theta$ for $X^{(r)}$ as
\begin{align}
\label{eq:fr-1}
 f^{(r)}_{\theta,s,t}(x) = g_{\theta,s}(z) \exp(- \br{\eta(\theta,t),\lambda^{(r)} u \wedge t^{-1}}- \br{\xi(\theta,t),\mu^{(r)} v \wedge t^{-1}}).
\end{align}
For this test function, we have to change \eq{eta-r} and \eq{xi-r} to
\begin{align}
\label{eq:eta-rt}
 &  e^{\theta_{k}} \dd{E}(e^{-\eta_{k}(\theta_{k},t) (T_{e,k} \wedge t^{-1})}) = 1, \qquad k \in \sr{E},\\
\label{eq:xi-rt}
  & \sum_{\ell \in \ol{\sr{K}}} P_{k,\ell} e^{-\theta_{k}+\theta_{\ell}} \dd{E}(e^{-\xi_{k}(\theta,t) (T_{s,k}\wedge t^{-1})}) = 1, \qquad k \in \sr{K}.
\end{align}

These $\eta_{k}(\theta_{k},t)$ and $\xi_{k}(\theta,t)$ are uniquely determined by \eq{eta-rt} and \eq{xi-rt} as shown in \cite{BravDaiMiya2017}, but the proof there is a bit complicated. So, we will verify these facts in a simpler way by \lem{Taylor-1} in \app{P-expansion}.

Define
\begin{align}
\label{eq:Theta}
\Theta = \dd{R}^{\sr{L}}_-\times \dd{R}^{\sr{H}}.
\end{align}
We are now ready to define two sets of moment generating functions (MGFs) for $Z^{(r)}$ and $X^{(r)}$. Recall that $H(k) \subset \mathcal{K}$ is the set of classes at station $s(k)$ with priority at least as high as $k$ (see \eq{Hk}). For $Z^{(r)}$, define, for  each $\theta \in \Theta$ and $r\in (0,1]$,
\begin{align}
  \label{eq:phi-r}
  \phi^{(r)}(\theta)=   \E[g_{\theta,r}(Z^{(r)})], \quad \phi^{(r)}_k(\theta) = \E[g_{\theta,r}(Z^{(r)}) \mid Z^{(r)}_{H(k)}=0],\quad k\in \mathcal{K}, \quad \theta \in \dd{R}_{-}^{K},
\end{align}
where $z\in \R^K$, $z_A$ is defined to be $(z_k; k\in A)$.  Note that $\phi^{(r)}(r\theta)$ is the Laplace transform of $rZ^{(r)}$ for $\theta \in \dd{R}_{-}^{\sr{K}} \subset \Theta$. Since $\Prob\{Z^{(r)}_{H(k)} = 0\}$ is the probability that there is no customer whose priority is at least $k$ at station $s(k)$, the following lemma is intuitively clear.
\begin{lemma}
\label{lem:idle-prob}
Under \assum{stable},
  \begin{align}
    \label{eq:idle-prob}
    \Prob\{Z^{(r)}_{H(k)} = 0\} = \beta^{(r)}_k, \qquad k\in \mathcal{K},\quad r\in (0,1],
  \end{align}
where $\beta^{(r)}_{k}$ is defined in (\ref{eq:beta}).
\end{lemma}
\begin{proof}
\eq{idle-prob} is a special case  of \eq{Rs2} (by setting $n=0$ and $c=0$ there) in \lem{R} of Section~\ref{sec:BAR-general} because $\dd{P}(Z^{(r)}_{k} > 0, Z^{(r)}_{H_{+}(k)}=0) = \dd{P}(Z^{(r)}_{H_{+}(k)}=0) - \dd{P}(Z^{(r)}_{H(k)}=0)$.
\end{proof}

For $X^{(r)}$, we let $s=r$ and $t=r^{1-\varepsilon_{0}}$, where $\varepsilon_{0} \in (0, \delta_{0}/(1+\delta_{0})]$, and define truncated MGFs $\psi^{(r)}(\theta)$ and $\psi^{(r)}_k(\theta)$ for $k\in \mathcal{K}$ as
\begin{align}
\label{eq:psi}
& {\psi}^{(r)}(\theta) = \dd{E}\Big[{f}^{(r)}_{\theta,r, {r^{1-\varepsilon_{0}}}}(X^{(r)})\Big], \qquad {\psi}^{(r)}_{k}(\theta) = \dd{E}\Big[{f}^{(r)}_{\theta,r, {r^{1-\varepsilon_{0}}}}(X^{(r)})\Big \mid Z^{(r)}_{H(k)} = 0 \Big],
\end{align}
hich are well defined for each $\theta \in \Theta$ because $f^{(r)}_{\theta, s, t}(x)$ is bounded in $x$ for each $\theta\in \Theta$, $r,s, t\in(0, 1]$. Note that these MGFs cannot be called Laplace transforms because their domains are not limited to $\dd{R}_{-}^{K}$.

\section{Proof of \thm{main}}
\label{sec:proof-main}
\setnewcounter

The aim of this section is to prove \thm{main}. Throughout this section,
we assume Assumptions~\ref{assum:moment}--\ref{assum:R} and use
$X^{(r)} \equiv (Z^{(r)},R^{(r)}_{e},R^{(r)}_{s})$ to denote an $S$-valued random variable subject to the stationary distribution of the Markov process $X^{(r)}(\cdot)$. This proof of \thm{main} requires several steps. We first outline the proof in six steps in \sectn{outline}. All steps are fully detailed in the subsequent sections.

\subsection{Outline of the proof}
\label{sec:outline}

Recall that, once \lem{rbm-bar} is proved, the proof of \thm{main} is completed with help of \lem{tight1}. Thus, we aim to prove the BAR \eq{rbm-bar1} in \lem{rbm-bar}. This BAR will be obtained from pre-limit BARs, which takes six steps.

\begin{mylist}{3}
\item [(Step 1)]  Using the MGFs $\psi^{(r)}$ and $\psi^{(r)}_{k}$ of \eq{psi}, we derive an asymptotic BAR for $X^{(r)}$ in the following proposition. 

\begin{proposition}\label{pro:a-bar}
  Assume the assumptions in Theorem~\ref{thm:main}. Then, for each fixed $\theta \in \Theta$, 
\begin{align}
\label{eq:pre-bar0}
 &  q^{(r)}(r\theta, r^{1-\epsilon_0})  \psi^{(r)}(r\theta) - \sum_{\ell \in \mathcal{L}} \beta^{(r)}_{\ell} \mu^{(r)}_{\ell} \xi_{\ell}(r\theta,r^{1-\epsilon_0})  \Big(\psi_{\ell}^{(r)}(r\theta)-\psi^{(r)}(r\theta)\Big) \nonumber \\
 &  \quad + \sum_{\ell\in\mathcal{H}}  \beta^{(r)}_{\ell} \Big (\mu^{(r)}_{\ell-} \xi_{\ell-}(r\theta, r^{1-\epsilon_0})-\mu^{(r)}_{\ell} \xi_{\ell}(r\theta, r^{1-\epsilon_0})\Big) \Big(\psi_{\ell}^{(r)}(r\theta)-\psi^{(r)}(r\theta)\Big)=o(r^2)
\end{align}
where, for $s\in (0,1)$, $q^{(r)}(\theta,s)$ is defined by
\begin{align}\label{eq:qrr}
  q^{(r)}(\theta, s) = \sum_{k\in \sr{E}} \lambda^{(r)}_{k} \eta_{k}(\theta_{k},s) + \sum_{k\in \sr{K}}
        \alpha^{(r)}_{k} \xi_{k}(\theta,s), \qquad \theta \in \dd{R}^{K}.
\end{align}
\end{proposition}
We call \eq{pre-bar0} an asymptotic BAR of $X^{(r)}$. The proof of this proposition is lengthy and complicated because it requires SSC under the Palm distributions. We defer it  to \sectn{deriving-BAR}.

\item [(Step 2)]  To rewrite \eq{pre-bar0} in more tractable form, we prepare asymptotic expansions of $\eta_{k}$ and $\xi_{k}$. Namely, uniformly  bound $\eta_{k}(r\theta_{k},r^{1-\epsilon_{0}})$ for $k \in \sr{E}$ and $\xi_{k}(r\theta,r^{1-\epsilon_{0}})$ for $k \in \sr{K}$ by linear functions of $\theta_{k}$ and $\theta$, respectively, and expand them as quadratic functions $\eta_{k}^{*}(r\theta_{k})$ of $\theta_{k}$ and $\xi_{k}^{*}(r\theta)$ of $\theta$, respectively, plus $o(r^{2})$, where they are defined as
\begin{align}
\label{eq:eta-xi}
 & \eta^{*}_{k}(\theta_{k}) = \ol{\eta}_{k}(\theta_{k}) + \widetilde{\eta}_{k}(\theta_{k}), \quad k \in \sr{E}, \qquad \xi^{*}_{k}(\theta) = \ol{\xi}_{k}(\theta_{k}) + \widetilde{\xi}_{k}(\theta_{k}), \quad k \in \sr{K},
\end{align}
where 
\begin{align}
  \label{eq:eta1}
 & \ol{\eta}_{k}(\theta_{k}) = \theta_{k}, \qquad \widetilde{\eta}_{k}(\theta_{k}) = \frac 12 c_{e,k}^{2} \theta_{k}^{2}, \qquad k \in \sr{E}, \\
\label{eq:xi1}
 & \ol{\xi}_{k}(\theta) = - \theta_{k} + \sum_{\ell \in \mathcal{K}} P_{k,\ell} \theta_{\ell}, \qquad k \in \sr{K},\\
\label{eq:xi2}
 & \widetilde{\xi}_{k}(\theta) = \frac 12 \left( \sum_{\ell \in \sr{K}} P_{k,\ell} \theta_{\ell}^{2} - \left( \sum_{\ell \in \mathcal{K}} P_{k,\ell} \theta_{\ell}\right)^{2} + c_{s,k}^{2} \left(-\theta_{k} + \sum_{\ell \in \mathcal{K}} P_{k,\ell} \theta_{\ell}\right)^{2}\right), \qquad k \in \sr{K}.
\end{align}

\item [(Step 3)] Using the results obtained in Step 2, we replace $\eta_{k}(r\theta_{k},r^{1-\epsilon_{0}})$ and $\xi_{k}(r\theta,r^{1-\epsilon_{0}})$ in \eq{pre-bar0} by $\eta_{k}^{*}(\theta_{k})$ and $\xi_{k}^{*}(\theta)$ of $\theta$. This yields, for $\theta \in \Theta \equiv \dd{R}^{\sr{L}}_-\times \dd{R}^{\sr{H}}$ (see \eq{Theta}),
\begin{align}
\label{eq:pre-tbar3}
  &  q^{*}(r\theta,r)  \psi^{(r)}(r\theta) - \sum_{\ell \in \mathcal{L}} \mu^{(r)}_{\ell} \xi^{*}_{\ell}(r\theta) \beta^{(r)}_{\ell} \left(\psi_{\ell}^{(r)}(r\theta)-\psi^{(r)}(r\theta)\right) \nonumber \\
 &  \quad  +\sum_{\ell\in\mathcal{H}} \left(\mu^{(r)}_{\ell-} \xi^{*}_{\ell-}(r\theta) - \mu^{(r)}_{\ell} \xi^{*}_{\ell}(r\theta)\right)  \beta^{(r)}_{\ell} \left(\psi_{\ell}^{(r)}(r\theta)-\psi^{(r)}(r\theta)\right)=o(r^2),
\end{align}
where
\begin{align}\label{eq:qstar}
  q^{*}(\theta,r) = \sum_{k\in \sr{E}} \lambda^{(r)}_{k} \eta^{*}_{k}(\theta_{k}) + \sum_{k\in \sr{K}}
        \alpha^{(r)}_{k} \xi^{*}_{k}(\theta), \quad \theta\in \R^K.
\end{align}
\item [(Step 4)] Note that $\phi^{(r)}(r\theta)$ is a Laplace transform for $\theta \in \dd{R}^K_{-}$  and is bounded by 1. Hence, by a standard ``diagonal argument,'' we have the following.
\begin{lemma}[Theorem 5.19 in \cite{Kall2001}]
\label{lem:limits}
For any sequence in $(0,1]$ that goes to $0$,  there exists a subsequence $r_n\to 0$, and Laplace transforms $\phi(\theta)$ and $\phi_{k}(\theta)$, $k\in \mathcal{K}$, of finite measures such that
\begin{align}
\label{eq:limit}
 & \lim_{n\to\infty}    \big(\phi^{(r_n)}(r_{n}\theta), \phi^{(r_n)}_{k}(r_{n}\theta), k \in \sr{K} \big) =\big(\phi(\theta), \phi_{k}(\theta), k\in \sr{K}\big), \qquad \theta \in \dd{R}_{-}^{K}.
\end{align}
\end{lemma}
We call $(\phi, \phi_{k}, k\in \sr{K})$ in (\ref{eq:limit}) a
\emph{limit point}.   The limit point may depend on the original
sequence in $(0,1]$. Thus, there could be multiple limit points.
 Following the general theory, each component of a
limit point $(\phi(\theta),\phi_{k}(\theta), k \in \sr{K})$ for
$\theta \in \dd{R}_{-}^{K}$ is not necessarily the Laplace transform
of a probability measure.
  
\item [(Step 5)] Using the moment-SSC \assum{m-ssc2} and the expansions in Step 2, we prove the following.
\begin{lemma}
  \label{lem:mgf-ssc}
  Under  heavy-traffic \assum{ht} and moment-SSC \assum{m-ssc2}, the following 
  transform SSCs hold. For each $\theta\in \Theta$,
\begin{align}
    \label{eq:tphi-ssc1}
 & \lim_{r \downarrow 0} \left(\psi^{(r)}(r\theta) - {\phi}^{(r)}(\theta_{L},0)\right) = 0, \quad \lim_{r \downarrow 0} \left({\psi}^{(r)}_k (r\theta) - \phi^{(r)}_k(\theta_{L},0)\right) = 0, \quad k\in \mathcal{K}.
\end{align}
\end{lemma}
In the remainder of this step and Step 6 below, we use
$(\phi, \phi_{k}, k\in \sr{K})$ to denote a fixed limit point with
corresponding subsequence $\{r_{n}\}$.
For notational simplicity, we omit index $n$ and simply write $r_{n}$ as $r$
in both $r \downarrow 0$ and $o(r^2)$. This does not cause any problems because the
subsequence $\{r_{n}; n \ge 1\}$ is fixed once it is chosen.

By Lemmas
\lemt{limits} and \lemt{mgf-ssc}, we have
\begin{align}
  \lim_{r \downarrow 0} \psi^{(r)}(r\theta) = \phi(\theta_{L},0), \qquad \lim_{r \downarrow 0} \psi^{(r)}_k (r\theta) = \phi_k(\theta_{L},0).\label{eq:psi-ssc}
\end{align}
Hence, we can replace $\psi^{(r)}(r\theta)$ and $\psi^{(r)}_{\ell}(r\theta)$ in \eq{pre-tbar3} by ${\phi}(\theta_{L},0)$ and ${\phi}_{k}(\theta_{L},0)$, which yields
\begin{align}
\label{eq:pre-bar2}
  &  q^{*}(r\theta,r) \phi(\theta) - \sum_{\ell \in \mathcal{L}} \mu^{(r)}_{\ell} \xi^{*}_{\ell}(r\theta) \beta^{(r)}_{\ell} \left(\phi_{\ell}(\theta_{L},0_{H}) - \phi(\theta_{L},0_{H})\right) \nonumber \\
 &  +\sum_{\ell\in\mathcal{H}} \left(\mu^{(r)}_{\ell-} \xi^{*}_{\ell-}(r\theta) - \mu^{(r)}_{\ell} \xi^{*}_{\ell}(r\theta)\right)  \beta^{(r)}_{\ell} \left(\phi_{\ell}(\theta_{L},0_{H}) - \phi(\theta_{L},0_{H})\right)=o(r^2), \; \theta \in \Theta.
\end{align}

\item [(Step 6)] Suitably choosing $\theta_{H}$, we can remove the second summation in \eq{pre-bar2} as $r \downarrow 0$. In this way, we prove the next lemma.
\begin{lemma}
\label{lem:limiting-bar1}
Assume Assumptions~\ref{assum:moment}--\ref{assum:R}. Then each limit point $(\phi, \phi_{k}, k\in \sr{K})$ in \eq{limit} satisfies
  \begin{align}
\label{eq:bar-limit 1}
   & q\big(\theta_L,\theta_H\big) \phi(\theta_{L},0)  
   + \sum_{\ell \in \mathcal{L}} b_{\ell} \br{\theta_{L}, {R}^{(\ell))}} \Big(\phi_{\ell}(\theta_{L},0) - \phi(\theta_{L},0)\Big) = 0, \quad \theta_{L} \in \dd{R}_{-}^{\sr{L}},
  \end{align}
where $q$ is defined by \eq{q1} and  $\theta_H  = - (A_{H}^{-1})' (A_{LH})' \theta_L$. 
\end{lemma}
Obviously, this lemma yields the BAR \eq{rbm-bar1} in \lem{rbm-bar}, where $\Sigma$ is
  determined through \eq{Sigma}. Furthermore, $\phi(\theta_{L},0_{H})$ and $\phi_{k}(\theta_{L}.0_{H})$ are the Laplace transforms of unique probability measures by \lem{tight1}. Denote those probability measures by $\nu$ and $\nu_\ell$, $\ell\in \mathcal{L}$, respectively. Then, they are uniquely stationary distributions of SRBM with $(R, \Sigma, b)$ since $R$ is assumed to be completely $\sr{S}$ and $\Sigma$ is assumed to be non-degenerate. This completes the proof of \thm{main}.
\end{mylist}

In what follows, we detail Steps 2--6, including the proofs of Lemmas \lemt{mgf-ssc} and \lemt{limiting-bar1}, while Step 1 is proved in \sectn{deriving-BAR}.

\subsection{Expansions of $\eta_{k}, \xi_{k}$ and bounds for MGFs (Step 2)}
\label{sec:Step 2}

In moving from MGFs $\psi^{(r)}(r\theta), \psi^{(r)}_{k}(r\theta)$ to MGFs $\phi^{(r)}(r\theta), \phi^{(r)}_{k}(r\theta)$, we need to well control the extra terms involving $R^{(r)}_{e}, R^{(r)}_{s}$ in $\psi^{(r)}(r\theta), \psi^{(r)}_{k}(r\theta)$ so that they are ignorable as $r \downarrow 0$. For this, we bound and expand $\eta_{k}(\theta_{k},{r^{1-\varepsilon_{0}}})$ and $\xi_{k}(\theta,{r^{1-\varepsilon_{0}}})$.

\begin{lemma}
\label{lem:eta-xi asymp1}
For each fixed $a > 0$,
there are positive constants $d_{e,a}$ and $d_{s,a}$ such that for any $r \in (0,1]$ and any $\theta\in \R^K$ with $\abs{\theta}< a$
\begin{align}
\label{eq:eta-xi bound1}
 & |\eta_{k}(\theta_{k},{r^{1-\varepsilon_{0}}})| \le d_{e,a} |\theta_{k}|, \quad k \in \sr{E}, \qquad |\xi_{k}(\theta_{k},{r^{1-\varepsilon_{0}}})| \le d_{s,a} |\theta|, \quad k \in \sr{K},
\end{align}
where $|\theta| = \sum_{k \in \sr{K}} |\theta_{k}|$. 
\end{lemma}

For the expansions, recall the definitions \eq{eta-xi} in Step 2. Then, Taylor expansions for $\eta(r\theta, r^{1-\varepsilon_{0}}) $ and $\xi(r\theta, r^{1-\varepsilon_{0}})$ as $r\to 0$ are obtained as follows.

\begin{lemma}
\label{lem:eta-xi asymp2}
For each fixed $\theta \in \dd{R}^{K}$, as $r \downarrow 0$,
\begin{align}
\label{eq:eta asymp1}
 & \eta_{k}(r\theta_{k},{r^{1-\varepsilon_{0}}}) = \eta^{*}_{k}(r\theta_k) + o(r^{2}), \qquad k \in \sr{E},\\
\label{eq:xi asymp1}
 & \xi_{k}(r\theta, {r^{1-\varepsilon_{0}}}) = \xi^{*}_{k}(r\theta) + o(r^{2}), \qquad k \in \sr{K}.
\end{align}
\end{lemma}

These two lemmas are essentially the same as Lemmas 4.2 and 4.3 in \cite{BravDaiMiya2017}, but the results are notationally much simplified. For completeness and easy reference, these two lemmas will be proved in \app{P-expansion}.
We observe  that when all distributions are exponential, Taylor expansions
for $\eta_k(r\theta)$ and $\xi_k(r\theta)$ \eq{2s-taylor-1}--\eq{2s-taylor-2} are identical to the ones given in  (\ref{eq:eta asymp1}) and (\ref{eq:xi asymp1}), respectively. 

To bound $f^{(r)}_{\theta,r,{r^{1-\varepsilon_{0}}}}$ of \eq{fr-1}, we rewrite it as
\begin{align}
\label{eq:fr-2}
  f^{(r)}_{\theta,r,{r^{1-\varepsilon_{0}}}}(x)= g_{\theta, r}(z) e^{-\Lambda^{(r)}_{\theta, {r^{1-\varepsilon_{0}}}}(u, v)},\quad \text{ for } x=(z, u, v) \text{ and } r\in (0,1],
\end{align}
where
\begin{align}
\label{eq:Lambda1}
  \Lambda^{(r)}_{\theta,{r^{1-\varepsilon_{0}}}}(u,v) = \br{\eta(\theta,r^{1-\varepsilon_0}),\lambda^{(r)} u \wedge {r^{\varepsilon_{0}-1}}} + \br{\xi(\theta,r^{1-\varepsilon_0}), \mu^{(r)} v \wedge {r^{\varepsilon_{0}-1}}}.
\end{align}
Then, we have the following facts.

\begin{lemma}
\label{lem:basic1}
For each $\theta \in \Theta$, we have
\begin{align}
\label{eq:g-bound1}
  & g_{r\theta,r}(z) \le e^{|\theta_{H}|}, \qquad \forall z \in \dd{Z}_{+}^{K} \text{ and }  \forall r\in (0,1],
\end{align}
and for each $a>0$,
\begin{align}  
\label{eq:Lambda-bound1}
  |\Lambda^{(r)}_{r\theta,{r^{1-\varepsilon_{0}}}}(u,v)| &
 \le |\theta| \Big( d_{e,a} \sum_{\ell \in \sr{E}} (r \lambda^{(r)}_{\ell} u_\ell \wedge {r^{\varepsilon_{0}}}) + d_{s,a} \sum_{\ell \in \sr{K}} (r\mu^{(r)}_{\ell} v_\ell \wedge {r^{\varepsilon_{0}}}) \Big)  \\
\label{eq:Lambda-bound2}
  & \le {r^{\varepsilon_{0}}} \abs{\theta} (d_{e,a }E + d_{s,a}K), \quad 
 \forall x \in S, \forall r\in (0,1] \text{ with }  r|\theta| < a,
\end{align}
where we recall that $d_{e,a}$ and $d_{s,a}$ are constants in \lem{eta-xi asymp1}, $ E=\abs{\mathcal{E}}$, and $K =\abs{\mathcal{K}}$.  Hence, for each fixed $\theta\in \Theta$ and $a>0$,
\begin{align}\label{eq:Lambda-bound3}
 f^{(r)}_{r\theta,r,{r^{1-\varepsilon_{0}}}}(x) & \le
  e^{|\theta_{H}|+ |\theta| (d_{e,a} E + d_{s,a} K)} \nonumber\\
 &  \text{ for  $x=(z,u,v) \in S$  and $r\in (0,1]$ with $\abs{\theta}\le a$}.
\end{align}
\end{lemma}
\begin{proof}
  \eq{g-bound1} is immediate from the definition \eq{tg1} of $g_{\theta,s}(x)$ for $s=r$. To prove \eq{Lambda-bound1}, we apply \lem{eta-xi asymp1} to the definition \eq{Lambda1}, and then for any $r\in (0,1]$,
  \begin{align*}
  \left|\Lambda^{(r)}_{r\theta,{r^{1-\varepsilon_{0}}}}(u,v) \right| & \le r |\theta| \Big( d_{e,a} \sum_{\ell \in \sr{E}} (\lambda^{(r)}_{\ell} u_\ell \wedge {r^{\varepsilon_{0}-1}}) + d_{s,a} \sum_{\ell \in \sr{K}} (\mu^{(r)}_{\ell} v_\ell \wedge {r^{\varepsilon_{0}-1}}) \Big),
\end{align*}
which proves \eq{Lambda-bound1} and \eq{Lambda-bound2}.
Finally, the bound \eq{Lambda-bound3} on $f^{(r)}_{r\theta, r, r}$ immediately follows from \eq{g-bound1} and \eq{Lambda-bound2}.
\end{proof}

To prove \pro{a-bar}, we also use the following lemma, which is a direct consequence of \eq{Re2} and \eq{Rs2} in \lem{R} of Section~\ref{sec:BAR-general}.
\begin{lemma}\label{lem:R2}
  Assume Assumption~\ref{assum:moment} and Assumption~\ref{assum:stable}.
  For each $\ell\in \mathcal{E}$ and $k\in \mathcal{K}$,
  $\{R^{(r)}_{e,\ell}, r\in (0,1]\}$ and   $\{R^{(r)}_{s,k} 1(Z^{(r)}_k>0, Z^{(r)}_{H_+(k)}=0), r\in (0,1]\}$
  are uniformly integrable.
\end{lemma}

\subsection{Tractable BAR (Step 3) and limit points (Step 4)}
\label{sec:Step 3-3}

For Step 3, we first bound $\psi^{(r)}(r\theta)$ and $\psi^{(r)}_{\ell}(r\theta)$ for each $\theta \in \Theta$ by the deterministic bound in \eq{Lambda-bound3}. Namely,
\begin{align}
\label{eq:psi-bound1}
  \max(\psi^{(r)}(r\theta), \psi^{(r)}_{k}(r\theta), k \in \sr{K}) \le e^{|\theta_{H}|+ |\theta| (d_{e,a} E + d_{s,a} K)}, \;\; r\in (0,1], \theta \in \{\zeta \in \Theta; \abs{\zeta}\le a\}.
\end{align}
Then, \eq{pre-tbar3} is obtained from the asymptotic BAR \eq{pre-bar0} by applying Lemmas \lemt{eta-xi asymp2} and \lemt{basic1}.  In Step 4, \lem{limits} shows the existence of the limit points $\phi(\theta)$ and $\phi_{k}(\theta)$ for $\phi^{(r)}(r\theta)$ and $\phi^{(r)}_{k}(r\theta)$ for $\theta \in \dd{R}_{-}^{K}$. Note that $\phi(\theta)$ and $\phi_{k}(\theta)$ are the Laplace transforms of finite measures  but may not be the Laplace transform of probability measures. 

\subsection{State space collapse (Step 5)}
\label{sec:ssc}

Using the moment-SSC \assum{m-ssc2}, we prove a version of transform-SSC.
\begin{lemma}
\label{lem:limits2}
Under \assum{m-ssc2}, for each $\theta=(\theta_L, \theta_H)\in \R^K_-$,
\begin{align}
    \label{eq:limit3}
  & \lim_{r \downarrow 0} \big(\phi^{(r)}(r\theta) - \phi^{(r)}(r\theta_{L},0)\big) = 0, \qquad \lim_{r \downarrow 0} \big(\phi^{(r)}_{k}(r\theta) - \phi^{(r)}_{k}(r\theta_{L},0)\big) = 0, \quad k\in \mathcal{K}.
    \end{align}
Furthermore, each limit point $(\phi, \phi_{k}, k\in \mathcal{K})$
    satisfies 
\begin{align}
\label{eq:limit2}
 & \phi(\theta) = \phi(\theta_{L},0_{H}), \qquad \phi_{k}(\theta) = \phi_{k}(\theta_{L},0_{H}), \quad k \in \sr{K}, \qquad \theta \in \dd{R}_{-}^{K}.
\end{align}
\end{lemma}
\begin{proof}
    We prove \eq{limit3}. Fix a $\theta=(\theta_L, \theta_H)\in \Theta$. 
One can verify that 
\begin{align*}
  |g_{r(\theta_L, 0),r}(z) - g_{r\theta,r}(z)| & = \exp(\br{\theta_{L},rz_{L}} ) \left|\exp( \br{\theta_{H}, rz_{H} \wedge 1}) - 1\right|\\
                                             & {\le} e^{\abs{ \br{\theta_{H}, rz_{H} \wedge 1}}} \abs{ \br{\theta_{H}, rz_{H} \wedge 1}} \le e^{\abs{\theta_H}} \abs{\theta_H} \abs{rz_{H} \wedge 1} \\
  & \le r e^{\abs{\theta_H}} \abs{\theta_H} \abs{z_{H}},
\end{align*}
where the first equation follows from
 $\theta_L \le 0$ and
\begin{align}\label{eq:exinq}
  \abs{e^{x}-1} \le e^{\abs{x}} \abs{x} \quad \text{ for $x \in \dd{R}$}.
\end{align}
Hence, it follows  by Assumption~\ref{assum:m-ssc2}  that 
\begin{align*}
  |\phi^{(r)}(r\theta_L, 0) -\phi^{(r)}(r\theta))| & \le r e^{\abs{\theta_H}}| \theta_{H}| \sum_{\ell \in \sr{H}} \dd{E}\left[Z^{(r)}_{\ell}\right] \to 0, \qquad r \downarrow 0.
\end{align*}
Thus, the first equation of \eq{limit3} is obtained. For the second equation, we first consider it for $k \in \sr{H}$. Similarly to the case of the first equation, we have
\begin{align*}
  |\phi^{(r)}_{k}(r\theta_L,0) - \phi^{(r)}_{k}(r\theta))| & \le r e^{\abs{\theta_H}} |\theta_{H}| \sum_{\ell \in \sr{H}} \dd{E}\left[Z^{(r)}_{\ell} \mid  Z^{(r)}_{H(k)} = 0 \right]\\ 
  & \le \frac{r e^{\abs{\theta_H}}|\theta_{H}| }{\dd{P}(Z^{(r)}_{H(k)} = 0)} \sum_{\ell\in \mathcal{H}}\dd{E}\left[Z^{(r)}_{\ell} 1(Z^{(r)}_{H(k)} =0 )\right] \to 0,
\end{align*}
because $\dd{E}[Z^{(r)}_{\ell}]$ is uniformly bounded for each $\ell \in \sr{H}$ and $\dd{P}(Z^{(r)}_{H(k)} = 0) = \beta^{(r)}_{k} \to \beta_{k} > 0$ as $r \downarrow 0$ for $k \in \sr{H}$ by \lem{idle-prob}. We next consider the case for $k \in \sr{L}$. Because $\beta^{(r)}_{k} = r b_{k} + o(r)$ by \lem{mubetar}, we need to show that
\begin{align}
\label{eq:m-ssc 0}
  \lim_{r \downarrow 0} \dd{E}\left[Z^{(r)}_{\ell} 1(Z^{(r)}_{H(k)} =0 )\right] = 0, \qquad \ell \in \sr{H}, k \in \sr{L}.
\end{align}
For this, we use the following inequality:
\begin{align*}
   \dd{E}\big[Z^{(r)}_{\ell} 1(Z^{(r)}_{H(k)} =0 )\big] & = \dd{E}\big[Z^{(r)}_{\ell} 1(Z^{(r)}_{\ell} > r^{-1/2} Z^{(r)}_{H(k)} =0 )\big] +  \dd{E}\big[Z^{(r)}_{\ell} 1(Z^{(r)}_{\ell} \le r^{-1/2} Z^{(r)}_{H(k)} =0 )\big]\\
   & \le \dd{E}\big[Z^{(r)}_{\ell} 1(Z^{(r)}_{\ell} > r^{-1/2})\big] +  r^{-1/2} \dd{P}\big[Z^{(r)}_{H(k)} =0\big].
\end{align*}
This proves \eq{m-ssc 0} because of \assum{m-ssc2}. Thus, the second equation in \eq{limit3} is obtained for all $k \in \sr{K}$. Finally, \eq{limit2} is immediate from \eq{limit3}.
\end{proof} 

Lemma~\ref{lem:limits2} is the first version of transform-SSC. We extend it to the MGF $\psi^{(r)}(r\theta) \equiv \dd{E}[f^{(r)}_{\theta}(X^{(r)})]$, which is \lem{mgf-ssc}. Recall that $ \Theta \equiv \dd{R}_{-}^{L} \times \dd{R}^{H}$ and we use $\theta=(\theta_L, \theta_H)$ to denote an element in $\R^K=\R^L \times \R^H$ following convention (\ref{eq:zLH}). 

\begin{proof}[Proof of \lem{mgf-ssc}]
We first prove that
\begin{align}
\label{eq:limits5}
  \lim_{r \downarrow 0} (\psi^{(r)}(r\theta) -\phi^{(r)}(r\theta)) = 0, \qquad \lim_{r \downarrow 0} (\psi^{(r)}_{k}(r\theta) -\phi^{(r)}_{k}(r\theta)) = 0, \qquad \theta \in \Theta.
\end{align}
Fix a $\theta=(\theta_L, \theta_H)\in \Theta$. Applying \lem{basic1} to the expression (\ref{eq:fr-2}) of $f_{r\theta,r,{r^{1-\varepsilon_{0}}}}$, we have
\begin{align*}
  &  \abs{ f_{r\theta, r, {r^{1-\varepsilon_{0}}}}(x)-g_{r\theta, r}(z)} = g_{r\theta, r}(z) \abs{e^{-  \Lambda^{(r)}_{r\theta,{r^{1-\varepsilon_{0}}}}(u,v) }-1} 
    \le e^{\abs{\theta_H}}\abs{  \Lambda^{(r)}_{r\theta,{r^{1-\varepsilon_{0}}}}(u,v)} e^{\abs{ \Lambda^{(r)}_{r\theta,{r^{1-\varepsilon_{0}}}}(u,v)}} \\
  & \le {r^{\varepsilon_{0}}}|\theta|  (d_{e,a} E + d_{s,a} K) e^{\abs{\theta_H}+ (d_{e,a} E + d_{s,a} K){r^{\varepsilon_{0}}}|\theta|}
\end{align*}
for $r\in (0,1]$ satisfying  $r\abs{\theta}\le a$,
where the first inequality follows from \eq{g-bound1} and \eq{exinq}, the second from \eq{Lambda-bound2}. Since $r|\theta| \le a$ is satisfied for any $\theta \in \Theta$ and any $a > 0$ for sufficiently small $r > 0$, the first equation of \eq{limits5} is immediate from the above inequality. Similarly, the second equation of \eq{limits5} is obtained from
\begin{align*}
 |\psi^{(r)}_{k}(r\theta) -\phi^{(r)}_{k}(r\theta)| & \le \dd{E}[|f_{r\theta, r, {r^{1-\varepsilon_{0}}}}(X^{(r)}) - g_{r\theta, r}(Z^{(r)})| 1(Z^{(r)}_{H(k)} = 0)] / \dd{P}[Z^{(r)}_{H(k)} = 0] \\
 & \le {r^{\varepsilon_{0}}}|\theta|  (d_{e,a} E + d_{s,a} K) e^{\abs{\theta_H}+ (d_{e,a} E + d_{s,a} K){r^{\varepsilon_{0}}}|\theta|}.
\end{align*}
Combining \eq{limits5} with \eq{limit3} of \lem{limits2} proves \eq{tphi-ssc1} of \lem{mgf-ssc}.
\end{proof} 

\subsection{BAR for class $\sr{L}$ (Step 6)}
\label{sec:limiting-bar1}

In this section, we prove \lem{limiting-bar1}, which is composed of two parts.
We first derive a limit BAR for classes in $\mathcal{H}$, which is \lem{pre-bar-tg4} below. We then complete the proof of
\lem{limiting-bar1}.

\begin{lemma}\rm
\label{lem:pre-bar-tg4}
The limit point $(\phi, \phi_{k}, k\in \sr{K})$ in \eq{limit} satisfies the following equations.
\begin{align}
\label{eq:pre-bar-tg4}
 & \sum_{\ell\in\mathcal{H}} \left(\mu_{\ell-} \ol{\xi}_{\ell-}(\theta) - \mu_{\ell} \ol{\xi}_{\ell}(\theta)\right)  \beta_{\ell} \left(\phi_{\ell}(\theta_{L},0_{H}) - \phi(\theta_{L},0_{H})\right) = 0, \qquad \theta \in \Theta,\\
\label{eq:phi-ssc}
 & \phi_{k}(\theta_{L},0_{H}) = \phi(\theta_{L},0_{H}), \qquad k \in \sr{H}, \quad \theta_{L} \in \dd{R}_{-}^{\mathcal{L}}.
\end{align}
\end{lemma}

\begin{proof}
From the definition \eq{qstar} of $q^*(\theta, r)$ and \lem{eta-xi asymp2}, for each fixed $\theta \in \dd{R}^{K}$,
\begin{align}
\label{eq:qr-r1}
  q^*(r\theta, r) & = \sum_{k\in \mathcal{E}} \lambda^{(r)}_{k} \ol{\eta}_{k}(r\theta_{k}) + \sum_{k\in \sr{K}} \alpha^{(r)}_{k} \ol{\xi}_{k}(r\theta) + \sum_{k\in \mathcal{E}} \lambda^{(r)}_{k} \widetilde{\eta}_{k}(r\theta_{k}) + \sum_{k\in \sr{K}} \alpha^{(r)}_{k} \widetilde{\xi}_{k}(r\theta)   \nonumber\\
 & = r^{2} \sum_{k\in \mathcal{E}} \lambda^{(r)}_{k} \widetilde{\eta}_{k}(\theta_{k}) + r^{2} \sum_{k\in \sr{K}} \alpha^{(r)}_{k} \widetilde{\xi}_{k}(\theta)  = o(r),
\end{align}
because $\widetilde{\eta}_{k}(\theta_{k})$ and $\widetilde{\xi}_{k}(\theta)$ are quadratic in $\theta$ and
\begin{align*}
 & \sum_{k\in \mathcal{E}} \lambda^{(r)}_{k} \ol{\eta}_{k}(\theta_{k}) + \sum_{k\in \sr{K}} \alpha^{(r)}_{k} \ol{\xi}_{k}(\theta) = \sum_{k\in \mathcal{E}} \lambda^{(r)}_{k} \theta_{k} - \sum_{k\in \sr{K}} \alpha^{(r)}_{k} \Big(\theta_{k} - \sum_{\ell \in \sr{K}} P_{k,\ell} \theta_{\ell}\Big)\\
 & = \sum_{k\in \mathcal{E}} \lambda^{(r)}_{k} \theta_{k} - \sum_{\ell \in \sr{K}} \alpha^{(r)}_{k} \theta_{k} + \sum_{\ell \in \sr{K}} \sum_{k\in \sr{K}} \alpha^{(r)}_{k} P_{k,\ell} \theta_{\ell}\\
  & = \sum_{k\in \mathcal{E}} \lambda^{(r)}_{k} \theta_{k} - \sum_{k \in \sr{K}} \alpha^{(r)}_{k} \theta_{k} + \sum_{\ell \in \sr{K}} (\alpha^{(r)}_{\ell} - \lambda^{(r)}_{\ell}) \theta_{\ell} = 0
\end{align*}
by the definition of $\alpha^{(r)}$ in \eq{traffic1}. Since $\mu^{(r)}_{\ell} \xi^{*}_{\ell}(r\theta) \beta^{(r)}_{\ell} = o(r)$ for $\ell \in \sr{L}$ by
\eq{beta-rL} and \eq{xi asymp1}, we also have
\begin{align*}
  \sum_{\ell \in \mathcal{L}} \mu^{(r)}_{\ell} \xi^{*}_{\ell}(r\theta) \beta^{(r)}_{\ell} \left(\psi_{\ell}^{(r)}(r\theta)-\psi^{(r)}(r\theta)\right) = o(r).
\end{align*}
Hence, dividing \eq{pre-tbar3} by $r$ and letting $r \downarrow 0$, \lem{mgf-ssc} yields \eq{pre-bar-tg4} because $\mu^{(r)}_{\ell} \to \mu_{\ell}$ and $\beta^{(r)}_{\ell} \to \beta_{\ell}$ as $r \downarrow 0$ by \eq{beta-rH}.

We next use \eq{pre-bar-tg4} to prove \eq{phi-ssc}. For the
  proof, we fix a $\theta_L\in \dd{R}_{-}^{\sr{L}}$ and a class $k \in \sr{H}$. In
  the proof, we will use $\theta_L$ to construct a special
  $\theta=(\theta_L, \theta_H)\in \Theta$ that can be plugged into
  \eq{pre-bar-tg4}. Recall that $\theta_L$, $\theta_H$ and $\theta$
  are all envisioned as column vectors, even though we have adopted
  convention (\ref{eq:zLH}) in writing $\theta=(\theta_L, \theta_H)$.
For fixed $k \in \sr{H}$, define
  \begin{align}\label{eq:th1LH}
   \theta_H = - (A_{H}^{-1})' (A_{LH})' \theta_{L} +  (A_{H}^{-1})' e^{(k)}_H,
  \end{align}
  where $e^{(k)}_H$ is the $H$-vector with component $k$ being $1$ and all other components $0$.
Clearly $\theta=(\theta_L, \theta_H)\in \Theta$. We claim that
  \begin{align}
    &   \mu_{\ell-} \bar{\xi}_{\ell-}(\theta) - \mu_{\ell} \bar{\xi}_{\ell}(\theta) =0, \quad \ell\in \mathcal{H}\setminus\{k\}, \label{eq:f3}\\
   & \mu_{k-}\bar{\xi}_{k-}(\theta) - \mu_{k} \bar{\xi}_{k}(\theta) =1,  \label{eq:f4}
  \end{align}
   where recall that $k-$ is the highest class in $\{\ell \in \sr{K}; s(\ell)=s(k)\} \setminus H(k)$.
  Once this claim is verified,  \eq{phi-ssc} is immediate from
  \eq{pre-bar-tg4}.

  To verify (\ref{eq:f3}) and (\ref{eq:f4}), recall $\bar{\xi}_{k}(y)$ defined in \eq{xi1} for $k\in \sr{K}$ and $y\in \R^K$. In vector form,
  \begin{align}\label{eq:4}
    \bar{\xi}(y)=(P-I)y  \quad \text{ and } \quad   (\mu_1\bar{\xi}_1(y), \ldots,
    \mu_K \bar{\xi}_K(y))'=\diag(\mu)(P-I)y.
  \end{align}
Following from the definition of $B$ in (\ref{eq:B}),
\begin{align*}
    (\mu_{1-}\bar{\xi}_{1-}(y), \ldots, \mu_{K-}\bar{\xi}_{K-}(y))' = B'
  (\mu_1\bar{\xi}_1(y), \ldots,
  \mu_K \bar{\xi}_K(y))' =B'\diag(\mu)(P-I)y.
\end{align*}
Therefore,
\begin{align}
\label{eq:A'y}
   & (\mu_{1-}\bar{\xi}_{1-}(y), \ldots, \mu_{K-}\bar{\xi}_{K-}(y))' -
  (\mu_1\bar{\xi}_1(y), \ldots,
     \mu_K \bar{\xi}_K(y))'  =(B'-I)\diag(\mu)(P-I)y \nonumber\\
   &  = A' y
     =
   \begin{pmatrix}
     (A_L)' y_L + (A_{HL})' y_H \\
     (A_{LH})'y_L + (A_H)' y_H
   \end{pmatrix},
\end{align}
where $A$ is defined in (\ref{eq:A}) with its block decomposition defined in (\ref{eq:A-b}).
Fix $y_L=\theta_L$ and solve
\begin{align}\label{eq:5}
   (A_{LH})'y_L + (A_H)' y_H=e^{(k)}_H
\end{align}
yields unique solution $y_H=\theta_H$ in
(\ref{eq:th1LH}).  One can verify that with $y_L=\theta_L$,
(\ref{eq:5}) is equivalent to (\ref{eq:f3}) and (\ref{eq:f4}). In
solving (\ref{eq:5}), we assume $(A_H)'$ is invertible, which is 
assumed in Assumption~\ref{assum:R}.
\end{proof} 

We are now in the final step to prove \lem{limiting-bar1}.
\begin{proof}[Proof of \lem{limiting-bar1}]
  Fix a limit point $(\phi, \phi_k, k\in \mathcal{K})$ in (\ref{eq:limit}) with the corresponding subsequence $r_n\downarrow 0$ and a point
  $\theta_L\in \dd{R}_{-}^{\sr{L}}$. We would like to prove  that (\ref{eq:bar-limit 1})
  holds for this $\theta_L$.
  Recall the definition of  $\theta_{H}$ in (\ref{eq:thetaH}). Set $\theta=(\theta_L, \theta_H)$. Clearly $\theta\in \Theta$ since $\Theta$ has no sign restriction in $\theta_H$.   We now prove that the left side of (\ref{eq:pre-tbar3}), divided by $r_n^2$, goes to the left side of (\ref{eq:bar-limit 1}), proving the lemma. The left side has three terms. We now study the limit for each term.

We start with the third term.  Similar to the derivation of (\ref{eq:5}), one can check
  that
  \begin{align*}
   (A_{LH})'\theta_L + (A_H)'\theta_H=0,
  \end{align*}
  which is equivalent to 
  \begin{align}
       & \mu_{k-}\bar{\xi}_{k-}(\theta) - \mu_{k} \bar{\xi}_{k}(\theta) =0, \quad k\in \mathcal{H}. \label{eq:f5}
  \end{align}
Fix a $k\in \mathcal{H}$. Using definition \eq{eta-xi}, for each $r\in (0,1]$
\begin{align*}
 & \mu^{(r)}_{k-} \xi^{*}_{k-}(r\theta) - \mu^{(r)}_{k} \xi^{*}_{k}(r\theta)\\
 & \quad = \mu^{(r)}_{k-} (r \bar{\xi}_{k-}(\theta) + r^{2} \widetilde{\xi}_{k-}(\theta)) - \mu^{(r)}_{k} (r \bar{\xi}_{k}(\theta) + r^{2} \widetilde{\xi}_{k}(\theta)) \\
 & \quad = (\mu_{k-} + r \mu^{*}_{k-})(r \bar{\xi}_{k-}(\theta) + r^{2} \widetilde{\xi}_{k-}(\theta)) \\
 & \qquad - (\mu_{k} + r\mu^{*}_{k}) (r \bar{\xi}_{k}(\theta) + r^{2} \widetilde{\xi}_{k}(\theta)) + o(r^{2})\\
& \quad =  r^2 \mu^{*}_{k-} \bar{\xi}_{k-}(\theta) + r^{2}\mu_{k-} \widetilde{\xi}_{k-}(\theta)
 -  r^2\mu^{*}_{k} \bar{\xi}_{k}(\theta) + r^{2}  \mu_k\widetilde{\xi}_{k}(\theta) + o(r^{2}),
\end{align*}
where the second equality follows from (\ref{eq:mu-r})
and  the third equality follows from \eq{f5}.
It follows from Lemma~\ref{lem:mgf-ssc} and (\ref{eq:phi-ssc}) that the third term in the left side of (\ref{eq:pre-tbar3}), divided by $r^2$, goes to $0$ as $r \downarrow 0$.
  
Next we study the first term. From
\eq{qr-r1} and \eq{q1},
\begin{align*}
  \lim_{r \downarrow 0} r^{-2} q^{(r)}(r\theta, r) = \sum_{k\in \mathcal{E}} \lambda_{k} \widetilde{\eta}_{k}(\theta) + \sum_{k\in \sr{K}} \alpha_{k} \widetilde{\xi}_{k}(\theta) = q\big(\theta_L, -(A^{-1}_H)'(A_{LH})'\theta_L\big).
\end{align*}
This, together with
 Lemma~\ref{lem:mgf-ssc}, proves that
the first term in the left side of (\ref{eq:pre-tbar3}), divided by $r^2$, goes to
\begin{align*}
  q\big(\theta_L, \theta_{H} \big) \phi(\theta_{L},0) 
\end{align*}
as $r \downarrow 0$, recalling that $\theta_{H} = -(A_H^{-1})'A_{LH}' \theta_L$
in \eq{thetaH}.

Finally, we study the second term.
For $\ell \in \mathcal{L}$, we have
\begin{align}\label{eq:6}
  &  - \lim_{r \downarrow 0} r^{-2} \mu^{(r)}_{\ell} \xi^{*}_{\ell}(r\theta) \beta^{(r)}_{\ell} = - \lim_{r \downarrow 0} r^{-2} \mu^{(r)}_{\ell} \bar{\xi}_{\ell}(r\theta) \beta^{(r)}_{\ell} = \mu_\ell \bar{\xi}_\ell(\theta) b_{\ell}.
\end{align}
On the other hand, substituting $y=(\theta_{L},\theta_{H})$ into \eq{A'y} and choosing the entry $\ell \in \sr{L}$, we have
\begin{align*}
  \mu_\ell \bar{\xi}_\ell(\theta) =-\br{\theta_{L},R^{(\ell)}},
\end{align*}
because $\mu_{\ell-} \bar{\xi}_{\ell-}(\theta) = 0$ for $\ell \in \sr{L}$ by our convention and $R = A_{L} - A_{LH} A_{H}^{-1} A_{HL}$. Clearly, this and (\ref{eq:6}), together with
 Lemma~\ref{lem:mgf-ssc}, proves that
the second term in the left side of (\ref{eq:pre-tbar3}), divided by $r^2$, goes to
\begin{align*}
   \sum_{\ell \in \mathcal{L}} b_{\ell} \br{\theta_{L}, {R}^{(\ell))}} \Big(\phi_{\ell}(\theta_{L},0) - \phi(\theta_{L},0)\Big)
\end{align*}
as $r \downarrow 0$.
The study of these three terms leads to the conclusion that  the left side of
 the asymptotic BAR \eq{pre-tbar3}, divided by $r^2$, converges to the
left side of SRBM BAR  \eq{bar-limit 1},  which proves \lem{limiting-bar1}.
\end{proof}

\section{Deriving BARs}
\label{sec:deriving-BAR}

As we said in \sectn{outline}, the proof of \thm{main} is completed once the asymptotic BAR \eq{pre-bar0} in \pro{a-bar} is obtained. In this section, we prove \pro{a-bar}. This will be done step by step. We first derive a BAR of $X^{(r)}$ using the test function $f^{(r)}_{\theta,s,t}$ in \sectn{primitive-BAR}. From this BAR, we derive the asymptotic BAR \eq{pre-bar0}, which proves \pro{a-bar}. This will be done in \sectn{asymptotic-bar}, using three lemmas proved in Sections \sect{ssc-Palm} and \sect{two-lemmas}.

\subsection{Primitive BAR of $X^{(r)}$}
\label{sec:primitive-BAR}

We derive a BAR of $X^{(r)}$ for the test function $f^{(r)}_{\theta,s,t}$ of \eq{fr-1}. Recall that this test function can be written as
\begin{align}
\label{eq:fx1}
 &  f^{(r)}_{r\theta,r,{r^{1-\varepsilon_{0}}}}(x) = g_{r\theta,r}(z) \exp(-\Lambda^{(r)}_{r\theta,{r^{1-\varepsilon_{0}}}}(u,v)), \quad x \in S,
\end{align}
where we recall from (\ref{eq:tg1}) and \eq{Lambda1} that 
\begin{align*}
 & g_{r\theta,r}(z) = \exp(\br{r\theta_{L},z_L}+ \br{r\theta_{H}, z_{H} \wedge 1/r}),\\
 & \Lambda^{(r)}_{r\theta,{r^{1-\varepsilon_{0}}}}(u,v) = \br{\eta(r\theta, r^{1-\epsilon_0}), \lambda^{(r)} u \wedge r^{\varepsilon_{0}-1}} + \br{\xi(r\theta,r^{1-\epsilon_0}), \mu^{(r)} v \wedge r^{\varepsilon_{0}-1}}.
\end{align*}

For each $r\in (0,1]$ and each $\theta\in \Theta$,   $f^{(r)}_{r\theta,r,{r^{1-\varepsilon_{0}}}}\in \mathcal{D}$ by \lem{basic1}. This test function will be used in BAR \eq{bar}.
In what follows, $f^{(r)}_{r\theta,r,{r^{1-\varepsilon_{0}}}}$ is denoted by $f$ for simplicity.
To expand \eq{bar}, we compute the following quantities.
\begin{align*}
   \frac{\partial f}
    {\partial u_{k}}(x) & = \lambda^{(r)}_{k} \eta_{k}(r\theta_{k},{r^{1-\varepsilon_{0}}}) 1(\lambda^{(r)}_{k} u_{k} \le {r^{\varepsilon_{0}-1}}) f(x) \\
  & = \lambda^{(r)}_{k} \eta_{k}(r\theta_{k},{r^{1-\varepsilon_{0}}}) f(x) - \lambda^{(r)}_{k} \eta_{k}(r\theta_{k},{r^{1-\varepsilon_{0}}}) 1(\lambda^{(r)}_{k} u_{k} > {r^{\varepsilon_{0}-1}}) f(x) , \qquad k \in \sr{E},\\
  \frac{\partial f}{\partial v_{k}}(x) & = \mu^{(r)}_{k} \xi_{k}(r\theta,{r^{1-\varepsilon_{0}}}) 1\left(\mu^{(r)}_{k} v_{k} \le {r^{\varepsilon_{0}-1}}, z^{(r)}_{k} > 0, z^{(r)}_{H_{+}(k)} = 0\right)  \nonumber\\
   & =  \mu^{(r)}_{k} \xi_{k}(r\theta,{r^{1-\varepsilon_{0}}}) 1\left( z^{(r)}_{k} > 0, z^{(r)}_{H_{+}(k)} = 0\right) \\
  & \quad -  \mu^{(r)}_{k} \xi_{k}(r\theta,{r^{1-\varepsilon_{0}}}) 1\left(\mu^{(r)}_{k} v_{k} > {r^{\varepsilon_{0}-1}}, z^{(r)}_{k} > 0, z^{(r)}_{H_{+}(k)} = 0\right) , \qquad k \in \sr{K}.
\end{align*}
Then, it follows from \eq{A1} and \eq{bar} that
\begin{align}
\label{eq:pre-bar-tg1}
   &  \sum_{k\in \mathcal{E}} \lambda^{(r)}_{k} \eta_{k}(r\theta_{k},{r^{1-\varepsilon_{0}}}) \dd{E}\Big[f(X^{(r)})\Big] \nonumber\\
   & \quad + \sum_{k\in \sr{K}} \mu^{(r)}_{k} \xi_{k}(r\theta,{r^{1-\varepsilon_{0}}})
  \dd{E}\Big[1( Z^{(r)}_{k}>0, Z^{(r)}_{H_+(k)}=0)f(X^{(r)})\Big] + E(r,\theta,r^{1-\epsilon_{0}}) = 0,
\end{align}
where
\begin{align}
\label{eq:E-term}
  E(r,\theta,r^{1-\epsilon_{0}}) & = - \sum_{k\in \mathcal{E}} \lambda^{(r)}_{k} \eta(r\theta_{k},{r^{1-\varepsilon_{0}}}) \dd{E}\Big[1(\lambda^{(r)}_{k} R^{(r)}_{e,k} > {r^{\varepsilon_{0}-1}}) f (X^{(r)})\Big]   \nonumber\\
 & - \sum_{k\in \sr{K}} \mu^{(r)}_{k} \xi_{k}(r\theta,{r^{1-\varepsilon_{0}}})
      \dd{E}\Big[ 1(\mu^{(r)}_{k} R^{(r)}_{s,k} > {r^{\varepsilon_{0}-1}}, Z^{(r)}_{k}>0, Z^{(r)}_{H_+(k)}=0)f(X^{(r)})\Big] \nonumber\\
  & + \sum_{k\in \mathcal{E}} \lambda^{(r)}_k \E_{e,k}[\Delta f(X^{(r)}_{+},X^{(r)}_{-})] + \sum_{k\in \mathcal{K}} \alpha^{(r)}_k \E_{s,k}[\Delta f(X^{(r)}_{+},X^{(r)}_{-})].
\end{align}
Recall that $f(x)=f^{(r)}_{r\theta, r, r^{1-\varepsilon_0}}(x)$, and $\dd{E}_{e,k}$ and $\dd{E}_{s,k}$ stand for the expectations under the Palm distributions $\dd{P}_{e,k}$ and $\dd{P}_{s,k}$,  which, together with random variables $X^{(r)}_+$
and $X^{(r)}_-$,  are defined by \eq{palm_e} and \eq{palm_s} through  the counting processes $N^{(r)}_{e,k}(\cdot)$ and $N^{(r)}_{s,k}(\cdot)$, respectively. Equation \eq{pre-bar-tg1} can be considered  a BAR. We call it a primitive BAR. This BAR is the starting point of our analysis.

\subsection{Proof of \pro{a-bar}}
\label{sec:asymptotic-bar}

We aim to derive \eq{pre-bar0} from \eq{pre-bar-tg1}. Recall that $\psi^{(r)}(r\theta) = \dd{E}[f^{(r)}_{r\theta,r,{r^{1-\varepsilon_{0}}}}(X^{(r)})]$ and $\psi^{(r)}_{k}(r\theta) = \dd{E}[f^{(r)}_{r\theta,r,{r^{1-\varepsilon_{0}}}}(X^{(r)})|Z^{(r)}_{H(k)}=0]$ for $k \in \sr{K}$. If $E(r,\theta,r^{1-\epsilon_{0}})$ of \eq{E-term} has order $o(r^2)$ as $r \downarrow 0$, then we have 
\begin{align}
\label{eq:pre-bar-tg2}
  & \sum_{k\in \mathcal{E}} \lambda^{(r)}_{k} \eta_{k}(r\theta_{k},{r^{1-\varepsilon_{0}}}) \dd{E}\Big[f^{(r)}_{r\theta,r,{r^{1-\varepsilon_{0}}}} (X^{(r)})\Big]  \nonumber\\
 & + \sum_{k\in \sr{K}} \mu^{(r)}_{k} \xi_{k}(r\theta,{r^{1-\varepsilon_{0}}})
  \dd{E}\Big[1( Z^{(r)}_{k}>0, Z^{(r)}_{H_+(k)}=0)f^{(r)}_{r\theta,r,{r^{1-\varepsilon_{0}}}} (X^{(r)})\Big] = o(r^{2}).
\end{align}
Then, if \eq{pre-bar-tg2} holds, the proof of \pro{a-bar} is completed by the next lemma.

\begin{lemma}\rm
\label{lem:pre-bar0}
\eq{pre-bar-tg2} is equivalent to \eq{pre-bar0} in \pro{a-bar}.
\end{lemma}

\begin{proof}
Since
\begin{align*}
 & \dd{E}\Big[1( Z^{(r)}_{k}>0, Z^{(r)}_{H_+(k)}=0) f^{(r)}_{r\theta,r,{r^{1-\varepsilon_{0}}}}(X^{(r)})\Big] \\
  & \quad = \dd{E}\Big[(1( Z^{(r)}_{H_+(k)}=0) - 1( Z^{(r)}_{H(k)}=0)) f^{(r)}_{r\theta,r,{r^{1-\varepsilon_{0}}}}(X^{(r)}) \Big] = \beta^{(r)}_{k+} \psi^{(r)}_{k+}(\theta) - \beta^{(r)}_{k} \psi^{(r)}_{k}(\theta)
\end{align*}
by \lem{idle-prob}, we can write \eq{pre-bar-tg2} as
\begin{align}
\label{eq:pre-bar-g2}
   &  \sum_{k\in \mathcal{E}} \lambda^{(r)}_{k} \eta^{*}_{k}(r\theta_{k}) \psi^{(r)}(r\theta) + \sum_{k\in \sr{K}} \mu^{(r)}_{k} \xi^{*}_{k}(r\theta) \big(\beta^{(r)}_{k+} \psi^{(r)}_{k+}(r\theta) - \beta^{(r)}_{k} \psi^{(r)}_{k}(r\theta)\big) = o(r^{2}),
\end{align}
because $\eta_{k}(r\theta_{k},{r^{1-\varepsilon_{0}}}) = \eta^{*}_{k}(r\theta_{k}) + o(r^{2})$ and $\xi_{k}(r\theta_{k},{r^{1-\varepsilon_{0}}}) = \xi^{*}_{k}(r\theta_{k}) + o(r^{2})$ by \lem{eta-xi asymp1} and the boundedness of $\psi^{(r)}(r\theta)$ and $\psi^{(r)}_{k}(r\theta)$.

Since $\alpha^{(r)}_{k} = \mu^{(r)}_{k}(\beta^{(r)}_{k+} - \beta^{(r)}_{k})$ by (\ref{eq:beta}), 
\begin{align*}
 \sum_{k\in \sr{K}} \alpha^{(r)}_{k} \xi^{*}_{k}(\theta) \psi^{(r)}(\theta) = \sum_{k\in \sr{K}} \mu^{(r)}_{k} \xi^{*}_{k}(\theta) (\beta^{(r)}_{k+} - \beta^{(r)}_{k}) \psi^{(r)}(\theta).
\end{align*}
Hence, \eq{pre-bar-g2} can be written as
\begin{align*}
 & \sum_{k\in \mathcal{E}} \lambda^{(r)}_{k} \eta^{*}_{k}(\theta_{k}) \psi^{(r)}(\theta) + \sum_{k\in \sr{K}} \mu^{(r)}_{k} \xi^{*}_{k}(\theta) \big(\beta^{(r)}_{k+} \psi^{(r)}_{k+}(\theta) - \beta^{(r)}_{k} \psi^{(r)}_{k}(\theta)\big)\\
 &\qquad + \sum_{k\in \sr{K}} \alpha^{(r)}_{k} \xi^{*}_{k}(\theta) \psi^{(r)}(\theta) - \sum_{k\in \sr{K}} \mu^{(r)}_{k} (\beta^{(r)}_{k+} - \beta^{(r)}_{k}) \xi^{*}_{k}(\theta) \psi^{(r)}(\theta) = o(r^{2}).
\end{align*}
Hence, 
\begin{align*}
   & \Big(\sum_{k\in \mathcal{E}} \lambda^{(r)}_{k} \eta^{*}_{k}(\theta_{k}) + \sum_{k\in \sr{K}} \alpha^{(r)}_{k} \xi^{*}_{k}(\theta) \Big) \psi^{(r)}(\theta)\\
 &\qquad + \sum_{k\in \sr{K}} \mu^{(r)}_{k} \xi^{*}_{k}(\theta) \left(\beta^{(r)}_{k+} \big(\psi^{(r)}_{k+}(\theta) - \psi^{(r)}(\theta)\big) - \beta^{(r)}_{k} \big(\psi^{(r)}_{k}(\theta) - \psi^{(r)}(\theta)\big)\right)\\
 & \quad = q^{(r)}(\theta) \psi^{(r)}(\theta) + \sum_{k\in \sr{K}} \left(\mu^{(r)}_{k-} \xi^{*}_{k-}(\theta) - \mu^{(r)}_{k} \xi^{*}_{k}(\theta)\right) \beta^{(r)}_{k} \left(\psi^{(r)}_{k}(\theta) - \psi^{(r)}(\theta)\right) = o(r^{2}).
\end{align*}
This is equivalent to \eq{pre-bar0} because $\mu^{(r)}_{k-} = 0$ for $k \in \sr{L}$.
\end{proof} 

It remains to prove \eq{pre-bar-tg2}  to complete the proof of \pro{a-bar}. For this, it is sufficient to show that $E(r,\theta,r^{1-\epsilon_{0}}) = o(r^{2})$, which is proved by the following two lemmas.

\begin{lemma}
\label{lem:R-r2}
Under Assumptions \assumt{moment} and \assumt{stable}, for each $\theta \in \dd{R}_{-}^{L} \times \dd{R}^{H}$,
\begin{align}
\label{eq:eta-Re-r2}
 & \lambda^{(r)}_k\eta_k(r\theta_{k},{r^{1-\varepsilon_{0}}}) \dd{E}\Big[1(\lambda^{(r)}_{k} R^{(r)}_{e,k} > {r^{\varepsilon_{0}-1}}) f^{(r)}_{r\theta,r,{r^{1-\varepsilon_{0}}}} (X^{(r)})\Big] = o(r^{2}), \qquad k \in \sr{E},\\
\label{eq:xi-Rs-r2}
 & \mu^{(r)}_k\xi(r\theta,{r^{1-\varepsilon_{0}}}) \dd{E}\Big[1(\mu^{(r)}_{k} R^{(r)}_{s,k} > {r^{\varepsilon_{0}-1}}, Z^{(r)}_{k}>0, Z^{(r)}_{H_+(k)}=0) f^{(r)}_{r\theta,r,{r^{1-\varepsilon_{0}}}} (X^{(r)})\Big] = o(r^{2}), \quad k \in \sr{K}.
\end{align}
\end{lemma}

\begin{lemma}
  \label{lem:boundary3}
  Fix $\theta\in \Theta$ and  let $f=f^{(r)}_{r\theta, r, {r^{1-\varepsilon_{0}}}}$.
  \begin{align}
    \label{eq:Pe-k-r2}
    & \dd{E}_{e,k}\left[\Delta f(X^{(r)}_{+},X^{(r)}_{-})\right] = o(r^{2}), \quad k\in \mathcal{E}\\
\label{eq:Ps-k-r2}
    &       \dd{E}_{s,k}\left[\Delta f(X^{(r)}_{+},X^{(r)}_{-})\right]
      = o(r^{2}), \quad k\in \mathcal{K}.
\end{align}
\end{lemma}

Before proving these two lemmas, we prepare one lemma,  the SSC of $Z^{(r)}_{-,H}$ under Palm distributions, which will be used to prove \lem{boundary3}, where $Z^{(r)}_{-,H}$ is the $H$-dimensional random vector whose $k$th entry is $Z^{(r)}_{-,k}$ for $k \in \sr{H}$. This SSC is  itself interesting because it is not immediate from the SSC of $Z^{(r)}$ under $\dd{P}$. So, we verify it in \sectn{ssc-Palm}, separate from the proofs of Lemmas \lemt{R-r2} and \lemt{boundary3}. Finally,  these two lemmas are proved in \sectn{two-lemmas}. 

\subsection{SSC under Palm distributions}
\label{sec:ssc-Palm}

We prove the SSC of $Z^{(r)}_{-,H}$ under Palm distributions $\dd{P}_{e,k}$ and $\dd{P}_{s,k}$. 
\begin{lemma}\rm
\label{lem:ssc-Palm}
Under Assumptions \assumt{moment}--\assumt{m-ssc2}, for $\ell \in \mathcal{H}$,
  \begin{align}
    \label{eq:Ze-ssc1}
 & \dd{P}_{e,k}\{Z^{(r)}_{-, {\ell}} > r^{-1}-1\}=o(r), \qquad k \in \sr{E},\\
 & \label{eq:Zs-ssc1}
    \dd{P}_{s,k}\{Z^{(r)}_{-,{\ell}} > r^{-1}-1\}=o(r), \qquad k \in \sr{K}.
  \end{align}
\end{lemma}
  
The proof of this lemma requires the next lemma, which relates the tail probabilities of $Z^{(r)}$ under the Palm distributions to those under $\dd{P}$.

\begin{lemma}
\label{lem:Zes}
For each integer $n\ge 1$, each $r \in (0,]$, each $c\in \R_+$ and $k, \ell \in \sr{K}$,
\begin{align}
\label{eq:Rskz}
    &\Prob(Z^{(r)}_k\ge n, Z^{(r)}_{H_+(k)}=0) = \gamma_k[(1-P_{kk}) \Prob_{s,k}(Z^{(r)}_{-,k}\ge n+1)+ P_{kk}\Prob_{s,k}(Z^{(r)}_{-,k}\ge n)] \nonumber \\
    & \quad + 
      \lambda_k \E_{e,k}[R^{(r)}_{-,s,k}1(Z^{(r)}_{-,k}=n-1)] + \sum_{\ell\in \mathcal{K}\setminus\{k\}} \alpha_{\ell} P_{\ell k} \E_{s,\ell}[R^{(r)}_{-,s,k}1(Z^{(r)}_{-,k}=n-1)],
  \end{align} 
\vspace{-4ex}
\begin{align}
\label{eq:Rsklz}
  & \dd{P}(R^{(r)}_{s,k} \le c, Z^{(r)}_\ell\ge n, Z^{(r)}_{k} > 0, Z^{(r)}_{H_+(k)}=0) + \alpha_\ell (1-P_{\ell\ell})
    \E_{s,\ell}[(R^{(r)}_{s,k} \wedge c) 1(Z^{(r)}_{-,\ell}=n)] \nonumber \\
    & \quad = \gamma_k    \E[T_{s,k} \wedge (c/m_{k})]  [ P_{k\ell} \Prob_{s,k}(Z^{(r)}_{-,\ell} \ge n-1)+(1-P_{k\ell}) \Prob_{s,k}(Z^{(r)}_{-,\ell} \ge n)] \nonumber \\
&  \qquad  + \lambda_\ell \E_{e,\ell}[(R^{(r)}_{s,k} \wedge c) 1(Z^{(r)}_{-,\ell}=n-1)] \nonumber\\
   & \qquad + \sum_{k'\in \mathcal{K}\setminus\{k,\ell\}} \alpha_{k'} P_{k' k}
    \E_{s,k'}[(R^{(r)}_{s,k} \wedge c) 1(Z^{(r)}_{-,\ell}=n-1)],
\end{align}
\vspace{-2ex}
\begin{align}
\label{eq:Rekz}
    \dd{P}(R^{(r)}_{e,k} \le c,Z^{(r)}_{k}\ge n) & = \E[T_{e,k} \wedge (c/a_{k})]\Prob_{e,k}(Z^{(r)}_{-,k}\ge n-1) \nonumber\\
    & \quad - \alpha_k (1-P_{kk}) \E_{s,k}[ (R^{(r)}_{e,k} \wedge c) 1(Z^{(r)}_{-,k}=n)]\nonumber\\
     & \quad + \sum_{\ell\in \mathcal{K}\setminus\{k\}} \alpha_{\ell} P_{\ell k} \E_{s,\ell}[(R^{(r)}_{e,k} \wedge c) 1(Z^{(r)}_{-,k}=n-1)],
\end{align}
\vspace{-2ex}
\begin{align}
\label{eq:Reklz}
    \dd{P}(R^{(r)}_{e,k} \le c,Z^{(r)}_{\ell}\ge n) & = \E[T_{e,k} \wedge (c/a^{(r)}_{k})]\Prob_{e,k}(Z^{(r)}_{-,\ell}\ge n-1) - \alpha_{\ell} \dd{E}_{s,\ell}[(R^{(r)}_{e,k} \wedge c) 1(Z_{-,\ell} = n)] \nonumber\\
    & \quad + \sum_{k'\in \mathcal{K}\setminus\{\ell\}} \alpha_{k'} P_{k' \ell} \E_{s,k'}[(R^{(r)}_{e,k} \wedge c) 1(Z^{(r)}_{-,\ell}=n-1)],
\end{align}
where expectations under $\dd{P}_{e,k}$ and $\dd{P}_{e,\ell}$ vanish for $k, \ell \not\in \sr{E}$.
\end{lemma}

Since the proof of this lemma is lengthy, we defer it until we prove \lem{ssc-Palm}.

\begin{proof}[Proof of \lem{ssc-Palm}]
  We first prove \eq{Ze-ssc1} and \eq{Zs-ssc1} for $\ell =k \in \sr{H}$.  For this, we use \lem{Zes} and the SSC property,
\begin{align}
\label{eq:p-ssc}
  \dd{P}[Z^{(r)}_{k} \ge n] = o(r), \qquad k \in \sr{H},
\end{align}
which is immediate from the SSC \assum{m-ssc2}. Set $n=\lfloor r^{-1}\rfloor$. Then, by \eq{Rskz},
  \begin{align}\label{eq:temp1}
   (1 - P_{k,k}) \Prob_{s,k}(Z^{(r)}_{-,k}=n) \le (1 - P_{k,k})
    \Prob_{s,k}(Z^{(r)}_{-,k}\ge n) \le \frac{1}{\gamma_k^{(r)}}\dd{P} (Z^{(r)}_{k}\ge n-1).
  \end{align}
This and \eq{p-ssc} prove \eq{Zs-ssc1} for $\ell = k \in \sr{H}$ because $1 - P_{k,k} > 0$. We next note that $\E[T_{e,k}]=1$ and $a_k^{(r)} \to a_k>0$ as $r \downarrow 0$ for $k\in \mathcal{E}$, so we can find a $r_0\in(0,1]$ and $c>0$ such that
  \begin{align}
\label{eq:Tek-bound 1}
   \E[T_{e,k} \wedge (c/a_k^{(r)})] \ge 1/2 \text{ for } r\in (0,r_0).
  \end{align}
Therefore, it follows from \eq{Rekz} that for each $r\in (0,r_0)$,
  \begin{align*}
    \Prob_{e,k}(Z^{(r)}_{-,k} \ge n-1) \le 2 \dd{P}( Z^{(r)}_{k} \ge n) + 2 c\alpha^{(r)}_k(1-P_{k,k})
     \Prob_{s,k}(Z^{(r)}_{-,k}=n).
  \end{align*}
Since \eq{Zs-ssc1} is already proved for $\ell = k \in \sr{H}$, this and \eq{p-ssc} prove \eq{Ze-ssc1} for $\ell=k$. 
  
We next prove \eq{Ze-ssc1} and \eq{Zs-ssc1} for $\ell \in \sr{H}$ and $k \in \sr{K} \setminus \{\ell\}$. This time, we pick $r_0\in (0,1]$ and $c>0$ such that
  \begin{align*}
    \E[T_{s,k} \wedge c/m_k^{(r)}]\ge 1/2, \qquad r \in (0,r_{0}).
  \end{align*}
Then, it follows from  \eq{Rsklz} and \eq{temp1} for $k=\ell$ that
  \begin{align*}
   & \Prob_{s,k}(Z^{(r)}_{-,\ell}\ge n) \le \frac{2}{\gamma_k^{(r)}} [\Prob( Z^{(r)}_{\ell} \ge n)
    + c\alpha_\ell^{(r)}(1-P_{\ell \ell}) \Prob_{s,\ell}(Z^{(r)}_{-,\ell}=n)].
  \end{align*}
Since \eq{Zs-ssc1} is already proved for $k=\ell \in \sr{H}$, this inequality and \eq{p-ssc} prove \eq{Zs-ssc1} for $\ell \in \sr{H}$ and $k \in \sr{K} \setminus \{\ell\}$. Similarly, applying \eq{Tek-bound 1} to \eq{Reklz}, we have
\begin{align*}
  \Prob_{e,k}(Z^{(r)}_{-,\ell}\ge n-1) & \le 2 \dd{P}(R^{(r)}_{e,k} \le c,Z^{(r)}_{\ell}\ge n) + 2 \alpha_{\ell} \dd{E}_{s,\ell}[(R^{(r)}_{e,k} \wedge c) 1(Z_{-,\ell} = n)]  \nonumber\\
  & \le 2 \dd{P}(Z^{(r)}_{\ell}\ge n) + 2 \alpha_{\ell} \dd{P}_{s,\ell}[Z_{-,\ell} \ge n].
\end{align*}
Hence, \eq{p-ssc} and \eq{Zs-ssc1} for $k=\ell \in \sr{H}$ prove \eq{Ze-ssc1} for $\ell \in \sr{H}$ and $k \in \sr{K} \setminus \{\ell\}$.
\end{proof}

\begin{proof}[Proof of \lem{Zes}]
In this proof, we omit the superscript $^{(r)}$ in $Z^{(r)}_{k}$, $R^{(r)}_{-,e,k}$ and $R^{(r)}_{-,s,k}$ for simplicity. So, they are written as $Z_{k}$, $R_{-,e,k}$ and $R_{-,s,k}$, respectively.
    To prove \eq{Rskz}, we fix a $k\in \mathcal{K}$ and an integer $n\ge 1$. Take $f(x)=(v_k\wedge c) 1(z_k\ge n)$. Then
\begin{align*}
    & \mathcal{A}f(X) =- (R_{s,k} \le c, Z_{k} \ge n, Z_{H_+(k)}=0), \\
    & \Delta f(X_{+},X_{-}) = (R_{-,s,k} \wedge c) 1(Z_{-,k}= n-1) 1(R_{-,e,k} =0) \\
    & \hspace{15ex} + \sum_{\ell \in \sr{K} \setminus \{k\}} (R_{-,s,k} \wedge c) 1(\Phi^{(\ell)} = e^{(k)}) 1(Z_{-,k} =n-1, R_{-,s,\ell} = 0) \\
    & \hspace{15ex} + (m_k T_{s,k} \wedge c) [ 1(\Phi^{(k)} \not= e^{(k)}) 1(Z_{-,k}-1\ge n)+1(\Phi^{(k)} = e^{(k)})  1(Z_{-,k} \ge n)] 1(R_{-,s,k} = 0).
  \end{align*}
By \eq{bar},
\begin{align*}
  &\Prob(R_{s,k} \le c, Z_k\ge n, Z_{H_+(k)}=0) \\
  & \quad = \lambda_k \E_{e,k}[(R_{-,s,k} \wedge c) 1(Z_{-,k}=n-1)] + \sum_{\ell\in \mathcal{K}\setminus\{k\}} \alpha_{\ell} P_{\ell k} \E_{s,\ell}[(R_{-,s,k} \wedge c) 1(Z_{-,k}=n-1)]\\
& \qquad + \gamma_{k} \dd{E}[T_{s,k} \wedge c/m_{k}]([(1-P_{kk}) \Prob_{s,k}(Z_{-,k} \ge n+1)+ P_{kk}\Prob_{s,k}(Z_{-,k} \ge n)].
\end{align*}
Letting $c \to \infty$ in this equality proves \eq{Rskz} because the left-hand side is bounded by $1$, and all the terms in the right-hand side are nonnegative.

  To prove \eq{Rsklz}, we fix  $k, \ell \in \mathcal{K}$ with $k\neq \ell$, an integer $n\ge 1$, and $c\in \R_+$. Take $f(x)= (v_k \wedge c) 1(z_{\ell} \ge n)$. Then
\begin{align*}
    \mathcal{A}f(X) & = - 1(R_{s,k} \le c, Z_{\ell} \ge n) 1(Z_k>0,Z_{H_+(k)}=0), \\
    \Delta f(X_{+},X_{-}) & = (R_{s,k} \wedge c) 1(Z_{-,\ell}+1=n) 1(R_{-,e,\ell}=0), \\
    & \quad - 1(\Phi^{(\ell)} \not= e^{(\ell)}) (R_{s,k} \wedge c) 1(Z_{-,\ell}=n) 1(R_{-,s,\ell}=0)\\
    & \quad + (m_k(T_{s,k} \wedge c) [1(\Phi^{(k)} = e^{(\ell)}) 1(Z_{-,\ell}+1 \ge n) \\
    & \hspace{25ex} + 1(\Phi^{(k)} \not= e^{(\ell)}) 1(Z_{-,\ell} \ge n)] 1(R_{-,s,k} = 0)\\
    & \quad + (R_{s,k} \wedge c) \sum_{k' \in  \mathcal{K} \setminus \{k,\ell\}} 1(\phi^{(k')} =e^{(\ell)}) 1(Z_{-,\ell}+1= n) 1(R_{-,s,k'} = 0).
  \end{align*}
By \eq{bar}, 
\begin{align*}
  & \Prob(R_{s,k}\le c, Z_\ell\ge n, Z_{k} > 0, Z_{H_+(k)}=0) = \lambda_\ell \E_{e,\ell}[(R_{s,k} \wedge c) 1(Z_{-,\ell} =n-1)] \\
  & -(1-P_{\ell\ell}) \alpha_\ell \E_{s,\ell}[(R_{s,k} \wedge c) 1(Z_{-,\ell} =n)]\\
    & + \gamma_k \E[T_{s,k} \wedge (c/m_{k})]  [ P_{k\ell} \Prob_{s,k}(Z_{-,\ell} \ge n-1)+(1-P_{k\ell}) \Prob_{s,k}(Z_{-,\ell} \ge n)] \\
  &    + \sum_{k'\in \mathcal{K}\setminus\{k,\ell\}} \alpha_{k'} P_{k' k}
      \E_{s,k'}[(R_{s,k} \wedge c) 1(Z_{-\ell} =n-1)],
\end{align*}
which is equivalent to \eq{Rsklz}.

To prove \eq{Rekz}, we fix  $k \in \mathcal{E}$, an integer $n\ge 1$, and $c\in \R_+$. Take $f(x)= (u_k \wedge c) 1(z_k\ge n)$. 
Then
\begin{align*}
   \mathcal{A}f(X) & =- 1(R_{e,k} \le c) 1(Z_k\ge n), \\
  \Delta f(X_{+},X_{-}) & = (a_k T_{e,k} \wedge c) 1(Z_{-,k}+1 \ge n)1(R_{-,e,k}=0) \\
    & \quad - (R_{e,k} \wedge c) 1(\Phi^{(k)} \not= e^{(k)}) 1(Z_{-,k} =n) 1(R_{-,s,k} = 0)\\
    & \quad + (R_{e,k} \wedge c) \sum_{\ell \in \mathcal{K}\setminus\{k\}}1(\Phi^{(\ell)} =e^{(k)}) 1(Z_{-,k} = n-1)1(R_{-,s,\ell} = 0).
  \end{align*}
By \eq{bar},
\begin{align*}
  \dd{P}(R_{e,k} \le c, Z_k\ge n) & = \E[T_{e,k} \wedge (c/a_{k})]\Prob_{e,k}(Z_{-,k} \ge n-1) \\ 
  & \quad -\alpha_k (1-P_{kk}) \E_{s,k}[(R_{e,k} \wedge c) 1(Z_{-,k} =n)] \\
   & \quad + \sum_{\ell\in \mathcal{K}\setminus\{k\}} \alpha_{\ell} P_{\ell k} \E_{s,\ell}[(R_{e,k} \wedge c) 1(Z_{-,k} =n-1)],
\end{align*}
which proves \eq{Rekz}. Finally, \eq{Reklz} is similarly proved using $f(x)= (u_k \wedge c) 1(z_{\ell}\ge n)$.
\end{proof}

\subsection{Proofs of Lemmas \lemt{R-r2} and \lemt{boundary3}}
\label{sec:two-lemmas}

This is the final step for completing the proof of \pro{a-bar}.

\begin{proof}[Proof of \lem{R-r2}]
  Because $\lambda^{(r)}_{k} \eta(r\theta,{r^{1-\varepsilon_{0}}}) = O(r)$ by \eq{3.4} and \eq{eta-xi bound1} and $f^{(r)}_{r\theta,r,{r^{1-\varepsilon_{0}}}}(z)$ is uniformly bounded by \lem{basic1},  \eq{eta-Re-r2} is obtained if we show that $\dd{P}(R^{(r)}_{e,k} > {r^{\varepsilon_{0}-1}} a^{(r)}_{k}) = o(r)$.
The latter holds because
\begin{align*}
  \dd{P}(R^{(r)}_{e,k} > {r^{\varepsilon_{0}-1}} a^{(r)}_{k}) & \le {r^{1-\varepsilon_{0}}} (a^{(r)}_{k})^{-1} \dd{E}\left[R^{(r)}_{e,k} 1(R^{(r)}_{e,k} > {r^{\varepsilon_{0}-1}} a^{(r)}_{k})\right]\\
  & = \frac{{r^{1-\varepsilon_{0}}}}{2} \E[ (T^2_{e,k} - r^{2(\varepsilon_{0}-1)})1(T_{e,k}> {r^{\varepsilon_{0}-1}} )]\\
  & \le \frac{{r^{(1-\varepsilon_{0})(1+\delta_{0})}}}{2} \E[ T^{2+\delta_{0}}_{e,k}1(T_{e,k}> {r^{\varepsilon_{0}-1}}) ] =o(r), \qquad r \downarrow 0,
\end{align*}
where the first equality follows from \eq{Re2} in \lem{R}, and the second
is due to {$\E[T_{e,k}^{2+\delta_{0}}]<\infty$ and $(1-\varepsilon_{0})(1+\delta_{0}) \ge 1$ for $\varepsilon_{0} \in (0,\delta_{0}/(1+\delta_{0})$}.

Similarly, \eq{xi-Rs-r2} can be proved because
\begin{align*}
  &  \dd{P}(R^{(r)}_{s,k} > {r^{\varepsilon_{0}-1}} m^{(r)}_{k}, Z^{(r)}_{k}>0, Z^{(r)}_{H_+(k)}=0 ) \\
 & \le r^{1-\varepsilon_{0}} (m^{(r)}_{k})^{-1} \dd{E}\left[R^{(r)}_{s,k} 1(R^{(r)}_{s,k} > {r^{\varepsilon_{0}-1}} m^{(r)}_{k}, Z^{(r)}_{k}>0, Z^{(r)}_{H_+(k)}=0 ) \right] \\
&  = \frac{r^{1-\varepsilon_{0}}}{2} \gamma^{(r)}_k \dd{E}[(T_{s,k}^{2} - r^{2(\varepsilon_{0}-1)})1(T_{s,k}>{r^{\varepsilon_{0}-1}})]\\
&  \le \frac{r^{(1-\varepsilon_{0})(1+\delta_{0})}}{2} \gamma^{(r)}_k \dd{E}[T_{s,k}^{2+\delta_{0}}1(T_{s,k}>{r^{\varepsilon_{0}-1}})]=o(r), 
\end{align*}
where the equality follows from  \eq{Rs2} in \lem{R}.
\end{proof}

We finally prove \lem{boundary3}.

\begin{proof}[Proof of \lem{boundary3}]
  We first prove \eq{Pe-k-r2}. Following definition \eq{martingale_e},
  \begin{align}
\label{eq:fe+L}
 & \dd{E}_{e,k} [f(X^{(r)}_{+})]  \nonumber\\
 & \quad = \dd{E}_{e,k} \left[f(X^{(r)}_{-}) e^{r\theta_k} \dd{E}_{e,k}\Big[ e^{-\eta_k(r\theta_k,{r^{1-\varepsilon_{0}}}) \big( T_{e, k} \wedge {r^{\varepsilon_{0}-1}}\big)}\Big \mid X^{(r)}_{-} \Big]\right] = \dd{E}_{e,k}[f(X^{(r)}_{-})], \quad k\in \mathcal{L}, \\
\label{eq:fe+H}
 & \dd{E}_{e,k} [f(X^{(r)}_{+})] = \dd{E}_{e,k} \left[f(X^{(r)}_{-})  e^{r\theta_k\big((Z^{(r)}_{-,k}+1)\wedge (1/r)-Z^{(r)}_{-,k}\wedge (1/r)\big)-r\theta_k }\right], \qquad k\in \mathcal{H}.
  \end{align}
  Thus, \eq{Pe-k-r2} trivially holds for $k\in \mathcal{E}\cap\mathcal{L}$. Now,
  fix $k\in \mathcal{E}\cap \mathcal{H}$.
For each $\ell \in \mathcal{K}$, $z_\ell \in \dd{Z}_+$, and $r\in(0,1]$,
define
  \begin{align}
  \label{eq:ek+}
  &  {e^{+}}(r, z_k) = (z_{k}+1)\wedge (1/r)-z_{k}\wedge (1/r),\\
  \label{eq:ek-}
  &  {e^{-}}(r, z_k) = (z_{k}-1)\wedge (1/r)-z_{k}\wedge (1/r).
  \end{align}
 One can check that
  \begin{align*}
    & e^{+} (r, z_k)-1 =
      \begin{cases}
        0 & \text{ if } z_k \le 1/r - 1, \\
         r^{-1}-z_k-1  & \text{ if }  1/r - 1 < z_k  < 1/r, \\
        -1 & \text{ if } 1/r \le z_k.
      \end{cases}
  \end{align*}
It follows from this and similar expression for $e^{-}(r,z_{k})$ that
\begin{align}
\label{eq:ek-bound}
  |e^{+}(r, z_k)-1| \le 1(z_{k} > 1/r -1), \qquad |e^{-}(r, z_k) + 1| \le 1(z_{k} > 1/r).
\end{align}
Thus, we have
  \begin{align*}
  & \abs{r\theta_{k} e^{+}(r,  z_k)-r\theta_k}\le r\abs{\theta_k} 1(z_k > 1/r - 1).
  \end{align*}
Hence, it follows from \eq{exinq} and \eq{fe+H} that
  \begin{align*}
   & \abs{\Delta f(X^{(r)}_{+},X^{(r)}_{-})}1(R^{(r)}_{-,e,k}=0) = \abs{e^{r\theta_{k} e^{+}(r,Z^{(r)}_{-,k})-r\theta_k} -1}  f(X^{(r)}_{-}) 1(R^{(r)}_{-,e,k}=0) \\
   & \quad \le  r\abs{\theta_k} e^{r\abs{\theta_k}}1(Z^{(r)}_{-,k}+1>1/r) f(X^{(r)}_{-}) 1(R^{(r)}_{-,e,k}=0)\\
   & \quad \le  r\abs{\theta_k} 1(Z^{(r)}_{-,k}+1>1/r)  e^{|\theta_{k}|+|\theta_{H}|+ (d_{e,a} E + d_{s,a} K) a} 1(R^{(r)}_{-,e,k}=0),
  \end{align*}
  where the last inequality follows from \lem{basic1}.
  Therefore, by \lem{ssc-Palm},
  \begin{align*}
  & \abs{\dd{E}_{e,k}\big[ \Delta f(X^{(r)}_{+},X^{(r)}_{-})1(R^{(r)}_{-,e,k}=0)\big
]}\\
   & \quad \le  r\abs{\theta_k} e^{|\theta_{k}|+|\theta_{H}|+ (d_{e,a} E + d_{s,a} K) a}
      \dd{P}_{e,k}\{Z^{(r)}_{-,k} >1/r-1\} = o(r^{2}).
  \end{align*}
  
We next prove \eq{Ps-k-r2}. We first prove it for
  $k\in \mathcal{L}$. From the definition \eq{X+-} and \lem{X+}, under $\dd{P}_{s,k}$,
\begin{align*}
  & f(X^{(r)}_{+})1(R^{(r)}_{-,s,k} = 0) \\
   & \quad = f(X^{(r)}_{-}) 1(R^{(r)}_{-,s,k} = 0) \Big[ \sum_{\ell \in \sr{L} \cup \{0\}} 1(\Phi^{(k)}=e^{(\ell)}) e^{- \theta_{k} + \theta_{\ell} } e^{-\xi_k(r\theta_k, {r^{1-\varepsilon_{0}}}) \left( T_{s, k} \wedge {r^{\varepsilon_{0}-1}}\right)}\\
   & \hspace{25ex} + \sum_{\ell \in \sr{H}} 1(\Phi^{(k)}= e^{(\ell)}) e^{- \theta_{k} + \theta_{\ell} e^{+}\left(r,Z^{(r)}_{-,\ell}\right) } e^{-\xi_k(r\theta_k, {r^{1-\varepsilon_{0}}}) \left( T_{s, k} \wedge {r^{\varepsilon_{0}-1}}\right)} \Big].
\end{align*}
Hence, using the definition of $\xi_k(\theta, r)$ in \eq{xi-rt} for $k \in \sr{L}$, which is
  \begin{align*}
   \sum_{\ell\in \ol{\mathcal{K}}} P_{k\ell} e^{r\theta_\ell-r\theta_k }
  \E\Big[ e^{-\xi_k(r\theta_k,{r^{1-\varepsilon_{0}}}) \big(T_{s, k} \wedge {r^{\varepsilon_{0}-1}}\big)}\Big] =1,
  \end{align*}
we have
\begin{align*}
 & \dd{E}_{s,k}[\Delta f(X^{(r)}_{+},X^{(r)}_{-})] = \dd{E}_{s,k}\left[f(X^{(r)}_{-}) \right] \sum_{\ell \in \sr{L} \cup \{0\}} P_{k,\ell} e^{- \theta_{k} + \theta_{\ell} } \dd{E}\left(e^{-\xi_k(r\theta_k,{r^{1-\varepsilon_{0}}}) \left( T_{s, k} \wedge {r^{\varepsilon_{0}-1}}\right)}\right)\\
  & \qquad + \dd{E}_{s,k}\left[f(X^{(r)}_{-}) \sum_{\ell \in \sr{H}} P_{k,\ell} e^{- \theta_{k} + \theta_{\ell} e^{+}(r,Z^{(r)}_{-,\ell}) } \right] \dd{E}_{s,k}\left(e^{-\xi_k(r\theta_k,{r^{1-\varepsilon_{0}}}) \left( T_{s, k} \wedge {r^{\varepsilon_{0}-1}}\right)}\right)\\
 & \quad = \dd{E}_{s,k}\left[f(X^{(r)}_{-}) \sum_{\ell \in \sr{H}} P_{k,\ell} \left(e^{- \theta_{k} + \theta_{\ell} e^{+}(r,Z^{(r)}_{-,\ell}) } - 1\right) \right] \dd{E}_{s,k}\left(e^{-\xi_k(r\theta_k,{r^{1-\varepsilon_{0}}}) \left( T_{s, k} \wedge {r^{\varepsilon_{0}-1}}\right)}\right).
\end{align*}
This and \lem{ssc-Palm} prove \eq{Ps-k-r2} because both $ f^{(r)}_{r\theta, r,{r^{1-\varepsilon_{0}}}}(x)$ and $ \E\Big[ e^{-\xi_k(r\theta_k,{r^{1-\varepsilon_{0}}}) \big( T_{s, k} \wedge
    {r^{\varepsilon_{0}-1}}\big)}\Big]$ are bounded in $r$ and 
  \begin{align*}
    \abs{ e^{e^{+}(r, z_\ell) -r\theta_k } -e^{r\theta_\ell-r\theta_k} }
    = e^{r\theta_\ell-r\theta_k} \abs{ e^{e^{+} (r, z_\ell)}-1} \le
     e^{r\theta_\ell-r\theta_k} e^{r\abs{\theta_\ell}}r\abs{\theta_\ell}1(z_\ell+1>1/r).
  \end{align*}
  
  It remains to prove \eq{Ps-k-r2} for $k \in \sr{H}$, but we omit this proof because the result is obtained similarly to the case when   $k \in \sr{L}$.
\end{proof}

\appendix

\section*{Appendices}
\setcounter{section}{0}

\setnewcounter
\section{Proof of \lem{tight1}}
\label{app:tight1}

To prove (i) of \lem{tight1}, We first note that  for any subset $A$ of $\sr{L}$ and any vector $c_{A} \in \dd{R}_{+}^{\sr{L}}$ whose $\ell$th entry is $c_{\ell} 1(\ell \in A)$ with $c_{\ell} > 0$, $\phi_{A}(0-) $ and $\phi_{A,\ell}(0-)$ exist in the following sense:
\begin{align*}
\phi_{A}(0-) =  \lim_{\alpha \uparrow 0} \phi(\alpha c_{A}), \quad \phi_{A,\ell}(0-) =  \lim_{\alpha \uparrow 0} \phi_{\ell}(\alpha c_{A}), \qquad \alpha \in \R,
\end{align*}
because $\phi(\theta)$ and $\phi_{\ell}(\theta)$ are nondecreasing continuous functions from $\R^{L}_-$ to $[0, 1]$. Furthermore, these limits do not depend on $c_{i}$'s as long as they are positive, as shown in Lemma~5.1 of \cite{BravDaiMiya2017}. Thus, $\phi_{A}(0-) $ and $\phi_{A,\ell}(0-) $ are well defined and satisfy \eq{phi0}.

  To prove (ii), we  use the tight system \assum{R}  to
show that $(R,b)$ is a tight system by verifying all the conditions in \dfn{tight} for $(x_{A}, x^{(j)}_{A}) \equiv (\phi_{A}, \phi_{j,A})$. Let us make a few observations about this. Since $\phi(\theta)$ is a monotone function, it follows that $\phi_{A}(0-) \geq \phi_{A'}(0-)$ if $A \subset A' \subset \mathcal{L}$; the same holds for $\phi_{A,\ell}(0-)$ and $\phi_{A',\ell}(0-)$. Furthermore, if $\ell \in  A$, then  $\phi_{A,\ell}(0-) = \phi_{A \setminus \{\ell\},\ell}(0-)$ since $\phi_{A,\ell}(\theta)$ does not depend on $\theta_{\ell}$. It remains to verify \eq{tight-condition1}. This condition is equivalent to
\begin{align}
 \sum_{\ell \in \sr{L}} b_{\ell} R_{i,\ell} \big(\phi_{A,\ell}(0-)-\phi_{A}(0-)\big) = 0, \quad i \in A,\ A \subset \mathcal{L}. \label{eq:Aminus}
\end{align} 
We argue that this equation  can be  obtained from \eq{bar-limit 1}. To see this, we set $\theta = \alpha c_{A}$ in \eq{bar-limit 1}, divide both sides by $\alpha$, and take $\alpha \uparrow 0$; then we have
\begin{align*}
\sum_{\ell \in \sr{L}} b_{\ell} \langle c_{A}, R^{(\ell)} \rangle \big(\phi_{A,\ell}(0-)-\phi_{A}(0-)\big) = 0.
\end{align*}
For $i \in A$, letting $c_{A} \downarrow c_{i} e$ in this formula yields \eqref{eq:Aminus}. Hence, all the conditions in \dfn{tight} are satisfied.
 By \assum{R}, $(R,b)$ is a tight system. Thus,
\begin{align*}
  \phi(0-)=x_{L} = 1, \qquad \phi_{\ell}(0-)=x^{(\ell)}_{L} = 1, \qquad \ell \in \sr{L},
\end{align*}
which proves (ii) of \lem{tight1}.

\setnewcounter
\section{\thm{main} under \assum{moment-weaker}}
\label{app:moment-weaker}

If we replace \assum{moment} in \thm{main} with  \assum{moment-weaker}, then we cannot truncate the remaining times $R^{(r)}_{e,k}$ for $k \in \sr{E}$ and $R^{(r)}_{s,k}$ for $k \in \sr{K}_{1}$ in test functions in \eq{fr-2} for $X^{(r)}$ by $r^{\varepsilon_{0}-1}$ because \eq{eta-Re-r2} and \eq{xi-Rs-r2} in \lem{R-r2} require the $(2+\delta_{0})$-th moments of $T_{e,k}$ and those of $T_{s,k}$, respectively. Hence, under \assum{moment-weaker}, we need to truncate $R^{(r)}_{e,\ell}$ for $\ell \in \sr{E}$ and $R^{(r)}_{s,k}$ for $k \in \sr{K}_{1}$ in test functions by $r$. Namely, we need to choose test functions
\begin{align*}
  f^{(r)}_{\theta}(x) = g_{\theta, r}(z) e^{\Lambda^{(r)}_\theta(u,v)}
\end{align*}
to replace $f^{(r)}_{\theta, r, r^{1-\varepsilon_0}}(x)$ in \eq{fr-2}
with
\begin{align*}
  \Lambda^{(r)}_\theta(u, v) = \langle\eta(\theta, r), \lambda^{(r)}u \wedge r\rangle + \sum_{k\in \mathcal{K}_1} \xi_k(\theta, r) \big(\mu^{(r)}_k v_k \wedge r\big)
  +  \sum_{k\in \mathcal{K}\setminus  \mathcal{K}_1} \xi_k(\theta, r^{1-\varepsilon_0}) \big(\mu^{(r)}_k v_k \wedge r^{\varepsilon_0-1}\big)
\end{align*}
replacing $\Lambda^{(r)}_{\theta, r^{1-\varepsilon_0}}(u, v)$ in \eq{Lambda1}.
This causes a problem in verifying \eq{limits5}. However, if $\dd{E}[R^{(r)}_{e,\ell}] $ for $\ell \in \sr{E}$ and $\dd{E}[R^{(r)}_{s,k}]$ for $k \in \sr{K}_{1}$ are uniformly bounded in $r \in (0,1]$, then
\begin{align*}
  \limsup_{r \downarrow 0} \dd{E}(r R^{(r)}_{e,\ell} \wedge 1) \le \limsup_{r \downarrow 0} r \dd{E}(R^{(r)}_{e,\ell}) = 0, \qquad \ell \in \sr{E},
\end{align*}
and similarly $\limsup_{r \downarrow 0} \dd{E}(r R^{(r)}_{s,k} \wedge 1) = 0$ for $k \in \sr{K}_{1}$. Hence,
\begin{align}
\label{eq:psi-phi-diff1}
  &  \abs{\psi^{(r)}(r\theta) - \phi^{(r)}(r\theta)}\nonumber \\
  & \quad
  \le |\theta| e^{\abs{\theta_H}+ (d_{e,a} E + d_{s,a} K)|\theta|} \Big( d_{e,a} \sum_{\ell \in \sr{E}} \lambda^{(r)}_{\ell} \dd{E}(r R^{(r)}_{e,\ell} \wedge 1)  \nonumber\\
 & \qquad 
   + d_{s,a} \sum_{\ell \in \sr{K}_{1}} \mu^{(r)}_{\ell}\dd{E}(r R^{(r)}_{s,\ell} \wedge 1) + d_{s,a} r^{\varepsilon_{0}} \sum_{\ell \in \sr{K} \setminus \sr{K}_{1}} \mu^{(r)}_{\ell}\dd{E}(r^{1-\varepsilon_{0}} R^{(r)}_{s,\ell} \wedge 1)  \Big) =o(1),
\end{align}
where the inequality follows from an analogous version of \eq{Lambda-bound1}.
Thus, we have the first equation of \eq{limits5}, and its second equation is similarly proved. Hence, we  need to prove only  the uniform boundedness of $\dd{E}[R^{(r)}_{e,\ell}] $ for $\ell \in \sr{E}$ and $\dd{E}[R^{(r)}_{s,k}]$ for $k \in \sr{K}_{1}$. This is easily checked by \lem{R}. Namely, for $\ell \in \sr{E}$, \eq{Re2} yields
\begin{align}
\label{eq:Re-ell-1}
  2 \dd{E}[R^{(r)}_{e,\ell}] = \lambda^{(r)}_{\ell} \dd{E}_{e,\ell}[T_{e,\ell}^{2}] < \infty.
\end{align}
Similarly for $k \in \sr{K}_{1}$, \eq{Rs2} yields
\begin{align*}
  2\dd{E}[R^{(r)}_{s,k}1(Z^{(r)}_{k} > 0)] = \alpha^{(r)}_{k} \dd{E}(T_{s,k}^{2}) < \infty,
\end{align*}
because $H_{+}(k) = \emptyset$. Since $\dd{E}[R^{(r)}_{s,k}1(Z^{(r)}_{k} = 0)] = m^{(r)}_{k} \dd{E}[T_{s,k}1(Z^{(r)}_{k} = 0)]$, \eq{Rs2} yields
\begin{align}
\label{eq:Rs2-1}
  \dd{E}[R^{(r)}_{s,k}] = \dd{E}[R^{(r)}_{s,k}1(Z^{(r)}_{k} > 0)] + \dd{E}[R^{(r)}_{s,k}1(Z_{k} = 0)] \le \frac 12 \alpha^{(r)}_{k} \dd{E}(T_{s,k}^{2}) +\dd{E}[T_{s,k}]  < \infty.
\end{align}
Thus, \thm{main} can be proved under \assum{moment-weaker}. This proof also shows that  if $\dd{E}(R^{(r)}_{s,k})$ is uniformly bounded for $\sr{K} \setminus \sr{K}_{1}$ in addition to \assum{moment-weaker}, then \thm{main} can be proved under \assum{moment-conjecture}. However, we have not been able to prove that $\dd{E}(R^{(r)}_{s,k}) $ is uniformly bounded for $\sr{K} \setminus \sr{K}_{1}$ in this paper.

\setnewcounter
\section{Taylor expansions}
\label{app:P-expansion}

In this section, we prove Lemmas \lemt{eta-xi asymp1} and \lemt{eta-xi asymp2}  by  deriving Taylor expansions for $\eta_{k}(\theta,w)$ and $\xi_{k}(\theta,w)$ defined by \eq{eta-r} and \eq{xi-r}.
The latter two quantities are defined in \cite{BravDaiMiya2017} with a sign difference. The following lemma is a key to this derivation, which also follow immediately from Lemma 2.4 of \cite{Miya2017} and Lemma 4.1 of \cite{BravDaiMiya2017}.

\begin{lemma}
\label{lem:Taylor-1}
Let $T$ be a nonnegative {random variable with} $\dd{E}(T)>0$. For $t \in (0,1]$, let $T_{t} = T \wedge 1/t$. Then, there exists a unique function $f_{t}$: $\dd{R} \to \dd{R}$ for each fixed $t \in (0,1]$ such that
\begin{align}\label{eq:inv}
  \dd{E}(e^{-f_{t}(x) T_{t}}) = e^{-x}, \qquad x \in (-\infty,-\log \dd{P}(T = 0)).
\end{align}
where $- \log 0 = \infty$. Furthermore,
\begin{itemize}
\item [(i)] For each $x \in (0,-\log \dd{P}(T=0))$, $f_{0+}(x) \equiv \lim_{r \downarrow 0} f_{t}(x)$ exists and is finite.
\item [(ii)] For $x \in (-\infty,0]$ (resp. $x \in (0,-\log \dd{P}(T=0))$), $f_{t}(x)$ is decreasing (resp. increasing) in $t \in (0,1]$ as $t$ is decreasing. Hence, $\sup_{t \in (0,1]} |f_{t}(x)| \le |f_{0+}(x)|+|f_{1}(x)|$ for $x \in (0,-\log \dd{P}(T=0))$.
\item [(iii)] $f_{t}(x)$ is increasing, convex and infinitely differentiable in $x \in (0,-\log \dd{P}(T=0))$.
\item [(iv)] For any $t \in (0,1]$ and any $a \in (0,-\log \dd{P}(T = 0))$,
\begin{align}
\label{eq:Taylor1}
 & \left|f_{t}(x) - \lambda_{T_{t}} x - \frac 12 \lambda_{T_{t}}^{3} \sigma_{T_{t}}^{2} x^{2}\right| \le \frac {x^{2}}{2} \sup_{|y| \le |x|}|f_{t}''(y) - f_{t}''(0)|,\\
\label{eq:Taylor2}
  & |f_{t}(x)| \le \max\left(\lambda_{T_{1}}, (|f_{0+}(a)|+|f_{1}(a)|)/a\right) |x|, \qquad |x| < a, 
\end{align}
where $\lambda_{T_{t}} = 1/\dd{E}(T_{t})$ and $\sigma^{2}_{T_{t}}$ is the variance of $T_{t}$, and
\begin{align}
\label{eq:fw 2nd-de}
  f_{t}''(y) = \frac {e^{-y}} {\dd{E}(T_{t} e^{-f_{t}(y) T_{t}})} \left( e^{-2y} \frac {\dd{E}(T_{t}^{2} e^{-f_{t}(y) T_{t}})} {[\dd{E}(T_{t} e^{-f_{t}(y) T_{t}})]^{2}} - 1\right), \qquad y \in \dd{R}.
\end{align}
\end{itemize}
\end{lemma}

\begin{proof}[Proof of \lem{eta-xi asymp1}]
  To prove   the first inequality of \eq{eta-xi bound1},
  fix a $k \in \sr{E}$ and
  we apply \lem{Taylor-1} for $T_{e,k}$ and $x = \theta_{k}\in \R$.
  Following definitions in \eq{eta-rt} and \eq{inv}, one has
  \begin{align*} \eta_k(\theta_k,t) = f_{t}(\theta_k).
  \end{align*}
  Therefore, the inequality follows 
  immediately from \eq{Taylor2} of \lem{Taylor-1}.

  To prove   the second inequality of \eq{eta-xi bound1},
  fix a $k \in \sr{K}$ and
  we apply \lem{Taylor-1} for $T_{s,k}$ and $x = x_{s}(\theta)$, where
  $\theta\in \R^K$ and 
\begin{align*}
  x_{s}(\theta) = \log \Big(e^{-\theta_{k}} \sum_{\ell \in \ol{\sr{K}}} P_{k,\ell} e^{\theta_{\ell}}\Big).  
\end{align*}
  Following definitions in \eq{xi-rt} and \eq{inv}, one has
  \begin{align*} \xi_{k}(\theta,t) = f_{t}(x_s(\theta)).
  \end{align*}
The second inequality follows similarly
from \eq{Taylor2} of \lem{Taylor-1}
because 
\begin{align*}
  |x_{s}(\theta)| \le |\theta_{k}| + \log \Big(\sum_{\ell \in \ol{\sr{K}}} P_{k,\ell} e^{|\theta_{\ell}|}\Big) \le |\theta| (1+e^{a}), \qquad |\theta| \le a,
\end{align*}
by
\begin{align*}
  \log \Big(\sum_{\ell \in \ol{\sr{K}}} P_{k,\ell} e^{|\theta_{\ell}|}\Big) \le \sum_{\ell \in \ol{\sr{K}}} P_{k,\ell} (e^{|\theta_{\ell}|} - 1) \le \sum_{\ell \in \ol{\sr{K}}} P_{k,\ell} |\theta_{\ell}| e^{|\theta_{\ell}|} \le |\theta| e^{|\theta|},
\end{align*}
where the first inequality from $\log y \le y -1$ for $y > 0$, and the second inequality from $e^{y} - 1 \le y e^{y}$ for $y \ge 0$.
\end{proof}

{To prove \lem{eta-xi asymp2}, we take $\eta_{k}(x,t)$ or $\xi_{k}(x,t)$ for $f_{t}(x)$, and put $x = r \theta$ or $x = x_{s}(r\theta)$, respectively, and $t = r^{1-\varepsilon_{0}}$, then let $r \downarrow 0$. Hence, it is sufficient to consider the convergence of $f_{r^{1-\varepsilon_{0}}}(x)$ as $r \downarrow 0$ for $0 < |x| \le a\, r$ for each positive constant $a$.} By \eq{Taylor2}, {
\begin{align*}
  |f_{r^{1-\varepsilon_{0}}}(x)(T \wedge r^{\varepsilon_{0}-1})| \le a (r T \wedge r^{\varepsilon_{0}}) \max\left(\lambda_{T_{1}}, (|f_{0+}(a)|+|f_{1}(a)|)/a\right), \qquad \abs{x}\le a. 
\end{align*}
Hence, for $0 < |x| \le a\, r$ and $t = r^{1-\varepsilon_{0}}$, $|f_{t}(x)T_{t}|$ is bounded by $a \times \max\left(\lambda_{T_{1}}, (|f_{0+}(a)|+|f_{1}(a)|)/a\right)$, which is independent of $r$, and vanishes as $|x| \downarrow 0$, and therefore $\dd{E}(T_{t} e^{-f_{t}(y) T_{t}})$ and $\dd{E}(T_{t}^{2} e^{-f_{t}(y) T_{t}})$ converge to $\dd{E}(T_{t})$ and $\dd{E}(T_{t}^{2})$, respectively, as $|y| \downarrow 0$ uniformly in $r$ for $|y| \le |x| \le a\, r$. Hence, by \eq{fw 2nd-de}, $\sup_{|y| \le |x|}|f_{t}''(y) - f_{t}''(0)|$ in \eq{Taylor1} vanishes uniformly in $r \in (0,1]$ as $|x| \downarrow 0$ if $|x| \le a\, r$. Thus}, we have the following corollary of \lem{Taylor-1}.

\begin{corollary}
\label{cor:Taylor-1}
Under the same assumptions as \lem{Taylor-1}, for $t = r^{1-\varepsilon_{0}}$ and any constant ${a} >0$, as $|x| \downarrow 0$,
\begin{align}
\label{eq:Taylor3}
  \sup_{r \in ({a}^{-1}|x|,1]} \left|f_{t}(x) - \lambda_{T_{t}} x - \frac 12 \lambda_{T_{t}}^{3} \sigma_{T_{t}}^{2} x^{2}\right| = o(|x|^{2}). 
\end{align}
\end{corollary}

\begin{proof}[Proof of \lem{eta-xi asymp2}]
We first prove \eq{eta asymp1}. Let $f_{r^{1-\varepsilon_{0}}}(r\theta_{k}) = \eta_{k}(r\theta_{k},r^{1-\varepsilon_{0}})$. Then it follows from \eq{Taylor3} with $T = T_{e,k}$ and $x = r\theta_{k}$ that
\begin{align*}
  \eta_{k}(r\theta_{k},r^{1-\varepsilon_{0}})\ = r \theta_{k} \frac 1{\dd{E}(T_{e,k} \wedge r^{\varepsilon_{0}-1})} + \frac 12 r^{2}\theta_{k}^{2} \frac {\Var(T_{e,k} \wedge r^{\varepsilon_{0}-1})} {(\dd{E}(T_{e,k} \wedge r^{\varepsilon_{0}-1}))^{3}} + o(r^{2}), \qquad k \in \sr{E}.
\end{align*}
This implies \eq{eta asymp1} if we can show that
\begin{align}
\label{eq:eta-expansion1}
  \frac 1{\dd{E}(T_{e,k} \wedge r^{\varepsilon_{0}-1})} - 1 = o(r), \qquad \frac {\Var(T_{e,k} \wedge r^{\varepsilon_{0}-1})} {(\dd{E}(T_{e,k} \wedge r^{\varepsilon_{0}-1}))^{3}} - c_{e,k}^{2} = o(1).
\end{align}
We now choose $\varepsilon_{0}> 0$ such that $\varepsilon_{0} \le \delta_{0}/(1+\delta_{0})$, then $(1-\varepsilon_{0})(1+\delta_{0}) \ge1$, and therefore the first formula is obtained from the fact that, as $r \downarrow 0$,
\begin{align}
\label{eq:w-reason}
  {0 \le} 1 - \dd{E}(T_{e,k} \wedge r^{\varepsilon_{0}-1}) {\le} \dd{E}(T_{e,k} 1(T_{e,k} > r^{\varepsilon_{0}-1})) \le r \dd{E}(T_{e,k}^{2+\delta_{0}} 1(T_{e,k} > r^{\varepsilon_{0}-1})) = o(r),
\end{align}
because $\dd{E}(T_{e,k})=1$ and $\dd{E}(T_{e,k}^{2}) < \infty$. {Since $\Var(T_{e,k}) = c_{e,k}^{2}$, the second formula obviously holds.} Thus, we have proved \eq{eta asymp1}.

Finally, we prove \eq{xi asymp1}. Similarly to \eq{eta asymp1}, we apply \eq{Taylor3} of \cor{Taylor-1} for $T = T_{s,k}$ and $x = x_{s}(r\theta)$, then
\begin{align*}
  \xi_{k}(r\theta,r^{1-\varepsilon_{0}}) =  \frac {x_{s}(r\theta)} {\dd{E}(T_{s,k} \wedge r^{\varepsilon_{0}-1})} + \frac 12 x_{s}^{2}(r\theta) \frac {\Var(T_{s,k} \wedge r^{\varepsilon_{0}-1})} {(\dd{E}(T_{s,k} \wedge r^{\varepsilon_{0}-1}))^{3}} + o(x_{s}^{2}(r\theta)), \qquad k \in \sr{E}.
\end{align*}
Hence, using Taylor expansion of $x_{s}(r\theta)$ concerning $r$ around the origin,
\begin{align}
\label{eq:x-s Taylor-1}
  x_{s}(r\theta) = r(- \theta_{k} + \sum_{\ell \in K} P_{k,\ell} \theta_{\ell}) +\frac{1}{2} r^{2}\Big(\sum_{\ell \in K} P_{k,\ell} \theta_{\ell}^2 - \Big(\sum_{\ell \in K} P_{k,\ell} \theta_{\ell}\Big)^{2} \Big) + o(r^{2})
\end{align}
and similar asymptotic behaviors to \eq{eta-expansion1} for $T_{s,k}$, we have \eq{xi asymp1}.
\end{proof}

\bibliography{/Users/jimdai/Dropbox/dai/texmf/bib/dai20230220} 

\end{document}